\documentclass[a4paper,12pt,reqno]{amsart}
\usepackage{amssymb,amsmath,amsfonts,amsthm,accents}
\usepackage[margin=1in]{geometry}

\newcommand{\ds}{\displaystyle}

\newcommand{\NN}{\mathbb N}

\newcommand{\CC}{\mathbb C}
\newcommand{\RR}{\mathbb R}
\newcommand{\ZZ}{\mathbb Z}
\newcommand{\EE}{\mathcal E}
\newcommand{\DD}{\mathcal D}
\newcommand{\SSS}{\mathcal S}

\newcommand{\ssum}{\mbox{$\sum_j\,$}}

\newtheorem{theorem}{Theorem}[section]
\newtheorem{proposition}[theorem]{Proposition}
\newtheorem{lemma}[theorem]{Lemma}
\newtheorem{corollary}[theorem]{Corollary}

\theoremstyle{remark}
\newtheorem{remark}[theorem]{Remark}

\theoremstyle{definition}
\newtheorem{definition}[theorem]{Definition}
\newtheorem{example}[theorem]{Example}
\allowdisplaybreaks
\numberwithin{equation}{section}

\newcommand{\beq}{\begin{eqnarray}}
\newcommand{\eeq}{\end{eqnarray}}

\newcommand{\beqs}{\begin{eqnarray*}}
\newcommand{\eeqs}{\end{eqnarray*}}

\newcommand{\Op}{\mathrm{Op}}

\begin{document}

\author[S. Pilipovi\' c]{Stevan Pilipovi\' c}
\thanks{The work of S. Pilipovi\'c was supported by the Ministry of Sciences of Serbia, project 174024, and by the Serbian Academy of Sciences and Arts F10}
\address{Department of Mathematics and Informatics,
University of Novi Sad, Trg Dositeja Obradovi\'{c}a 4, 21000 Novi Sad, Serbia}
\email{stevan.pilipovic@dmi.uns.ac.rs}

\author[B. Prangoski]{Bojan Prangoski}
\thanks{The work of B. Prangoski was partially supported by the bilateral project ``Microlocal analysis and applications'' funded by the Macedonian and Serbian academies of sciences and arts.}
\address{Department of Mathematics, Faculty of Mechanical
Engineering, Ss. Cyril and Methodius University in Skopje, Karposh 2 b.b., 1000 Skopje, Macedonia}
\email{bprangoski@yahoo.com}

\author[J. Vindas]{Jasson Vindas}
\thanks{The work of J. Vindas was supported by Ghent University, through the BOF-grants 01J11615 and 01J04017.}
\address{Department of Mathematics: Analysis, Logic and Discrete Mathematics, Ghent University, Krijgslaan 281,
9000 Ghent, Belgium}
\email{jasson.vindas@UGent.be}

\title[Infinite order $\Psi$DOs]{Infinite order $\Psi$DOs: Composition with entire functions, new Shubin-Sobolev spaces, and index theorem}



\frenchspacing

\begin{abstract}
	We study global regularity and spectral properties of power series of the Weyl quantisation $a^w$, where  $a(x,\xi) $ is a classical elliptic Shubin polynomial. For a suitable entire function $P$, we associate two natural infinite order operators to $a^{w}$,  $P(a^w)$ and  $(P\circ a)^{w},$ and prove that these operators and their lower order perturbations are globally Gelfand-Shilov regular. They have spectra consisting of real isolated eigenvalues diverging to $\infty$ for which we find the asymptotic behaviour of their eigenvalue counting function. In the second part of the article, we introduce Shubin-Sobolev type spaces by means of $f$-$\Gamma^{*,\infty}_{A_p,\rho}$-elliptic symbols, where $f $ is a function of ultrapolynomial growth and
	$\Gamma^{*,\infty}_{A_p,\rho}$ is a class of symbols of infinite order studied in this and our previous  papers.
	We study the regularity properties of these spaces, and show that the pseudo-differential operators under consideration are Fredholm operators on them. Their indices are independent on the order of the Shubin-Sobolev spaces; finally, we show that the index can be expressed via a Fedosov-H\" ormander integral formula.
\end{abstract}
\maketitle

\section{Introduction}

The Wigner-Weyl quantisation of quantum mechanics relates functions on the time frequency space and the pseudo-differential operators on various Sobolev spaces. The spectral analysis with the Weyl formula involved in this theory and the whole framework of quantisation are enormously extended in an extensive literature, part of which is given in the bibliography. In this article we continue with the extension of the framework and obtain a number of new results which may be relevant for the general concept of quantisation.

The aim of this article is twofold. On the one hand, we study the global regularity and spectral properties of operators defined via Weyl quantisation of power series of classical elliptic Shubin polynomials \cite{NR,Shubin}. Operators having such a form can be thought of as pseudo-differential operators of infinite order. On the other hand, our second goal is to introduce a new scale of infinite order Shubin-Sobolev type spaces describing global regularity properties of solutions of elliptic pseudo-differential equations in the infinite order context.

The determination of spectral properties for infinite order pseudo-differential operators has attracted some recent attention and there are several contributions in the literature devoted to particular instances \cite{KM,LST,TF}, usually arising from mathematical physics problems. We restrict our attention here to compositions of Shubin polynomials with entire functions of sub-exponential growth as well as their lower order perturbations. Our goal is then to prove global hypoellipticity and obtain Weyl asymptotic formulae for the operators associated to a large class of such entire functions. We also mention that in \cite{HelfferRobert} the composition of a function and a pseudo-differential operator is treated in a completely different way in comparison with the approach we shall employ in this article. In fact, the symbolic calculus developed in \cite{CPP,CPP1,PP,PP1,PPV-BMS,BojanP} (see also \cite{c1,c2}) provides a tool for studying the power series of Shubin operators under consideration.

The $\Psi$DO-calculus tools and the corresponding infinite order symbol classes, denoted as $\Gamma^{*,\infty}_{A_p,\rho}$, needed in this work are explained in the preliminary Section \ref{section preli}. The natural functional analytic framework for these symbol classes is the (generalised) Gelfand-Shilov spaces and their duals, the spaces of tempered ultradistributions. The symbols from the $\Gamma^{*,\infty}_{A_p,\rho}$-classes are allowed to grow sub-exponentially (i.e., ultrapolynomially), going beyond the classical Weyl-H\"ormander calculus. This gives us a way of fitting power series of elliptic Shubin differential operators within the $\Psi$DO-calculus for the $\Gamma^{*,\infty}_{A_p,\rho}$-classes, which, in turn, provides an effective device for answering questions about their global regularity properties. Indeed, we will prove that such operators and their ``lower order'' perturbations are hypoelliptic in $\Gamma^{*,\infty}_{A_p,\rho}$ sense under certain condition on the coefficients in the power series, which yields global regularity in Gelfand-Shilov sense. This is of independent interest since recently establishing Gelfand-Shilov, and Gevrey regularity in general, arises as an important problem in fluid dynamics (in studying the Gelfand-Shilov and Gevrey regularity of solutions of the Boltzmann equation); see for example \cite{BHRV,DFT,LMPSX,ZY} and the references therein.

As we have recently shown in \cite{PPV-BMS} and we explain in Subsection \ref{spec1}, $\Gamma^{*,\infty}_{A_p,\rho}$-hypoellipticity, plus suitable elliptic type bounds on the symbols, plays an important role for the validity of Weyl asymptotic formulae in the context of infinity order $\Psi$DOs. In order to motivate our results from Section \ref{powerseriesofop}, let $H=|x|^2-\Delta$ be the harmonic oscillator and consider the infinite order operator $$A=\sum_{n=0}^{\infty}\frac{h^nH^{n}}{n^{2sn}}+\mbox{``lower order pertubations''},$$ where $h>0$ and $s>1$ are suitable constants. Then our results immediately imply that $A$ is $\Gamma^{*,\infty}_{A_p,\rho}$-hypoelliptic and hence globally Gelfand-Shilov regular. Furthermore, writing $N(\lambda)$ for the spectral counting function of $A$, it will follow from our analysis that
$$
N(\lambda) \sim \frac{ e^{2ds}}{h^{d}s^{2ds} 2^{d(2s+1)} d!}\:(\ln \lambda)^{2ds}, \quad \lambda\to\infty.
$$
\indent The second part of this article is devoted to the introduction of infinite order Shubin-Sobolev spaces as an effective apparatus for fine-scale measurement of the Gelfand-Shilov regularity of solutions of elliptic equations. The results on the spectral properties (and spectral asymptotics) of power series of Shubin differential operators that we obtain in Section \ref{powerseriesofop} will play a key role in establishing the ``degree'' of Gelfand-Shilov regularity of the elements of the Shubin-Sobolev spaces we introduce in Section \ref{Elliptic Shubin} (see Proposition \ref{ktr991101} and Corollary \ref{corollaryontheregofsob}). The elliptic operators in our class are Fredholm when acting on appropriate Shubin-Sobolev spaces and, employing a classical argument, one derives that their indices are independent of the order of the Shubin-Sobolev spaces. We also prove that this remains true for matrix valued $\Psi$DOs of infinite order.

Furthermore, we prove in Section \ref{integral formula index} that the index of such operator with symbol $\mathbf{a}$ is given (up to a constant) by the integral of the form $\mathrm{tr}\,(\mathbf{a}^{-1}d\mathbf{a})^{2d-1}$ over the boundary of a large enough ball outside of which $\mathbf{a}$ is invertible. This expresses the index of the operator by only using its symbol and coincides with the Fedosov-H\"ormander integral formula for finite order operators (see \cite{fedosov1,fedosov2,hor-1,pil-pra1}).\footnote{In fact, Fedosov \cite[p. 312]{fedosov3} points out that a formula of this type was previously presented by Dynin on a conference but was never published.} Of course, this also agrees with the Atiyah-Singer index theorem \cite{as3} (see \cite{fedosov3} for its ``translation'' into the language of differential forms).

\section{Preliminaries}\label{section preli}

\subsection{Notation}\label{not}

As standard, $\langle x\rangle $ stands for $(1+|x|^2)^{1/2}$, $x\in\RR^d$. When $\alpha\in\NN^d$ we use the notation $D^{\alpha}=D^{\alpha_1}_1\ldots D^{\alpha_d}_d$, where $D^{\alpha_j}_j=i^{-|\alpha_j|}\partial^{\alpha_j}/\partial x_j^{\alpha_j}$. Following \cite{Komatsu1} (see also \cite{Petzsche88}), let $M_p$, $p\in\NN$, be a weight sequence of positive numbers which satisfies some of the conditions $(M.1)$, $(M.2)$, $(M.3)$, $(M.3)'$; we always assume $M_0=1$. Often we will impose the following additional condition:
 $(M.4)$ $M_p^2/p!^2\leq (M_{p-1}/(p-1)!)\cdot (M_{p+1}/(p+1)!)$, $p\in\ZZ_+$.
The best example of a sequence satisfying all of the above conditions is  $p!^s$, $p\in\NN$, with $s>1$. If $M_p$ and $\tilde{M}_p$ are two weight sequences, then the notation $M_p\subset \tilde{M}_p$ means that there are $C,L>0$ such that $M_p\leq CL^p\tilde{M}_p$, $\forall p\in \NN$, while $M_p\prec\tilde{M}_p$ means that this inequality holds true for each $L>0$ and a corresponding $C=C_L>0$. We will use the notation $M_{\alpha}$ for $M_{|\alpha|}$, where $\alpha\in\NN^d$, $|\alpha|=\alpha_1+...+\alpha_d$. Next (see \cite[Section 3]{Komatsu1}), $m_p=M_p/M_{p-1}$, $p\in\ZZ_+$, and the associated function to $M_p$ is $M(\rho)=\sup_{p\in\NN}\ln_+ \rho^{p}/M_{p}$, $\rho > 0$. When $M_p=p!^{s}$, with $s>0$,  we have $M(\rho)\asymp \rho^{1/s}$.\\
\indent Let $K$ be a regular compact subset of an open set $U\subseteq\RR^d$ and $h>0$. Then $\EE^{\{M_p\},h}(K)$ is the
Banach space ($(B)$-space) of all
$\varphi\in C^{\infty}(\mathrm{int}\, K)$ whose derivatives extend to continuous functions on $K$ and satisfy
$\sup_{\alpha\in\NN^d}\sup_{x\in
K}|D^{\alpha}\varphi(x)|/(h^{\alpha}M_{\alpha})<\infty$; $\DD^{\{M_p\},h}_K$ denotes its subspace consisting of the functions supported by $K$. We work with the locally convex spaces (l.c.s.)
$
\EE^{(M_p)}(U),$
 $\EE^{\{M_p\}}(U),$
 $\DD^{(M_p)}(U),$
 $\DD^{\{M_p\}}(U)$ and their strong duals, the corresponding spaces of ultradistributions  of Beurling and Roumieu type, cf. \cite{Komatsu1,Komatsu2,Komatsu3}.\\
\indent We denote by $\mathfrak{R}$ the set of all positive sequences which monotonically increase to infinity. It is a directed set under the partial order defined by $(r_p)\leq (k_p)$ if $r_p\leq k_p$, $\forall p\in\ZZ_+$.
 Let $(r_p)\in\mathfrak{R}$,  $N_0=1$, $N_p=M_p\prod_{j=1}^{p}r_j$, $p\in\ZZ_+$. Then, $N_p$  satisfies $(M.1)$ and $(M.3)'$ when $M_p$ does so and its associated function will be denoted by $N_{r_p}(\rho)$, i.e. $N_{r_{p}}(\rho)=\sup_{p\in\NN} \ln_+ \rho^{p}/(M_p\prod_{j=1}^{p}r_j)$, $\rho > 0$.\footnote{Here and throughout the rest of the article we use the principle of vacuous (empty) product, i.e. $\prod_{j=1}^0 r_j=1$.} Note that for $(r_{p})\in\mathfrak{R}$ and $k > 0 $ there is $\rho _{0} > 0$ such that $N_{r_{p}} (\rho ) \leq M(k \rho )$, for $\rho > \rho _{0}$.\\
\indent A measurable function $f$ on $\RR^d$ is said to have
ultrapolynomial growth of class $(M_p)$ (resp. of class $\{M_p\}$)
if $\|e^{-M(h|\cdot|)}f\|_{L^{\infty}(\RR^d)}<\infty$ for some
$h>0$ (resp. for every $h>0$). It is useful  to have in mind the following fact (cf. \cite[Lemma 4.5]{DVV} and \cite[Lemma 2.1]{PP1}):
 Let $B\subseteq C(\RR^d)$. Then,
 $$(\forall h>0)(\exists C>0)(|f(x)|\leq Ce^{M(h|x|)},x\in \mathbb R^d,f\in B )\Leftrightarrow$$
 $$(\exists(r_p)\in\mathfrak{R})(\exists C>0)(|f(x)|\leq Ce^{N_{r_p}(|x|)}, x\in\RR^d, f\in B).$$
 \indent An entire function $P(z) =\sum _{\alpha \in \NN^d}c_{\alpha } z^{\alpha}$, $z \in \CC^d$, is called an ultrapolynomial of class $(M_{p})$ (resp. of class $\{M_{p}\}$), if $c_{\alpha }$ satisfy: $|c_{\alpha }|  \leq C L^{|\alpha| }/M_{\alpha}$, $\alpha \in \NN^d$, for some $L > 0$ and $C>0$ (resp. for every $L > 0 $ and some $C=C(L) > 0$). The corresponding operator $P(D)=\sum_{\alpha} c_{\alpha}D^{\alpha}$ is called an ultradifferential operator of  class $(M_{p})$ (resp. of class $\{M_{p}\}$) and, when $M_p$ satisfies $(M.2)$, it acts continuously on $\EE^{(M_p)}(U)$ and $\DD^{(M_p)}(U)$ (resp. on $\EE^{\{M_p\}}(U)$ and $\DD^{\{M_p\}}(U)$) and on the corresponding duals.\\
\indent If $M_p$ satisfies $(M.1)$ and $(M.3)'$, for $m>0$, we
denote by $\SSS^{M_p,m}_{\infty}(\RR^d)$ the $(B)$-space of all
$\varphi\in C^{\infty}(\RR^d)$ for which the norm
 $\sup_{\alpha\in
\NN^d}m^{|\alpha|}\|e^{M(m|\cdot|)}D^{\alpha}\varphi\|_{L^{\infty}(\RR^d)}/M_{\alpha}$
is finite. The spaces of sub-exponentially decreasing
ultradifferentiable function of Beurling and Roumieu type are
defined as
 $$
\SSS^{(M_{p})}(\RR^d)=\lim_{\substack{\longleftarrow\\
m\rightarrow\infty}}\SSS^{M_{p},m}_{\infty}\left(\RR^d\right)\quad
\mbox{and}\quad
\SSS^{\{M_{p}\}}(\RR^d)=\lim_{\substack{\longrightarrow\\
m\rightarrow 0}}\SSS^{M_{p},m}_{\infty}\left(\RR^d\right).
$$
Their strong duals $\SSS'^{(M_{p})}(\RR^d)$ and $\SSS'^{\{M_{p}\}}(\RR^d)$ are the spaces of tempered ultradistributions of Beurling and Roumieu type, respectively. When $M_p=p!^{s}$, $s>1$, the Roumieu space is the well-known Gelfand-Shilov space $\SSS^{s}_{s}(\RR^d)$ \cite{NR}. If $M_p$ additionally satisfies $(M.2)$, then the ultradifferential operators of class $*$ (we use $*$ as a common notation for the Beurling and Roumieu case; cf. \cite{Komatsu1}) act continuously on $\SSS^*(\RR^d)$ and $\SSS'^*(\RR^d)$; these spaces are nuclear and the Fourier transform is a topological isomorphism on them (\cite{PilipovicK,PilipovicU,PPV-JMPA}). Moreover, the space $\SSS^{\{M_p\}}(\RR^d)$ is topologically isomorphic to $\ds\lim_{\substack{\longleftarrow\\ (r_p)\in\mathfrak{R}}}\SSS^{M_p, (r_p)}_{\infty}(\RR^d)$, where the projective limit is taken with respect to the natural order on $\mathfrak{R}$ defined above and $\SSS^{M_p, (r_p)}_{\infty}(\RR^d)$ is the $(B)$-space of all $\varphi\in C^{\infty}(\RR^d)$ for which the norm\\
$\sup_{\alpha\in \NN^d}\|e^{N_{r_p}(|\cdot|)}D^{\alpha}\varphi\|_{L^{\infty}(\RR^d)}/ (M_{\alpha}\prod_{j=1}^{|\alpha|}r_j)$ is finite (see \cite{PilipovicK}).

\subsection{Symbol classes and symbolic calculus}\label{sub calculus}
Let $A_p$ and $M_p$ be two weight sequences of positive numbers such that $A_0=A_1=M_0=M_1=1$. We assume that $M_p$ satisfies $(M.1)$, $(M.2)$ and $(M.3)$,  that $A_p$ satisfies $(M.1)$, $(M.2)$, $(M.3)'$, $(M.4)$ and that $A_p\subset M_p$. Without losing generality, we can assume the constants $c_0$ and $H$ appearing in $(M.2)$ are the same for both sequences $M_p$ and $A_p$. Let $\rho_0=\inf\{\rho\in\RR_+|\,A_p\subset M_p^{\rho}\}$; clearly $0<\rho_0\leq 1$. In the sequel $\rho$ is a fixed number satisfying $\rho_0\leq \rho\leq1$, if the infimum is reached, or, otherwise $\rho_0< \rho\leq1$.\\
\indent Let $h,m>0$. As in \cite{BojanP}, we denote by $\Gamma_{A_p,\rho}^{M_p,\infty}(\RR^{2d};h,m)$ the $(B)$-space of all $a\in C^{\infty}(\RR^{2d})$ for which the norm
\beqs
\sup_{\alpha,\beta\in\NN^d}\sup_{(x,\xi)\in\RR^{2d}}\frac{\left|D^{\alpha}_{\xi}D^{\beta}_x
a(x,\xi)\right| \langle
(x,\xi)\rangle^{\rho|\alpha|+\rho|\beta|}e^{-M(m|\xi|)}e^{-M(m|x|)}}
{h^{|\alpha|+|\beta|}A_{\alpha}A_{\beta}}
\eeqs
is finite. As l.c.s., we define (see \cite{BojanP})
$$
\Gamma_{A_p,\rho}^{(M_p),\infty}(\RR^{2d};m)=
\lim_{\substack{\longleftarrow\\h\rightarrow 0}}
\Gamma_{A_p,\rho}^{M_p,\infty}(\RR^{2d};h,m);
\quad
\Gamma_{A_p,\rho}^{(M_p),\infty}(\RR^{2d})=
\lim_{\substack{\longrightarrow\\m\rightarrow\infty}}
\Gamma_{A_p,\rho}^{(M_p),\infty}(\RR^{2d};m);
$$
$$
\Gamma_{A_p,\rho}^{\{M_p\},\infty}(\RR^{2d};h)=
\lim_{\substack{\longleftarrow\\m\rightarrow 0}}
\Gamma_{A_p,\rho}^{M_p,\infty}(\RR^{2d};h,m);
\quad
\Gamma_{A_p,\rho}^{\{M_p\},\infty}(\RR^{2d})=
\lim_{\substack{\longrightarrow\\h\rightarrow\infty}}
\Gamma_{A_p,\rho}^{\{M_p\},\infty}(\RR^{2d};h).
$$
Then,
$\Gamma_{A_p,\rho}^{(M_p),\infty}(\RR^{2d};m)$ and
$\Gamma_{A_p,\rho}^{\{M_p\},\infty}(\RR^{2d};h)$ are $(F)$-spaces.
The spaces $\Gamma_{A_p,\rho}^{*,\infty}(\RR^{2d})$ are barrelled
and bornological.
 For $\tau\in\RR$ and
$a\in\Gamma_{A_p,\rho}^{*,\infty}(\RR^{2d})$, the
$\tau$-quantisation of $a$ is the operator $\Op_{\tau}(a)$, continuous on
 $\SSS^*(\RR^d)$,
given by the  iterated integral:
\beqs
\left(\Op_{\tau}(a)u\right)(x)=\frac{1}{(2\pi)^d}\int_{\RR^d}\int_{\RR^d}e^{i(x-y)\xi}a((1-\tau)x+\tau
y,\xi) u(y)dyd\xi.
\eeqs
Of course, $\Op_{\tau}(a)$ extends to a continuous operator on $\SSS'^*(\RR^d)$ (\cite[Proposition 3.1]{PP1}).\\
\indent Let $t\geq0$. We denote $Q_t=\left\{(x,\xi)\in\RR^{2d}|\,\langle
x\rangle<t, \langle \xi\rangle<t\right\}$ and
$Q_t^c=\RR^{2d}\backslash Q_t$. If $0\leq t\leq 1$, then
$Q_t=\emptyset$ and $Q_t^c=\RR^{2d}$. Let $B\geq 0$ and $h,m>0$.
Denote by $FS_{A_p,\rho}^{M_p,\infty}(\RR^{2d};B,h,m)$ the vector
space of all formal series $\sum_{j=0}^{\infty}a_j(x,\xi)$ such
that $a_j\in  C^{\infty}(\mathrm{int\,}Q^c_{Bm_j})$,
$D^{\alpha}_{\xi} D^{\beta}_x a_j(x,\xi)$ can be extended to a
continuous function on $Q^c_{Bm_j}$ for all $\alpha,\beta\in\NN^d$
and
\beqs
\sup_{j\in\NN}\sup_{\alpha,\beta\in\NN^d}\sup_{(x,\xi)\in
Q_{Bm_j}^c}\frac{\left|D^{\alpha}_{\xi}D^{\beta}_x
a_j(x,\xi)\right| \langle
(x,\xi)\rangle^{\rho|\alpha|+\rho|\beta|+2j\rho}e^{-M(m|\xi|)}e^{-M(m|x|)}}
{h^{|\alpha|+|\beta|+2j}A_{\alpha}A_{\beta}A_jA_j}<\infty.
\eeqs
We use the convention $m_0=0$ and hence, $Q^c_{Bm_0}=\RR^{2d}$. With this norm,
$FS_{A_p,\rho}^{M_p,\infty}\left(\RR^{2d};B,h,m\right)$ becomes a
$(B)$-space. As l.c.s., we define (see \cite{BojanP,PP1}),
$$
FS_{A_p,\rho}^{(M_p),\infty}(\RR^{2d};B)=
\lim_{\substack{\longrightarrow\\m\rightarrow\infty}}
\lim_{\substack{\longleftarrow\\h\rightarrow
0}}
FS_{A_p,\rho}^{M_p,\infty}(\RR^{2d};B,h,m),
$$
$$
FS_{A_p,\rho}^{\{M_p\},\infty}(\RR^{2d};B)=\lim_{\substack{\longrightarrow\\h\rightarrow
\infty}} \lim_{\substack{\longleftarrow\\m\rightarrow 0}}
FS_{A_p,\rho}^{M_p,\infty}(\RR^{2d};B,h,m).
$$
Then, $FS_{A_p,\rho}^{(M_p),\infty}(\RR^{2d};B)$ and $FS_{A_p,\rho}^{\{M_p\},\infty}(\RR^{2d};B)$ are $(LF)$-spaces; both are  barrelled and bornological. The inclusion mapping $\Gamma_{A_p,\rho}^{*,\infty}(\RR^{2d})\rightarrow$ $ FS_{A_p,\rho}^{*,\infty}(\RR^{2d};B)$,
defined as $a\mapsto\sum_{j\in\NN}a_j$, where $a_0=a$ and $a_j=0$, $j\geq 1$, is continuous. We call this inclusion the canonical one. When $B_1\leq B_2$, the mapping $\sum_j a_j\mapsto \sum_j a_{j|_{Q_{B_2m_j}^c}}$, $FS_{A_p,\rho}^{*,\infty}(\RR^{2d};B_1)\rightarrow FS_{A_p,\rho}^{*,\infty}(\RR^{2d};B_2)$, is continuous. We also denote $\ds FS_{A_p,\rho}^{*,\infty}(\RR^{2d})=\lim_{\substack{\longrightarrow \\ B\rightarrow \infty}} FS_{A_p,\rho}^{*,\infty}(\RR^{2d};B)$ where the inductive limit is taken in an algebraic sense; clearly $FS_{A_p,\rho}^{*,\infty}(\RR^{2d})$ is nontrivial.\\
\indent If $\sum_j a_j\in FS_{A_p,\rho}^{*,\infty}(\RR^{2d};B)$
and $n\in\NN$, then $(\sum_j a_j)_n=a_n\in C^{\infty}(Q_{Bm_n}^c)$, while
$(\sum_j a_j)_{<n}=\sum_{j=0}^{n-1}
a_j\in C^{\infty}(Q_{Bm_{n-1}}^c)$;
$\mathbf{1}=\sum_j a_j\in
FS_{A_p,\rho}^{*,\infty}(\RR^{2d};B)$, where $a_0(x,\xi)=1$ and
$a_j(x,\xi)=0$, $\forall j\in\ZZ_+$.

\begin{definition}[{\cite[Definition 3]{BojanP}}]
Two sums $\sum_{j\in\NN}a_j,\,\sum_{j\in\NN}b_j\in
FS_{A_p,\rho}^{*,\infty}(\RR^{2d})$ are said to be equivalent, in
notation $\sum_{j\in\NN}a_j\sim\sum_{j\in\NN}b_j$, if there exist
$m>0$ and $B>0$ (resp. there exist $h>0$ and $B>0$), such that for
every $h>0$ (resp. for every $m>0$),
\beqs
\sup_{n\in\ZZ_+}\sup_{\alpha,\beta}\sup_{(x,\xi)\in
Q_{Bm_n}^c}\frac{\left|D^{\alpha}_{\xi}D^{\beta}_x
\sum_{j<n}\left(a_j(x,\xi)-b_j(x,\xi)\right)\right| \langle
(x,\xi)\rangle^{\rho|\alpha|+\rho|\beta|+2n\rho}}
{h^{|\alpha|+|\beta|+2n}A_{\alpha}A_{\beta}A_nA_ne^{M(m|\xi|)}e^{M(m|x|)}}<\infty.
\eeqs
\end{definition}

In what follows, we will often use the shorthand $w=(x,\xi)\in\RR^{2d}$.\\
\indent Let $f$ be a  positive continuous functions on $\RR^{2d}$ with ultrapolynomial growth of class $*$. Then \cite{PP1}, $U\subseteq FS_{A_p,\rho}^{*,\infty}(\RR^{2d};B')$ is subordinated to $f$ in $FS_{A_p,\rho}^{*,\infty}(\RR^{2d})$, in notation $U\precsim f$, if the following estimate holds: there exists $B\geq B'$ such that for every $h>0$ there
exists $C>0$ (resp. there exist $h,C>0$) such that
\beqs
\sup_{j\in\NN}\sup_{\alpha\in\NN^{2d}}\sup_{w\in
Q_{Bm_j}^c}\frac{\left|D^{\alpha}_w
a_j(w)\right|\langle
w\rangle ^{\rho(|\alpha|+2j)}}{h^{|\alpha|+2j}A_{|\alpha|+2j}f(w)}\leq
C,\,\,\, \mbox{for all}\,\, \ssum a_j\in U.
\eeqs
Given $U\subseteq FS_{A_p,\rho}^{*,\infty}(\RR^{2d};B_1)$ with $U\precsim f$, we say that a bounded set $V$ in $\Gamma_{A_p,\rho}^{(M_p),\infty}(\RR^{2d};m)$ for some $m>0$ (resp. in $\Gamma_{A_p,\rho}^{\{M_p\},\infty}(\RR^{2d};h)$ for some $h>0$) is subordinated to $U$ under $f$, in notations $V\precsim_f U$, if there exists a surjective mapping $\Sigma:U\rightarrow V$ such that the following estimate holds: there exists $B\geq B_1$ such that for every $h>0$ there
exists $C>0$ (resp. there exist $h,C>0$) such that for all $\sum_j
a_j\in U$ and the corresponding $\Sigma(\sum_j a_j)=a\in V$
\beqs
\sup_{n\in\ZZ_+}\sup_{\alpha\in\NN^{2d}}\sup_{w\in
Q_{Bm_n}^c}\frac{\left|D^{\alpha}_w\left(a(w)-
\sum_{j<n}a_j(w)\right)\right|\langle
w\rangle^{\rho(|\alpha|+2n)}}{h^{|\alpha|+2n}A_{|\alpha|+2n}f(w)}\leq
C.
\eeqs
If $V\precsim_f U$ and if
we denote by $\tilde{V}$ the image of $V$ under the canonical
inclusion $\Gamma_{A_p,\rho}^{*,\infty}(\RR^{2d})\rightarrow
FS_{A_p,\rho}^{*,\infty}(\RR^{2d};0)$, $a\mapsto
a+\sum_{j\in\ZZ_+}0$, then $\tilde{V}\precsim f$ in $FS_{A_p,\rho}^{*,\infty}(\RR^{2d};0)$ (see \cite[Section 3.1]{PP1}).
In such a case, we slightly abuse notation and write $V\precsim f$. If $V$ contains only one element $a$ we will often write $a\precsim f$ instead.
The above estimate also implies $\Sigma(\sum_j a_j)\sim\sum_j a_j$. To
see that for a given $U\subseteq
FS_{A_p,\rho}^{*,\infty}(\RR^{2d};B)$ there always exists
$V\precsim_f U$, we can proceed as follows. Let
$\psi\in\DD^{(A_p)}(\RR^{d})$ in the $(M_p)$ case and
$\psi\in\DD^{\{A_p\}}(\RR^{d})$ in the $\{M_p\}$ case
respectively, such that $0\leq \psi\leq 1$, $\psi(\xi)=1$ when
$\langle\xi\rangle\leq 2$ and $\psi(\xi)=0$ when
$\langle\xi\rangle\geq 3$. Set $\chi(x,\xi)=\psi(x)\psi(\xi)$,
$\chi_{n,R}(w)=\chi(w/(Rm_n))$ for $n\in\ZZ_+$ and $R>0$ and put
$\chi_{0,R}(w)=0$. Given $U\subseteq
FS_{A_p,\rho}^{*,\infty}(\RR^{2d};B)$ as above, for $\sum_j a_j\in
U$ denote $R(\sum_j a_j)(w)= \sum_{j=0}^{\infty}
(1-\chi_{j,R}(w))a_j(w)$. If $R> B$, this is a well defined smooth function on $\RR^{2d}$ since the series is locally finite.

\begin{proposition}[{\cite[Proposition 3.3]{PP1}}]\label{subexistest}
Let $U$ and $f$ be as above. There exists $R_0>B$ such that for each $R\geq R_0$ the set $V_R=\{R(\sum_j a_j)|\, \sum_j a_j\in U\}$ is bounded in $\Gamma_{A_p,\rho}^{(M_p),\infty}(\RR^{2d};m)$ for some $m>0$ (resp. in $\Gamma_{A_p,\rho}^{\{M_p\},\infty}(\RR^{2d};h)$ for some $h>0$) and $V_R\precsim_f U$; in this case the surjective mapping $\Sigma$ is $R:U\rightarrow V_R$.
\end{proposition}

If $a\in \Gamma_{A_p,\rho}^{*,\infty}(\RR^{2d})$ satisfies $a\sim 0$, then $\Op_{\tau}(a)\in\mathcal{L}(\SSS'^*(\RR^d),\SSS^*(\RR^d))$ for each $\tau\in\RR$ (see \cite[Theorem 3]{BojanP}). Moreover, we have the following result.

\begin{proposition}[{\cite[Proposition 3.4]{PP1}}]\label{eqsse}
Let $V$ be a bounded subset of $\Gamma_{A_p,\rho}^{(M_p),\infty}(\RR^{2d};\tilde{m})$ for some $\tilde{m}>0$ (resp. of $\Gamma_{A_p,\rho}^{\{M_p\},\infty}(\RR^{2d};\tilde{h})$ for some $\tilde{h}>0$). Assume that there exist $B,m>0$ such that for every $h>0$ there exists $C>0$ (resp. there exist $B,h>0$ such that for every $m>0$ there exists $C>0$) such that
\beqs
\sup_{a\in V}\sup_{N\in\ZZ_+}\sup_{\alpha\in\NN^{2d}}\sup_{w\in Q_{Bm_N}^c}\frac{\left|D^{\alpha}_w a(w)\right|\langle w\rangle^{\rho(|\alpha|+2N)}}{h^{|\alpha|+2N}A_{|\alpha|+2N}e^{M(m|w|)}}\leq C.
\eeqs
Then, for each $\tau\in\RR$, $\{\mathrm{Op}_{\tau}(a)|\, a\in U\}$ is an equicontinuous subset of $\mathcal{L}(\SSS'^*(\RR^d),\SSS^*(\RR^d))$.
\end{proposition}

We will often call the elements of $\mathcal{L}(\SSS'^*(\RR^d),\SSS^*(\RR^d))$ $*$-regularising operators.\\
\indent The class of $\Psi$DOs with symbols in $\Gamma_{A_p,\rho}^{*,\infty}(\RR^{2d})$ is closed modulo $*$-regularising operators with respect to composition, adjoints and change of quantisation, cf. \cite{BojanP}.

\subsection{Weyl quantisation. The sharp product and ring
structure of $FS_{A_p,\rho}^{*,\infty}(\RR^{2d};B)$} \label{sub Weyl quantasation}

In the sequel, we will be particularly interested in the Weyl quantisation, i.e. the quantisation obtained for $\tau=1/2$. As standard, we use $a^w$ as shorthand for $\Op_{1/2}(a)$. We recall few necessary results from \cite{PP1}.\\
\indent Let $\sum_j a_j,\sum_j b_j\in
FS_{A_p,\rho}^{*,\infty}(\RR^{2d};B)$. We define their sharp
product $\sum_j a_j \# \sum_j b_j$, via the formal series $\sum_j
c_j=\sum_j a_j \# \sum_j b_j$, where
\beqs
c_j(x,\xi)=\sum_{s+k+l=j}\sum_{|\alpha+\beta|=l}\frac{(-1)^{|\beta|}} {\alpha!\beta!2^l}\partial^{\alpha}_{\xi}D^{\beta}_x
a_s(x,\xi)\partial^{\beta}_{\xi} D^{\alpha}_x b_k(x,\xi),\,\,
(x,\xi)\in Q^c_{Bm_j}.
\eeqs
It is easy to verify that $\sum_j c_j$
is a well defined element of $FS_{A_p,\rho}^{*,\infty}(\RR^{2d};B)$.
If $a\in \Gamma_{A_p,\rho}^{*,\infty}(\RR^{2d})$, then $a\#\sum_j
b_j$ will denote the $\#$ product of the image of $a$ under the
canonical inclusion
$\Gamma_{A_p,\rho}^{*,\infty}(\RR^{2d})\rightarrow
FS_{A_p,\rho}^{*,\infty}(\RR^{2d};B)$ and $\sum_j b_j$. The same convention applies if $b\in \Gamma_{A_p,\rho}^{*,\infty}(\RR^{2d})$ or if both
$a,b\in \Gamma_{A_p,\rho}^{*,\infty}(\RR^{2d})$.

\begin{remark}\label{real-valuedsym}
If $\sum_j a_j,\sum_j b_j\in FS_{A_p,\rho}^{*,\infty}(\RR^{2d};B)$ and $\sum_j c_j=\sum_j a_j\#\sum_jb_j$, then $\sum_j\overline{c_j}=\sum_j\overline{b_j}\#\sum_j\overline{a_j}$. In particular, if $a_j$ and $b_j$ are real-valued for all $j\in\NN$ and $\sum_ja_j\#\sum_jb_j=\sum_jb_j\#\sum_ja_j$, then $c_j$ are real-valued for all $j\in\NN$.
\end{remark}

As one might expect, the $\#$-product corresponds to the composition of two Weyl quantisations: for $a,b\in \Gamma_{A_p,\rho}^{*,\infty}(\RR^{2d})$, $a^wb^w-c^w$ is $*$-regularising where $c\in \Gamma_{A_p,\rho}^{*,\infty}(\RR^{2d})$ has asymptotic expansion $a\# b$; i.e. $c\sim a\# b$. We have the following more precise result.

\begin{theorem}[{\cite[Theorem 4.2]{PP1}, \cite[Corollary 4.3]{PP1}}]\label{weylq}
Let $U_1,U_2\subseteq FS_{A_p,\rho}^{*,\infty}(\RR^{2d};B)$ be
such that $U_1\precsim f_1$ and $U_2\precsim f_2$ in
$FS_{A_p,\rho}^{*,\infty}(\RR^{2d};B)$ for some continuous
positive functions $f_1$ and $f_2$ with ultrapolynomial growth of
class $*$. Then:
\begin{itemize}
\item[$(i)$] $U_1\#U_2\precsim f_1f_2$ in $FS_{A_p,\rho}^{*,\infty}(\RR^{2d};B)$.
\item[$(ii)$] For $\sum_j a_j\in U_1$ and $\sum_j b_j\in U_2$ denote $\sum_j c_{j,a,b}=\sum_j a_j\#\sum_j b_j\in U_1\# U_2$. Then, there exists $R>0$, which can be chosen arbitrarily large, such that
\beqs
\left\{a^wb^w-c^w\big|\, a=R(\ssum a_j),\, b=R(\ssum b_j),\, c=R(\ssum c_{j,a,b})\right\}
\eeqs
is an equicontinuous subset of $\mathcal{L}(\SSS'^*(\RR^d),\SSS^*(\RR^d))$ and
    \beq\label{krh1791}
    \left\{R(\ssum a_j\# \ssum b_j)\big|\, \ssum a_j\in U_1,\, \ssum b_j\in U_2\right\}\precsim_{f_1f_2}U_1\#U_2.
    \eeq
\end{itemize}
\end{theorem}

\begin{remark}
Theorem \ref{weylq} $(ii)$ is applicable when $U_1$ and $U_2$ are bounded subsets of $\Gamma_{A_p,\rho}^{(M_p),\infty}(\RR^{2d};m)$ for some $m>0$ (resp. of $\Gamma_{A_p,\rho}^{\{M_p\},\infty}(\RR^{2d};h)$ for some $h>0$). In this case, the theorem reads: there exists $R>0$, which can be chosen arbitrary large, such that $\{a^wb^w-\Op_{1/2}(R(a\# b))|\, a\in U_1,\, b\in U_2\}$ is an equicontinuous $*$-regularising set and $\{R(a\# b)|\, a\in U_1,\, b\in U_2\}$ is bounded in $\Gamma_{A_p,\rho}^{(M_p),\infty}(\RR^{2d};m)$ for some $m>0$ (resp. of $\Gamma_{A_p,\rho}^{\{M_p\},\infty}(\RR^{2d};h)$ for some $h>0$).
\end{remark}

\begin{proposition}[{\cite[Proposition 4.5]{PP1}}]\label{hyporingg}
For each $B\geq 0$, $FS_{A_p,\rho}^{*,\infty}(\RR^{2d};B)$ is a
ring with the pointwise addition and multiplication given by $\#$.
Moreover, the multiplication
$\#:FS_{A_p,\rho}^{*,\infty}(\RR^{2d};B)\times
FS_{A_p,\rho}^{*,\infty}(\RR^{2d};B)\rightarrow
FS_{A_p,\rho}^{*,\infty}(\RR^{2d};B)$ is hypocontinuous.
\end{proposition}

The multiplicative unity of the ring $FS_{A_p,\rho}^{*,\infty}(\RR^{2d};B)$ is $\mathbf{1}=1+0+0+\ldots$.

\begin{remark}\label{ktrsln159}
For $k\in\ZZ_+$ and $\sum_j a_j\in FS_{A_p,\rho}^{*,\infty}(\RR^{2d};B)$, we will use $(\sum_j a_j)^{\# k}$ as a shorthand for $\ds\underbrace{\ssum a_j\#\ldots\# \ssum a_j}_{k}$. Additionally, $(\sum_j a_j)^{\# 0}$ will just mean $\mathbf{1}$.
\end{remark}

\subsection{Realisations on $L^2(\RR^d)$. Hypoelliptic operators of infinite order} \label{sub realisations}

We start by recalling the notion of hypoellipticity (\cite[Definition 1.1]{CPP}). A symbol $a\in\Gamma^{*,\infty}_{A_p,\rho}(\RR^{2d})$ is hypoelliptic if
\begin{itemize}
\item[$i$)] there exists $B>0$ such that there are $c,m>0$ (resp. for every $m>0$ there is $c>0$) such that
\beq\label{dd1}
|a(x,\xi)|\geq c e^{-M(m|x|)-M(m|\xi|)},\quad
(x,\xi)\in Q^c_B,
\eeq
\item[$ii$)] there exists $B>0$ such that for every $h>0$ there is $C>0$ (resp. there are $h,C>0$) such that
\beq\label{dd2}
\left|D^{\alpha}_{\xi}D^{\beta}_x
a(x,\xi)\right|\leq
C\frac{h^{|\alpha|+|\beta|}|a(x,\xi)|A_{\alpha}A_{\beta}} {\langle(x,\xi)\rangle^{\rho(|\alpha|+|\beta|)}},\,\,
\alpha,\beta\in\NN^d,\, (x,\xi)\in Q^c_B.
\eeq
\end{itemize}

We can explicitly write the asymptotic expansion of a parametrix of a hypoelliptic Weyl quantisation.

\begin{proposition}[{\cite[Proposition 5.2]{PP1}}]\label{parametweyl}
Let $a\in\Gamma^{*,\infty}_{A_p,\rho}(\RR^{2d})$ be hypoelliptic. Define $q_0(w)=a(w)^{-1}$ on $Q^c_B$ and inductively, for $j\in\ZZ_+$,
\beqs
q_j(x,\xi)=-q_0(x,\xi)\sum_{s=1}^j\sum_{|\alpha+\beta|=s} \frac{(-1)^{|\beta|}}{\alpha!\beta!2^s}\partial^{\alpha}_{\xi} D^{\beta}_x q_{j-s}(x,\xi) \partial^{\beta}_{\xi} D^{\alpha}_x a(x,\xi),\,\, (x,\xi)\in Q^c_B.
\eeqs
Then, for every $h>0$ there exists $C>0$ (resp. there exist $h,C>0$) such that
\beq\label{estofpara}
\left|D^{\alpha}_w q_j(w)\right|\leq C\frac{h^{|\alpha|+2j}A_{|\alpha|+2j}}{|a(w)|\langle w\rangle^{\rho(|\alpha|+2j)}},\,\, w\in Q^c_B,\,\alpha\in\NN^{2d},\, j\in\NN.
\eeq
If $B\leq 1$, then $(\sum_j q_j)\# a= \mathbf{1}$ in $FS_{A_p,\rho}^{*,\infty}(\RR^{2d};0)$. If $B>1$, one can extend $q_0$ to an element of $\Gamma^{*,\infty}_{A_p,\rho}(\RR^{2d})$ by modifying it on $Q_{B'}\backslash Q_B$, for $B'>B$. In this case $\sum_j q_j\in FS_{A_p,\rho}^{*,\infty}(\RR^{2d};B')$, $((\sum_j q_j)\# a)_k=0$ on $Q^c_{B'}$, $\forall k\in\ZZ_+$, and $((\sum_j q_j)\#a)_0-1=q_0a-1$ belongs to $\DD^{(A_p)}(\RR^{2d})$ (resp. $\DD^{\{A_p\}}(\RR^{2d})$).\\
\indent In particular, for $q\sim\sum_j q_j$ there exists $*$-regularising operator $T$ such that $q^wa^w=\mathrm{Id}+T$.
\end{proposition}

Since hypoelliptic operators have parametrices, they are globally $\SSS^*$-regular, i.e. globally ultra-regular of class $*$, meaning if $a^wf\in \SSS^*(\RR^d)$ then $f\in\SSS^*(\RR^d)$.\footnote{When $a$ is hypoelliptic, then $\Op_{\tau}(a)$ is $\SSS^*$-regular for any $\tau\in\RR$ since we can always change the quantisation modulo $*$-regularising operator, see \cite[Proposition 3.5]{PP1}.}

\begin{remark}\label{kts951307}
A similar construction yields $\tilde{q}\in\Gamma^{*,\infty}_{A_p,\rho}(\RR^{2d})$ such that $a^w\tilde{q}^w=\mathrm{Id}+\tilde{T}$ with $\tilde{T}\in\mathcal{L}(\SSS'^*(\RR^d),\SSS^*(\RR^d))$ (see \cite[Subsection 6.2.1]{PP1} for more details). Knowing this, it is easy to prove that we can use the left parametrix $q^w$ as a right one as well, i.e. there exists $T_1,T_2\in\mathcal{L}(\SSS'^*(\RR^d),\SSS^*(\RR^d))$ such that $q^wa^w=\mathrm{Id}+T_1$ and $a^wq^w=\mathrm{Id}+T_2$.
\end{remark}

\begin{remark}\label{ktv957939}
For hypoelliptic $a\in\Gamma^{*,\infty}_{A_p,\rho}(\RR^{2d})$, we can construct a parametrix $q$ out of $\sum_j q_j\in FS_{A_p,\rho}^{*,\infty}(\RR^{2d};B')$ in a specific way. Namely (see \cite[Remark 8.7]{PPV-BMS} for the details), there exists $R>0$ and a $*$-regularising operator $T$ such that $q^wa^w=\mathrm{Id}+T$, where $q=R(\sum_j q_j)\in \Gamma^{*,\infty}_{A_p,\rho}(\RR^{2d})$ satisfies the following conditions: there exist $B''\geq B'$ and $c'',C''>0$ such that
\beq\label{ktr991509}
c''/|a(w)|\leq |q(w)|\leq C''/|a(w)|,\,\, \forall w\in Q^c_{B''}.
\eeq
Moreover, for every $h>0$ there exists $C>0$ (resp. there exist $h,C>0$) such that
\beq\label{ktl997133}
\left|D^{\alpha}_w q(w)\right|\leq Ch^{|\alpha|}A_{\alpha}|a(w)|^{-1}\langle w\rangle^{-\rho|\alpha|},\,\, w\in Q^c_{B''},\,\alpha\in\NN^{2d}.
\eeq
Thus, $q$ is hypoelliptic. If we additionally assume that $|a(w)|\rightarrow \infty$ as $|w|\rightarrow\infty$, then it follows that $q^w$ is compact operator on $L^2(\RR^d)$ (see \cite[Remark 8.7]{PPV-BMS}).
\end{remark}

\begin{remark}\label{kth995559}
Let $b\in\Gamma^{*,\infty}_{A_p,\rho}(\RR^{2d})$ be hypoelliptic and $\tau\in\RR$. Applying \cite[Proposition 3.5]{PP1}, Remark \ref{kts951307} and Remark \ref{ktv957939} one can find a hypoelliptic $\tilde{q}\in\Gamma^{*,\infty}_{A_p,\rho}(\RR^{2d})$ such that $c'_1/|b(w)|\leq |\tilde{q}(w)|\leq c'_2/|b(w)|$, $\forall w\in Q^c_{B'_1}$, for some $c'_1,c'_2,B'_1>0$, and $\Op_{\tau}(\tilde{q})\Op_{\tau}(b)-\mathrm{Id}$ and $\Op_{\tau}(b)\Op_{\tau}(\tilde{q})-\mathrm{Id}$ are $*$-regularising. This immediately yields that if $\tilde{q}_1\in\Gamma^{*,\infty}_{A_p,\rho}(\RR^{2d})$ is any other left $\tau$-parametrix of $b$, i.e. $\Op_{\tau}(\tilde{q}_1)\Op_{\tau}(b)-\mathrm{Id}\in \mathcal{L}(\SSS'^*(\RR^d),\SSS^*(\RR^d))$, then $\Op_{\tau}(\tilde{q}_1)-\Op_{\tau}(\tilde{q})$ is $*$-regularising, which, in turn, yields that we can use $\Op_{\tau}(\tilde{q}_1)$ as a right parametrix as well.
\end{remark}

Let $a\in\Gamma^{*,\infty}_{A_p,\rho}(\RR^{2d})$ and $A$
be the corresponding unbounded operator on $L^2(\RR^d)$ with domain $\SSS^*(\RR^d)$
defined as $A \varphi=a^w\varphi$, $\varphi\in\SSS^*(\RR^d)$. Considering
$a^w$ as a mapping on  $\SSS'^*(\RR^d)$, its restriction to the subspace $\{g\in L^2(\RR^d)|\, a^wg\in
L^2(\RR^d)\}$ defines a closed extension of $A$ which is called
the maximal realisation of $A$. As standard, we denote by $\overline{A}$ the
closure of $A$, also called the minimal realisation of $A$. When $a$ is hypoelliptic, the minimal and maximal realisations coincide and they are given by the restriction of $a^w$ on the domain of $\overline{A}$; if additionally $a$ is real-valued, then $\overline{A}$ is a self-adjoint operator on $L^2(\RR^d)$ (see \cite[Proposition 4.4]{PPV-BMS}).

\subsection{Spectrum and the asymptotics of the eigenvalue counting function for infinite order $\Psi$DOs}\label{spec1}
We start by pointing out that the spectrum of the closure of a hypoelliptic operator with real-valued symbol whose absolute value tends to infinity solely consists of a sequence of unbounded eigenvalues, as stated in the ensuing proposition.
\begin{proposition}[{\cite[Proposition 4.6]{PPV-BMS}}]\label{discretness_of_spe}
Let $a\in\Gamma_{A_p,\rho}^{*,\infty}(\RR^{2d})$ be a hypoelliptic
real-valued symbol such that $|a(w)|\rightarrow \infty$ as
$|w|\rightarrow \infty$ and let $A$ be the unbounded
operator on $L^2(\RR^d)$ defined by $a^w$. Then the closure $\overline{A}$ of $A$
is a self-adjoint operator having spectrum given by a sequence of
real eigenvalues either diverging to $+\infty$ or to $-\infty$
according to the sign of $a$ at infinity. The eigenvalues have
finite multiplicities and the eigenfunctions belong to
$\SSS^*(\RR^d)$. Moreover, $L^2(\RR^d)$ has an orthonormal basis
consisting of eigenfunctions of $\overline{A}$.
\end{proposition}

For such $a$ tending to $+\infty$ as $|w|\rightarrow\infty$, under some additional hypothesis on its growth, there is an explicit formula for the asymptotical behaviour of the eigenvalue counting function of $\bar{A}$,
$
N(\lambda)=\sum_{\lambda_j\leq \lambda}1=\#\{j\in\mathbb{N}|\,\lambda_{j}\leq \lambda\},
$
where $\lambda_{0}\leq \lambda_{1}\leq \lambda_{2}\leq \dots\leq \lambda_{j}\leq \dots$ are the eigenvalues of $\bar{A}$ with multiplicities taken into account. To be precise, let $f:\RR\rightarrow\RR$ be positive, strictly increasing, of ultrapolynomial growth of class $*$ on some interval $[Y,\infty)$, for some $Y>0$, and absolutely continuous on each compact subinterval of $[Y,\infty)$. Denote $\sigma(\lambda)=(f^{-1}(\lambda))^{2d}$ for large $\lambda>0$.

\begin{theorem}[{\cite[Theorem 5.1]{PPV-BMS}}]\label{Weylth1}
Let $a\in\Gamma_{A_p,\rho}^{*,\infty}(\RR^{2d})$ be hypoelliptic, let
 $f$ satisfy
\beq
\label{weyleq2}
\lim_{y\to\infty} \frac{yf'(y)}{f(y)}=\infty,
\eeq
and let $\Phi$ be a positive continuous function on the sphere $\mathbb{S}^{2d-1}$.
Suppose that for each $\varepsilon\in (0,1)$ there are positive constants $c_{\epsilon},C_{\epsilon},B_{\epsilon}>0$ such that
\beq
\label{weyleq3}
c_{\varepsilon}f((1-\varepsilon) r \Phi(\vartheta))\leq a(r\vartheta)\leq C_{\varepsilon}f((1+\varepsilon) r \Phi(\vartheta)),
\eeq
for all $r\geq B_{\varepsilon}$ and $\vartheta\in\mathbb{S}^{2d-1}$. Then, $\lambda_{j}=f\left(\gamma j^{\frac{1}{2d}}(1+o(1))\right)$, $j\to\infty$, and
\beqs
\lim_{\lambda\to\infty}\frac{N(\lambda)}{\sigma(\lambda)}= \frac{\pi}{(2\pi)^{d+1}d}\int_{\mathbb{S}^{2d-1}}\frac{d\vartheta}{(\Phi(\vartheta))^{2d}}\:,
\eeqs
with $\gamma=\sqrt{2\pi}\cdot (2d/\int_{\mathbb{S}^{2d-1}}(\Phi(\vartheta))^{-2d}d\vartheta)^{\frac{1}{2d}}$.
Moreover, for each $h'<\gamma<h$,
\beqs
\lim_{j\to\infty}\frac{\lambda_{j}}{f(h' j^{\frac{1}{2d}})}=\infty
\quad \mbox{and} \quad \lim_{j\to\infty}\frac{\lambda_{j}}{f(h j^{\frac{1}{2d}})}=0.
\eeqs
\end{theorem}

Note that Theorem \ref{Weylth1} deals with operators which are truly of infinite order because integration of (\ref{weyleq2}) gives that $\langle w \rangle^{\beta}=o(a(w))$ for any $\beta>0$.

\section{Power series of Shubin type differential operators}\label{powerseriesofop}

Our main goal in this section is to study operators given as power series of elliptic Shubin differential operators and, more importantly, their ``lower order'' perturbations. The symbolic calculus which is recalled above is an effective tool for studying such operators. In fact, under certain conditions on the coefficients appearing in the power series, we will prove that these are in fact $\Gamma_{A_p,\rho}^{*,\infty}$-hypoelliptic which in turn yields their global ultra-regularity. Furthermore, by applying Theorem \ref{Weylth1}, we will explicitly give the asymptotic behaviour of their eigenvalue counting function. The motivating example presented in the introduction is the case of power series of the harmonic oscillator.

\subsection{The iterated $\#$-product of polynomial symbols}

\begin{lemma}\label{shproduct}
Let $a\in\Gamma^{*,\infty}_{A_p,\rho}(\RR^{2d})$. Then, for each $n\in\ZZ_+$, $n\geq 2$, and $j\in\NN$, we have\footnote{See Remark \ref{ktrsln159} for the meaning of $a^{\# n}$.}
\beqs
(a^{\# n})_j(w)&=&\sum_{s_1+\ldots+s_{n-1}=j}\,\, \sum_{|\boldsymbol{\alpha}^{1}+\boldsymbol{\beta}^{1}|=s_1,\ldots, |\boldsymbol{\alpha}^{n-1}+\boldsymbol{\beta}^{n-1}|=s_{n-1}} \frac{(-1)^{|\tilde{\boldsymbol{\beta}}|}} {\tilde{\boldsymbol{\alpha}}!\tilde{\boldsymbol{\beta}}!(2i)^j}\\
&{}&\cdot\prod_{l=1}^n \partial^{\beta^{l-1,1}+\ldots+\beta^{l-1,l-1}+\alpha^{l,l}+\ldots+\alpha^{n-1,l}}_{\xi} \partial^{\alpha^{l-1,1}+\ldots+\alpha^{l-1,l-1}+\beta^{l,l}+\ldots+\beta^{n-1,l}}_x a(w),
\eeqs
where $\boldsymbol{\alpha}^l=(\alpha^{l,1},\ldots,\alpha^{l,l})\in\NN^{dl}$, $\boldsymbol{\beta}^l=(\beta^{l,1},\ldots,\beta^{l,l})\in\NN^{dl}$, for $l=1,\ldots,n-1$, and $\tilde{\boldsymbol{\alpha}}=(\boldsymbol{\alpha}^1,\ldots,\boldsymbol{\alpha}^{n-1})\in \NN^{dn(n-1)/2}$, $\tilde{\boldsymbol{\beta}}=(\boldsymbol{\beta}^1,\ldots,\boldsymbol{\beta}^{n-1})\in \NN^{dn(n-1)/2}$. Furthermore, $(a^{\# n})_0(w)=(a(w))^n$. If $a$ is real-valued, then $(a^{\#n})_j$ are real-valued for all $j\in\NN$.\\
\indent If $a$ is a polynomial of degree $m\in\ZZ_+$ in the variable $w$, then $(a^{\# n})_j=0$ for all $j>[nm/2]$, and the polynomial $(a^{\# n})_j$ has degree at most $nm-2j$, $j=1,\ldots, [nm/2]$.
\end{lemma}

\begin{remark}
In the above formula we apply the principle of vacuous summation in the summation of the multi-indices $\alpha$ and $\beta$ in the derivatives in $x$ and $\xi$. This means that when $l=1$ the corresponding term is just $\partial^{\alpha^{1,1}+\ldots+\alpha^{n-1,1}}_{\xi} \partial^{\beta^{1,1}+\ldots+\beta^{n-1,1}}_x a(w)$; similarly when $l=n$.
\end{remark}

\begin{proof} To prove the formula for $(a^{\# n})_j$, notice that when $n=2$ this is nothing else but the definition of $(a\# a)_j$, $j\in\NN$. The proof can be done by induction on $n$. The fact $(a^{\# n})_0(w)=(a(w))^n$ follows directly from the formula. The fact that $(a^{\# n})_j$, $j\in\NN$, are real-valued when such is $a$ readily follows from Remark \ref{real-valuedsym} and Proposition \ref{hyporingg}.\\
\indent Assume that $a$ is a polynomial of degree $m\in\ZZ_+$. Fix $j\in\ZZ_+$ and notice that the degree of each term in the product that appears in $(a^{\# n})_j$ is at most
\begin{multline*}
m-|\beta^{l-1,1}+\ldots+\beta^{l-1,l-1}+\alpha^{l,l}+\ldots+\alpha^{n-1,l}|\\
-|\alpha^{l-1,1}+\ldots+\alpha^{l-1,l-1}+\beta^{l,l}+\ldots+\beta^{n-1,l}|,\quad \forall l=1,\ldots, n.
\end{multline*}
If we sum these quantities and notice that each multi-index $\alpha^{l,k}$ and $\beta^{l,k}$, $1\leq k\leq l\leq n-1$, appears exactly twice, we obtain that the degree of $(a^{\# n})_j$ is at most $nm-2j$. This implies the second part of the lemma.
\end{proof}

Let $a$ be a polynomial in $w$ of degree $m\in\ZZ_+$. Defining
\beq\label{kktlt17}
a^{(\# n)}(w)=\sum_{k=0}^{\infty}(a^{\# n})_k(w)=\sum_{k=0}^{[nm/2]}(a^{\# n})_k(w),
\eeq
Lemma \ref{shproduct} proves that $a^{(\# n)}$ is a polynomial in $w$ of degree $nm$. As a simple, but important, consequence of Lemma \ref{shproduct} we deduce that $\sum_{k=0}^{\infty}(a^{(\# n)}\# a)_k=a^{(\# (n+1))}$, $n\in\ZZ_+$, $n\geq 2$. Of course, this formula remains valid even for $n=1$ if we put $a^{(\# 1)}=a^{\# 1}=a$ (cf. Remark \ref{ktrsln159}). The importance of this observation lies in the following fact.

\begin{lemma}\label{ktn557713}
Let $a$ be a polynomial in $w$ of degree $m\in\ZZ_+$. With the same notations as above, we have $(a^w)^n=(a^{(\#n)})^w$, $\forall n\in\ZZ_+$. Furthermore, if $a$ is real-valued, then the same holds for $a^{(\# n)}$ for all $n\in\ZZ_+$.
\end{lemma}

\begin{proof} If $b_1$ and $b_2$ are polynomials in $w$, then \cite[Theorem 1.2.17, p. 34]{NR}, especially its proof, implies $b_1^wb_2^w=(\sum_{k=0}^{\infty} (b_1\# b_2)_k)^w$. Notice the sum is finite and it is a polynomial. For $n=2$, the claim in the lemma directly follows from this fact. Arguing by induction on $n$, and using the same fact, we conclude $(a^{(\# n)})^wa^w=(\sum_{k=0}^{\infty} (a^{(\# n)}\# a)_k)^w=(a^{(\#(n+1))})^w$. The very last equality follows from the fact we just proved before this lemma. Notice that the claim in the lemma is trivial for $n=1$. To finish the proof, when $a$ is real-valued $a^{(\# n)}$ is also real-valued by Lemma \ref{shproduct} and (\ref{kktlt17}).
\end{proof}

\subsection{Series of polynomial symbols}\label{power series subsection}

Let $a_n$, $n\in\ZZ_+$, be polynomials in $w\in\RR^{2d}$ all of them of degree at most $m\in\ZZ_+$ satisfying the following estimate: there exists $C_1\geq 1$ such that
\beq\label{obk77155311}
|D^{\alpha}a_n(w)|\leq C_1\langle w\rangle^{m-|\alpha|},\,\, \forall w\in\RR^{2d},\, \forall\alpha\in\NN^{2d},\, \forall n\in\ZZ_+.
\eeq
For each $a_n$, $n\in\ZZ_+$, let $a^{(\# n)}_n$ denote the polynomial (\ref{kktlt17}); the degree of $a^{(\# n)}_n$ is at most $nm$. We additionally define $a_0^{(\# 0)}(w)=1$, $\forall w\in\RR^{2d}$.

\begin{proposition}\label{lemfor135}
Let $a_n$, $n\in\ZZ_+$, and $a^{(\# n)}_n$, $n\in\NN$, be as above. Let $\zeta_n$, $n\in\NN$, be a sequence of complex numbers such that $\zeta_0=1$ and there exist $C',L'>0$ (resp. for every $L'>0$ there exists $C'>0$) such that $|\zeta_n|\leq C'L'^n/M^m_n$, $\forall n\in\NN$. Then the following statements hold true.
\begin{itemize}
\item[$(i)$] The function $w\mapsto \sum_{n=0}^{\infty}|\zeta_n||a_n(w)|^n$, $\RR^{2d}\rightarrow [0,\infty)$, is continuous and of ultrapolynomial growth of class $*$.
\item[$(ii)$] The series $\sum_{n=0}^{\infty}\zeta_na^{(\# n)}_n(w)$ and $\sum_{n=0}^{\infty}\zeta_n(a_n(w))^n$ absolutely converge in $\Gamma^{*,\infty}_{A_p,1}(\RR^{2d})$.
\item[$(iii)$] The series $\sum_{n=0}^{\infty}\zeta_n(a_n^w)^n$ absolutely converges in the spaces $\mathcal{L}_b(\SSS^*(\RR^d),\SSS^*(\RR^d))$ and $\mathcal{L}_b(\SSS'^*(\RR^d),\SSS'^*(\RR^d))$. It is a pseudo-differential operator with Weyl symbol $\sum_{n=0}^{\infty}\zeta_n a^{(\# n)}\in \Gamma^{*,\infty}_{A_p,1}(\RR^{2d})$; that is
    \beqs
    \left(\sum_{n=0}^{\infty}\zeta_na^{(\# n)}_n\right)^w=\sum_{n=0}^{\infty}\zeta_n(a_n^w)^n.
    \eeqs
\end{itemize}
\end{proposition}
\begin{remark}
The assumption on the sequence $\zeta_n$ is equivalent \cite[Proposition 4.5]{Komatsu1} to the growth estimate
$\left|\sum_{n=0}^{\infty}\zeta_n\lambda^n\right|\leq C_{h} e^{M(h |\lambda|^{1/m})}$, $\lambda\in \mathbb{C}$, for some $h>0$ (for each $h>0$).
\end{remark}
\begin{proof} Part $(i)$ follows directly from (\ref{obk77155311}). We prove $(ii)$. By Lemma \ref{shproduct}, for $n\in\ZZ_+$, $n\geq 2$, we have
\begin{align*}
D^{\gamma}_w& a^{(\# n)}_n(w)\\
&=\sum_{j=0}^{[nm/2]}\sum_{s_1+\ldots+s_{n-1}=j}\,\, \sum_{|\boldsymbol{\alpha}^{1}+\boldsymbol{\beta}^{1}|=s_1,\ldots, |\boldsymbol{\alpha}^{n-1}+\boldsymbol{\beta}^{n-1}|=s_{n-1}}\,\,  \sum_{\gamma^1+\ldots+\gamma^n=\gamma}\frac{(-1)^{|\tilde{\boldsymbol{\beta}}|}} {\tilde{\boldsymbol{\alpha}}!\tilde{\boldsymbol{\beta}}!(2i)^j}\cdot \frac{\gamma!}{\gamma^1!\ldots\gamma^n!}\\
&{}\quad\cdot\prod_{l=1}^n D^{\gamma^l}_w \partial^{\beta^{l-1,1}+\ldots+\beta^{l-1,l-1}+\alpha^{l,l}+\ldots+\alpha^{n-1,l}}_{\xi} \partial^{\alpha^{l-1,1}+\ldots+\alpha^{l-1,l-1}+\beta^{l,l}+\ldots+\beta^{n-1,l}}_x a_n(w).
\end{align*}
Each term in the above product is a polynomial with degree at most
\begin{multline*}
m-|\gamma^l|-|\beta^{l-1,1}+\ldots+\beta^{l-1,l-1}+\alpha^{l,l}+\ldots+\alpha^{n-1,l}|\\
- |\alpha^{l-1,1}+\ldots+\alpha^{l-1,l-1}+\beta^{l,l}+\ldots+\beta^{n-1,l}|,\quad \forall l=1,\ldots, n.
\end{multline*}
Summing these inequalities and noticing that each multi-index $\alpha^{l,k}$ and $\beta^{l,k}$, $1\leq k\leq l\leq n-1$, appears exactly twice, we obtain that the product has degree at most $nm-|\gamma|-2j$. Hence all the summands for $j> [(nm-|\gamma|)/2]$ are identically equal to zero. Thus, application of (\ref{obk77155311}) yields
\beqs
|D^{\gamma}_wa^{(\# n)}_n(w)|&\leq&\sum_{j=0}^{[(nm-|\gamma|)/2]}\sum_{s_1+\ldots+s_{n-1}=j}\,\, \sum_{|\boldsymbol{\alpha}^{1}+\boldsymbol{\beta}^{1}|=s_1,\ldots, |\boldsymbol{\alpha}^{n-1}+\boldsymbol{\beta}^{n-1}|=s_{n-1}}\,\,  \sum_{\gamma^1+\ldots+\gamma^n=\gamma}\\
&{}&\cdot \frac{1} {\tilde{\boldsymbol{\alpha}}!\tilde{\boldsymbol{\beta}}!2^j}\cdot \frac{\gamma!}{\gamma^1!\ldots\gamma^n!} C_1^n\langle w\rangle^{nm-|\gamma|-2j},
\eeqs
where we employed the principle of vacuous summation if $[(nm-|\gamma|)/2]<0$ (i.e. if $|\gamma|>nm$). Notice that
\beqs
\sum_{|\boldsymbol{\alpha}^l+\boldsymbol{\beta}^l|=s_l} \frac{1}{\boldsymbol{\alpha}^l!\boldsymbol{\beta}^l!}=\frac{(2dl)^{s_l}}{s_l!},\,\, l=1,\ldots,n-1.
\eeqs
Hence,
\beq\label{sst57}
\sum_{s_1+\ldots+s_{n-1}=j}\,\, \sum_{|\boldsymbol{\alpha}^{1}+\boldsymbol{\beta}^{1}|=s_1,\ldots, |\boldsymbol{\alpha}^{n-1}+\boldsymbol{\beta}^{n-1}|=s_{n-1}}\frac{1} {\tilde{\boldsymbol{\alpha}}!\tilde{\boldsymbol{\beta}}!2^j}= \frac{d^j}{j!}\cdot\frac{n^j(n-1)^j}{2^j}.
\eeq
As $\sum_{\gamma^1+\ldots+\gamma^n=\gamma}\gamma!/(\gamma^1!\ldots\gamma^n!)=n^{|\gamma|}$, we conclude
\beqs
|D^{\gamma}_wa^{(\# n)}_n(w)|\leq \frac{C_1^n}{\langle w\rangle^{|\gamma|}}\sum_{j=0}^{[(nm-|\gamma|)/2]}\frac{d^jn^{|\gamma|+2j}\langle w\rangle^{nm-2j}}{2^jj!}.
\eeqs
The condition $(M.2)$ on $M_p$ gives $M_{nm}\leq c^{m-1}_0H^{nm(m+1)/2}M^m_n$, $\forall n\in\NN$, and thus
\beqs
|\zeta_n|n^{|\gamma|+2j}\leq \frac{C'c^{m-1}_0H^{nm(m+1)/2}(L'e)^n(|\gamma|+2j)!}{M_{nm}}\leq \frac{C''H^{nm(m+1)/2}(L'e)^n2^{|\gamma|+2j}|\gamma|!}{M_{nm-2j}},
\eeqs
where, the very last inequality follows from the boundedness of $p!/M_p$. We infer
\beqs
|\zeta_n||D^{\gamma}_wa^{(\# n)}_n(w)|&\leq& \frac{C''2^{|\gamma|}|\gamma|!}{2^{nm}\langle w\rangle^{|\gamma|}}\sum_{j=0}^{[(nm-|\gamma|)/2]}\frac{(4eC_1dH^{\frac{m+1}{2}}L'^{1/m})^{2j} (2eC_1H^{\frac{m+1}{2}}L'^{1/m}\langle w\rangle)^{nm-2j}}{2^jj!M_{nm-2j}}\\
&\leq&\frac{C'''2^{|\gamma|}|\gamma|!e^{M(2eC_1H^{(m+1)/2}L'^{1/m}\langle w\rangle)}}{2^{nm}\langle w\rangle^{|\gamma|}}.
\eeqs
Notice that this estimate trivially holds when $n=0,1$. We deduce that $\sum_{n=0}^{\infty}\zeta_na^{(\# n)}_n(w)$ absolutely converges in $\Gamma^{(M_p),\infty}_{A_p,1}(\RR^{2d};\tilde{m})$ for some $\tilde{m}>0$ in the Beurling case and in $\Gamma^{\{M_p\},\infty}_{A_p,1}(\RR^{2d};4)$ in the Roumieu case respectively. Consequently, $\sum_{n=0}^{\infty}\zeta_na^{(\# n)}_n(w)$ absolutely converges in $\Gamma^{*,\infty}_{A_p,1}(\RR^{2d})$. The proof for $\sum_{n=0}^{\infty}\zeta_n(a_n(w))^n$ is completely analogous and we omit it. Notice that $(iii)$ follows from the part of $(ii)$ concerning $\sum_{n=0}^{\infty}\zeta_na^{(\# n)}_n(w)$ and \cite[Proposition 3.1]{PP1} and Lemma \ref{ktn557713}.
\end{proof}

\subsection{Power series of elliptic Shubin type polynomial symbols}
\label{power series elliptic subsection}

Assume now that $a(w)=\sum_{|\gamma|\leq m}c_{\gamma}w^{\gamma}$ is a real-valued elliptic Shubin polynomial of degree $m\geq 2$, $m\in\ZZ_+$, such that $a(w)>0$ when $w\in Q^c_{B_1}$, for some $B_1\geq 1$. Clearly, $c_{\gamma}\in\RR$, for all $|\gamma|\leq m$. There exists $C_1\geq 1$ such that
\beq
&{}&|D^{\alpha}a(w)|\leq C_1\langle w\rangle^{m-|\alpha|},\,\, \forall w\in\RR^{2d},\, \forall\alpha\in\NN^{2d},\label{obk771553}\\
&{}&\langle w\rangle^m\leq C_1a(w)\,\, \mbox{and}\,\, a(w)\geq 1,\,\, \forall w\in Q^c_{B_1},\label{ktsdr15}
\eeq
by increasing $B_1$ if it is necessary. Hence, for each $0<\rho'<1$, there exists $B'=B'(\rho')\geq B_1$ such that
\beq\label{est111elp}
|D^{\alpha}a(w)|\leq (a(w))^{1-\frac{|\alpha|\rho'}{m}},\,\, \forall\alpha\in\NN^{2d},\, \forall w\in Q^c_{B'}.
\eeq
Denote by $a'$ the principle part of $a$, i.e., $a'(w)=\sum_{|\gamma|=m}c_{\gamma}w^{\gamma}$. Then $a'$ is a real-valued elliptic Shubin polynomial. Increasing $B_1$ and $C_1$ if it is necessary, we may assume that $a'$ satisfies (\ref{obk771553}) and (\ref{ktsdr15}). We retain the notation $a^{(\# n)}$, $n\in\mathbb{Z}_{+}$, for the polynomials (\ref{kktlt17}). As before, we additionally define $a^{(\# 0)}$ by $a^{(\# 0)}(w)=1$, $\forall w\in\RR^{2d}$.

\begin{theorem}\label{serpolsym}
Let $a$, $a'$ and $a^{(\# n)}$, $n\in\NN$, be as above and assume that the parameter $\rho$ in $\Gamma^{*,\infty}_{A_p,\rho}(\RR^{2d})$ is such that $\rho<1$. Let $s>1/(1-\rho)$ be such that $M_p\subset p!^s$ in the $(M_p)$ case and $M_p\prec p!^s$ in the $\{M_p\}$ case respectively. Let $\widehat{M}_n$, $n\in\NN$, be a sequence of positive numbers such that $\widehat{M}_0=1$ and there exists $C_0\geq 1$ such that
\beq\label{stk997733}
C_0^{n-k}\frac{\widehat{M}_n}{(nm)!^s}\geq \frac{\widehat{M}_k}{(km)!^s},\,\, \forall n,k\in\NN,\ \  \mbox{with}\  n\geq k.
\eeq
Let $P:\RR\rightarrow \RR$ be defined as
\beq
\label{especial ultrapolydef}
P(\lambda)=\sum_{n=0}^{\infty}\frac{\lambda^n}{\widehat{M}_n}.
\eeq
Then the following statements hold true.
\begin{itemize}
\item[$(i)$] The functions $w\mapsto P(|a(w)|)$ and $w\mapsto P(|a'(w)|)$, $\RR^{2d}\rightarrow [0,\infty)$, are continuous and  of ultrapolynomial growth of class $*$.
\item[$(ii)$] The series $\sum_{n=0}^{\infty}a^{(\# n)}(w)/\widehat{M}_n$, $\sum_{n=0}^{\infty}(a(w))^n/\widehat{M}_n$ and $\sum_{n=0}^{\infty}(a'(w))^n/\widehat{M}_n$ absolutely converge in $\Gamma^{*,\infty}_{A_p,1}(\RR^{2d})$.
\item[$(iii)$] There exist $C>1$ and $B\geq B_1$ such that for all $w\in Q^c_B$ and $\gamma\in\NN^{2d}$,
    $$
    \left|\sum_{n=0}^{\infty}\frac{D^{\gamma}_wa^{(\# n)}(w)}{\widehat{M}_n}\right|\leq \frac{C^{|\gamma|+1}P(a(w))}{\langle w\rangle^{\rho|\gamma|}},\ \quad |D^{\gamma}_w P(a(w))|\leq \frac{C^{|\gamma|+1}P(a(w))}{\langle w\rangle^{\rho|\gamma|}},
    $$
    \beq\label{ttvv19911}
    \left|D^{\gamma}_w\left(P(a(w))- \sum_{n=0}^{\infty}\frac{a^{(\# n)}(w)}{\widehat{M}_n}\right)\right|\leq \frac{C^{|\gamma|+1}P(a(w))}{\langle w\rangle^{\rho(|\gamma|+2)}}.
    \eeq
    In particular, the symbols $\sum_{n=0}^{\infty}a^{(\# n)}(w)/\widehat{M}_n$ and $P(a(w))=\sum_{n=0}^{\infty}(a(w))^n/\widehat{M}_n$ are $\Gamma^{*,\infty}_{A_p,\rho}(\RR^{2d})$-hypoelliptic.
\item[$(iv)$] There exist $C>1$ and $B\geq B_1$ such that for all $w\in Q^c_B$ and $\gamma\in\NN^{2d}$,
    \beqs
    |D^{\gamma}_w P(a'(w))|\leq \frac{C^{|\gamma|+1}P(a'(w))}{\langle w\rangle^{\rho|\gamma|}},\,\,\, \left|D^{\gamma}_w\left(P(a'(w))- \sum_{n=0}^{\infty}\frac{a^{(\# n)}(w)}{\widehat{M}_n}\right)\right|\leq \frac{C^{|\gamma|+1}P(a'(w))}{\langle w\rangle^{\rho(|\gamma|+1)}}.
    \eeqs
    In particular, the symbol $P(a'(w))=\sum_{n=0}^{\infty}(a'(w))^n/\widehat{M}_n$ is $\Gamma^{*,\infty}_{A_p,\rho}(\RR^{2d})$-hypoelliptic.
\end{itemize}
\end{theorem}

\begin{proof}
First, observe that (\ref{stk997733}) implies $\widehat{M}_n\geq (nm)!^s/C_0^n$, $\forall n\in\NN$. Part $(i)$ follows directly from this observation together with the fact $M_p\subset p!^s$ (resp. $M_p\prec p!^s$). Part $(ii)$ is just a special case of Proposition \ref{lemfor135} $(ii)$.\\
\indent Next, we prove $(iii)$. Our immediate goal is to estimate
$$D^{\gamma}_w\left(\sum_{n=0}^{\infty}\frac{a^{(\# n)}(w)}{\widehat{M}_n}-P(a(w))\right)= \sum_{n=2}^{\infty}\frac{D^{\gamma}_w(a^{(\# n)}(w)-(a(w))^n)}{\widehat{M}_n}.$$
Fix $s'$ such that $s>s'>1/(1-\rho)$ (by assumption, $s>1/(1-\rho)$). Hence $1>\rho+s'^{-1}$. Pick $\rho'$ such that $1>\rho'>\rho+ s'^{-1}$. Thus $(s'\rho'-1)/s'>\rho$. For this $\rho'$ there exists $B'\geq B_1$ for which (\ref{est111elp}) holds. Because of Lemma \ref{shproduct} and (\ref{kktlt17}), we have
\begin{align}
D^{\gamma}_w& (a^{(\# n)}(w)-(a(w))^n)\nonumber \\
&= \sum_{j=1}^{[nm/2]}\sum_{s_1+\ldots+s_{n-1}=j}\,\, \sum_{|\boldsymbol{\alpha}^{1}+\boldsymbol{\beta}^{1}|=s_1,\ldots, |\boldsymbol{\alpha}^{n-1}+\boldsymbol{\beta}^{n-1}|=s_{n-1}}\,\,  \sum_{\gamma^1+\ldots+\gamma^n=\gamma}
 \frac{(-1)^{|\tilde{\boldsymbol{\beta}}|}} {\tilde{\boldsymbol{\alpha}}!\tilde{\boldsymbol{\beta}}!(2i)^j} \frac{\gamma!}{\gamma^1!\ldots\gamma^n!}\nonumber\\
&{}\quad \cdot\prod_{l=1}^n D^{\gamma^l}_w \partial^{\beta^{l-1,1}+\ldots+\beta^{l-1,l-1}+\alpha^{l,l}+\ldots+\alpha^{n-1,l}}_{\xi} \partial^{\alpha^{l-1,1}+\ldots+\alpha^{l-1,l-1}+\beta^{l,l}+\ldots+\beta^{n-1,l}}_x a(w).\label{terforthediff}
\end{align}
Similarly as in the prove of Proposition \ref{lemfor135} $(ii)$, one verifies that in the first sum $j\leq [(nm-|\gamma|)/2]$, because, in the contrary, all summands are identically equal to zero. Thus, applying (\ref{est111elp}), we have the following estimate on $Q^c_{B'}$
\begin{align*}
|D^{\gamma}_w& (a^{(\# n)}(w)-(a(w))^n)|\\
&\leq \sum_{j=1}^{[(nm-|\gamma|)/2]}\sum_{s_1+\ldots+s_{n-1}=j}\,\, \sum_{|\boldsymbol{\alpha}^{1}+\boldsymbol{\beta}^{1}|=s_1,\ldots, |\boldsymbol{\alpha}^{n-1}+\boldsymbol{\beta}^{n-1}|=s_{n-1}}\,\,  \sum_{\gamma^1+\ldots+\gamma^n=\gamma}\\
&{}\quad \cdot \frac{1} {\tilde{\boldsymbol{\alpha}}!\tilde{\boldsymbol{\beta}}!2^j}\cdot \frac{\gamma!}{\gamma^1!\ldots\gamma^n!}\cdot (a(w))^{n-\frac{(|\gamma|+2j)\rho'}{m}},
\end{align*}
where we applied the principle of vacuous summation if $[(nm-|\gamma|)/2]\leq0$. Now, (\ref{sst57}) implies
\beq
\sum_{n=2}^{\infty}\frac{|D^{\gamma}_w(a^{(\# n)}(w)-(a(w))^n)|}{\widehat{M}_n}&\leq& \sum_{n=2}^{\infty}\sum_{j=1}^{[(nm-|\gamma|)/2]} \frac{d^jn^{|\gamma|+2j}}{2^jj!\widehat{M}_n}\cdot (a(w))^{n-\frac{(|\gamma|+2j)\rho'}{m}}\nonumber \\
&\leq& \sum_{j=1}^{\infty}\frac{d^j}{2^jj!}\sum_{n=[(|\gamma|+2j)/m]}^{\infty} \frac{n^{|\gamma|+2j}(a(w))^{n-\frac{(|\gamma|+2j)\rho'}{m}}}{\widehat{M}_n}, \label{jtv97}
\eeq
for all $w\in Q^c_{B'}$, $\gamma\in\NN^{2d}$. Observe that
\begin{align}
\sum_{n=[(|\gamma|+2j)/m]}^{\infty} & \frac{n^{|\gamma|+2j}(a(w))^{n-\frac{(|\gamma|+2j)\rho'}{m}}}{\widehat{M}_n}\nonumber \\
&= \sum_{n=[(|\gamma|+2j)/m]}^{\infty} \frac{n^{|\gamma|+2j}(a(w))^{\left(n-\frac{(|\gamma|+2j)}{m}\right)\frac{1}{s'}}} {\widehat{M}^{1/s'}_n} \cdot \frac{(a(w))^{\frac{(s'-1)n}{s'}-\frac{|\gamma|+2j}{m}\cdot \frac{s'\rho'-1}{s'}}}{\widehat{M}^{(s'-1)/s'}_n}\nonumber \\
&\leq \left(\sum_{n=[(|\gamma|+2j)/m]}^{\infty} \frac{n^{(|\gamma|+2j)s'}(a(w))^{n-\frac{|\gamma|+2j}{m}}} {\widehat{M}_n}\right)^{1/s'}\label{trk99}\\
&{}\quad\cdot
\left(\sum_{n=[(|\gamma|+2j)/m]}^{\infty} \frac{(a(w))^{n-\frac{|\gamma|+2j}{m}\cdot \frac{s'\rho'-1}{s'-1}}}{\widehat{M}_n}\right)^{(s'-1)/s'},\nonumber
\end{align}
by H\"older's inequality with $p=s'>1$ and $q=s'/(s'-1)>1$. Denote by $g_1(w)$ and $g_2(w)$ the two functions comprising the above product. Notice that
\beqs
g_2(w)\leq (a(w))^{-\frac{|\gamma|+2j}{m}\cdot\frac{s'\rho'-1}{s'}}(P(a(w)))^{(s'-1)/s'}\leq C^{|\gamma|+1}_1\langle w\rangle^{-(|\gamma|+2)\rho}(P(a(w)))^{(s'-1)/s'}
\eeqs
on $Q^c_{B'}$; the very last inequality follows from the facts: $j\geq1$, (\ref{ktsdr15}) and the way we chose $s'$ and $\rho'$. In order to estimate $g_1$, fix $w\in Q^c_{B'}$, $j\geq 1$ and $\gamma\in\NN^{2d}$. There exists $k\in\ZZ_+$ such that $k^{s'}\leq (a(w))^{1/m}<(k+1)^{s'}$ (cf. (\ref{ktsdr15})) and there exists $l\in\NN$ such that $lm\leq |\gamma|+2j\leq (l+1)m-1$. We consider three cases.\\
\underline{Case 1}: $k+1\leq l$. We have
\beqs
(g_1(w))^{s'}=\frac{l^{(|\gamma|+2j)s'}(a(w))^{l-\frac{|\gamma|+2j}{m}}}{\widehat{M}_l}+ S_1+S_2,
\eeqs
where
\beqs
S_1=\sum_{n=l+1}^{m(l+1)} \frac{n^{(|\gamma|+2j)s'}(a(w))^{n-\frac{|\gamma|+2j}{m}}}{\widehat{M}_n},\,\, S_2=\sum_{n=m(l+1)+1}^{\infty} \frac{n^{(|\gamma|+2j)s'}(a(w))^{n-\frac{|\gamma|+2j}{m}}} {\widehat{M}_n}.
\eeqs
Notice that
\beqs
S_1= \sum_{n=l+1}^{m(l+1)} \frac{n^{(|\gamma|+2j)s'}(a(w))^{\frac{(l+1)m-|\gamma|-2j}{m}} (a(w))^{n-l-1}}{\widehat{M}_n}\leq\sum_{n=l+1}^{m(l+1)} \frac{n^{(l+1)ms'} (a(w))^{n-l-1}}{\widehat{M}_n}
\eeqs
and similarly
\beqs
S_2\leq \sum_{n=m(l+1)+1}^{\infty} \frac{n^{(l+1)ms'} (a(w))^{n-l-1}}{\widehat{M}_n}.
\eeqs
To estimate $S_1$, notice that $n^{(l+1)ms'}\leq \prod_{t=0}^{(l+1)m-1}(nm+n-t)^s$ (because $nm+n-t\geq nm+n-(l+1)m\geq n$). Hence,
\beqs
n^{(l+1)ms'}\leq (nm+n)!^s/((n-l-1)m+n)!^s\leq 2^{n(m+1)s}(nm)!^s/((n-l-1)m)!^s.
\eeqs
When $n$ ranges between $l+1$ and $m(l+1)$, we infer
\beqs
2^{n(m+1)s}\leq 2^{m(m+1)s}2^{m(m+1)ls}\leq 2^{m(m+1)s}2^{s(m+1)(|\gamma|+2j)}.
\eeqs
Applying (\ref{stk997733}), we conclude
\beqs
S_1\leq C_0C_0^l2^{m(m+1)s}2^{s(m+1)(|\gamma|+2j)}P(a(w))\leq C_02^{m(m+1)s}(C_02^{s(m+1)})^{(|\gamma|+2j)}P(a(w)).
\eeqs
To estimate $S_2$, notice that when $n\geq m(l+1)+1$ we have
\beqs
n^{(l+1)ms'}\leq \prod_{t=0}^{(l+1)m-1}(nm-t)^s=\frac{(nm)!^s}{((n-l-1)m)!^s}
\eeqs
(since $nm-t\geq nm-(l+1)m\geq n$). Thus, applying (\ref{stk997733}), we conclude $S_2\leq C_0C_0^{|\gamma|+2j}P(a(w))$. Finally, because of (\ref{ktsdr15}), we have $(a(w))^{l-\frac{|\gamma|+2j}{m}}\leq 1$. Hence,
\beqs
\frac{l^{(|\gamma|+2j)s'}(a(w))^{l-\frac{|\gamma|+2j}{m}}}{\widehat{M}_l}&\leq& \frac{C_0^ll^{(l+1)ms'}}{(lm)!^s}\leq \frac{C_0^le^{lms'}l!^{ms'}l^{ms'}}{l!^{ms}}\leq \frac{(C_0e)^{2lms'}}{l!^{m(s-s')}}\\
&=&\left(\frac{(C_0e)^{2ls'/(s-s')}}{l!}\right)^{m(s-s')}.
\eeqs
Obviously, the last term can be estimated by a constant $C'_1$ which is independent of $l$. As $P(a(w))\geq 1$ on $Q^c_{B'}$, we conclude
\beq\label{kth975137}
g_1(w)\leq C'_2(C_0^{1/s'}2^{(m+1)s/s'})^{|\gamma|+2j}P(a(w))^{1/s'},
\eeq
where the constant $C'_2>0$ is independent of $w\in Q^c_{B'}$, $j\geq 1$ and $\gamma\in\NN^{2d}$.\\
\underline{Case 2}: $l\leq k\leq m(l+1)-1$. We infer
\beqs
(g_1(w))^{s'}&=&\sum_{n=l}^k \frac{n^{(|\gamma|+2j)s'}(a(w))^{n-\frac{|\gamma|+2j}{m}}}{\widehat{M}_n} +\sum_{n=k+1}^{m(l+1)} \frac{n^{(|\gamma|+2j)s'}(a(w))^{n-\frac{|\gamma|+2j}{m}}} {\widehat{M}_n}\\
&{}&+\sum_{n=m(l+1)+1}^{\infty} \frac{n^{(|\gamma|+2j)s'}(a(w))^{n-\frac{|\gamma|+2j}{m}}} {\widehat{M}_n}.
\eeqs
The second and the third sum can be estimated by
$$
C_02^{m(m+1)s}(C_02^{s(m+1)})^{(|\gamma|+2j)}P(a(w))\quad  \mbox{and} \quad C_0C_0^{|\gamma|+2j}P(a(w)),
$$
respectively, by the same technique as for the estimates of the sums $S_1$ and $S_2$ from the first case. To estimate the first sum, notice that
\beqs
\sum_{n=l}^k \frac{n^{(|\gamma|+2j)s'}(a(w))^{n-\frac{|\gamma|+2j}{m}}}{\widehat{M}_n}\leq\sum_{n=0}^k \frac{(a(w))^n}{\widehat{M}_n} \leq P(a(w)).
\eeqs
Thus, we deduce (\ref{kth975137}) possibly with another $C'_2$ independent of $w\in Q^c_{B'}$, $j\geq 1$ and $\gamma\in\NN^{2d}$.\\
\underline{Case 3}: $m(l+1)\leq k$. We infer
\beqs
(g_1(w))^{s'}=\sum_{n=l}^k \frac{n^{(|\gamma|+2j)s'}(a(w))^{n-\frac{|\gamma|+2j}{m}}}{\widehat{M}_n}+ \sum_{n=k+1}^{\infty} \frac{n^{(|\gamma|+2j)s'}(a(w))^{n-\frac{|\gamma|+2j}{m}}}{\widehat{M}_n}.
\eeqs
The first sum can be estimated by $P(a(w))$ in the same way as for the first sum in the second case. The second sum can be estimated by $C_0C_0^{|\gamma|+2j}P(a(w))$ in the same way as for the sum $S_2$ from the first case. Hence, we obtain (\ref{kth975137}).\\
\indent Combining (\ref{kth975137}) with the estimate for $g_2$ and (\ref{jtv97}) and (\ref{trk99}), we conclude that there exists $C\geq 1$ such that
\beq\label{ttvv199}
\sum_{n=2}^{\infty}\frac{|D^{\gamma}_w(a^{(\# n)}(w)-(a(w))^n)|}{\widehat{M}_n}\leq \frac{C^{|\gamma|+1}P(a(w))}{\langle w\rangle^{\rho(|\gamma|+2)}},\,\, w\in Q^c_{B'},\, \gamma\in\NN^{2d},
\eeq
which proves (\ref{ttvv19911}). On the other hand, for $\gamma\in\NN^{2d}\backslash\{0\}$, we have
\beqs
|D^{\gamma}_w P(a(w))|\leq \sum_{n=1}^{\infty}\frac{|D^{\gamma}(a(w))^n|}{\widehat{M}_n}\leq \sum_{n=1}^{\infty}\frac{1}{\widehat{M}_n}\sum_{\gamma^1+\ldots+\gamma^n=\gamma} \frac{\gamma!}{\gamma^1!\ldots\gamma^n!}\prod_{l=1}^n \left|D^{\gamma^l}_wa(w)\right|.
\eeqs
Analogously as in the estimate for $|D^{\gamma}_wa^{(\# n)}(w)|$, we can deduce that the terms where $|\gamma|> nm$ are identically equal to zero. Thus, for $w\in Q^c_{B'}$, employing (\ref{est111elp}), we infer
\beqs
|D^{\gamma}_w P(a(w))|\leq\sum_{n=[|\gamma|/m]}^{\infty} \frac{n^{|\gamma|}(a(w))^{n-\frac{|\gamma|\rho'}{m}}} {\widehat{M}_n}.
\eeqs
Now, employing the same technique as for the estimation of (\ref{trk99}), we can conclude the existence of $C\geq1$ such that $|D^{\gamma}_wP(a(w))|\leq C^{|\gamma|+1} P(a(w))\langle w\rangle^{-\rho|\gamma|}$, for all $w\in Q^c_{B'}$, $\gamma\in\NN^{2d}$ (the estimate trivially holds when $\gamma=0$). This estimate together with (\ref{ttvv19911}) proves the rest of the claims in $(iii)$.\\
\indent It remains to prove $(iv)$. We recall that $a'$ satisfies (\ref{obk771553}) and (\ref{ktsdr15}). Let $s'$ and $\rho'$ be as above. Since $a'$ is elliptic and $a''=a-a'$ is a polynomial of degree at most $m-1$, there is $B'=B'(\rho')\geq B_1$ such that
\beq\label{kkttr1755}
|D^{\gamma}_wa'(w)|\leq (a'(w))^{1-\frac{|\gamma|\rho'}{m}},\,\,\, |D^{\gamma}_w a''(w)|\leq (a'(w))^{1-\frac{(|\gamma|+1)\rho'}{m}},
\eeq
when $w\in Q^c_{B'},\, \gamma\in\NN^{2d}$. For $n\in\ZZ_+$ and $\gamma\in\NN^{2d}$, we infer
$$
\left|D^{\gamma}\left((a(w))^n-(a'(w))^n\right)\right|
\leq\sum_{j=1}^n{n\choose j}\sum_{\gamma'+\gamma''=\gamma}\frac{\gamma!}{\gamma'!\gamma''!} \left|D^{\gamma'}\left((a'(w))^{n-j}\right)\right| \left|D^{\gamma''}\left((a''(w))^j\right)\right|.
$$
As $(a'(w))^{n-j}$ is a polynomial of degree $nm-jm$ and $(a''(w))^j$ is a polynomial of degree at most $jm-j$, employing similar technique as above, we can conclude that the terms where $j> nm-|\gamma|$ are identically equal to zero. Thus, as ${n\choose j}\leq n^j/j!$, (\ref{kkttr1755}) yields the following estimate on $Q^c_{B'}$
\begin{align*}
\sum_{n=1}^{\infty} & \frac{\left|D^{\gamma}\left((a(w))^n-(a'(w))^n\right)\right|} {\widehat{M}_n}\\
&\leq \sum_{n=1}^{\infty}\sum_{j=1}^{\min\{nm-|\gamma|,n\}}\frac{n^j}{j!\widehat{M}_n} \sum_{\gamma'+\gamma''=\gamma} \frac{\gamma!}{\gamma'!\gamma''!}(n-j)^{|\gamma'|}j^{|\gamma''|} (a'(w))^{n-\frac{(|\gamma|+j)\rho'}{m}}\\
&\leq\sum_{n=1}^{\infty}\sum_{j=1}^{nm-|\gamma|} \frac{n^{|\gamma|+j}(a'(w))^{n-\frac{(|\gamma|+j)\rho'}{m}}}{j!\widehat{M}_n}\leq \sum_{j=1}^{\infty}\frac{1}{j!}\sum_{n=[(|\gamma|+j)/m]}^{\infty} \frac{n^{|\gamma|+j}(a'(w))^{n-\frac{(|\gamma|+j)\rho'}{m}}}{\widehat{M}_n}.
\end{align*}
Repeating the same argument as for the estimation of (\ref{trk99}), we conclude that there exists $C>1$ such that
\beqs
\left|D^{\gamma}\left(P(a(w))-P(a'(w))\right)\right|\leq \sum_{n=1}^{\infty}\frac{\left|D^{\gamma}\left((a(w))^n-(a'(w))^n\right)\right|} {\widehat{M}_n}\leq \frac{C^{|\gamma|+1}P(a'(w))}{\langle w\rangle^{\rho(|\gamma|+1)}},
\eeqs
for all $w\in Q^c_{B'}$, $\gamma\in\NN^{2d}$. All claims in $(iv)$ are immediate consequences of this estimate together with $(iii)$.
\end{proof}

The significance of Theorem \ref{serpolsym} $(ii)$--$(iii)$ lies in the following result.

\begin{corollary}\label{harmosers}
Under the hypotheses of Theorem \ref{serpolsym}, the series $\sum_{n=0}^{\infty} (a^w)^n/\widehat{M}_n$ converges absolutely in $\mathcal{L}_b(\SSS^*(\RR^d),\SSS^*(\RR^d))$ and in $\mathcal{L}_b(\SSS'^*(\RR^d),\SSS'^*(\RR^d))$. Moreover, $P(a^w)=\sum_{n=0}^{\infty} (a^w)^n/\widehat{M}_n$ is a $\Gamma^{*,\infty}_{A_p,\rho}$-hypoelliptic pseudo-differential operator with real-valued Weyl symbol $\sum_{n=0}^{\infty} a^{(\# n)}/\widehat{M}_n\in \Gamma^{*,\infty}_{A_p,\rho}(\RR^{2d})$, that is,
$$
P(a^w)=\left(\sum_{n=0}^{\infty}\frac{ a^{(\# n)}}{\widehat{M}_n}\right)^w.
$$
Furthermore, the symbol $\sum_{n=0}^{\infty} a^{(\#n)}/\widehat{M}_n$ satisfies the conclusions of Theorem \ref{serpolsym}.
\end{corollary}

\begin{proof} By Theorem \ref{serpolsym} $(ii)$, the series $\sum_{n=0}^{\infty} a^{(\# n)}/\widehat{M}_n$ absolutely converges in $\Gamma^{*,\infty}_{A_p,\rho}(\RR^{2d})$ and it is real-valued by Lemma \ref{ktn557713}. Lemma \ref{ktn557713} together with \cite[Proposition 3.1]{PP1} proves that $\sum_{n=0}^{\infty} (a^w)^n/\widehat{M}_n$ converges absolutely in both $\mathcal{L}_b(\SSS^*(\RR^d),\SSS^*(\RR^d))$ and $\mathcal{L}_b(\SSS'^*(\RR^d),\SSS'^*(\RR^d))$ and $\sum_{n=0}^{\infty} (a^w)^n/\widehat{M}_n=\left(\sum_{n=0}^{\infty} a^{(\# n)}/\widehat{M}_n\right)^w$. The rest follows from Theorem \ref{serpolsym}.
\end{proof}

We also have the following easy corollary of the proofs of Proposition \ref{lemfor135} and Theorem \ref{serpolsym} concerning finite order $\Psi$DOs.

\begin{corollary}
Let $a$ and $a^{(\# n)}$, $n\in\ZZ_+$, be as in Theorem \ref{serpolsym} with $m\in\ZZ_+$, $m\geq 2$, being the degree of the polynomial $a$. For each fixed $n\in\ZZ_+$, there exist $C',B'>0$ such that
\begin{gather}
\langle w\rangle^{nm}/C'\leq \left|a^{(\# n)}(w)\right|\leq C'\langle w\rangle^{nm},\,\, \forall w\in Q^c_{B'}\label{est1forfintorder}\\
\left|D^{\gamma}\left(a^{(\# n)}(w)-(a(w))^n\right)\right|\leq C'^{|\gamma|+1}\langle w\rangle^{nm-|\gamma|-2},\,\, \forall w\in\RR^{2d},\, \forall\gamma\in\NN^{2d}.\label{est2forfintorder}
\end{gather}
In particular, $a^{(\# n)}$ is hypoelliptic symbol in $\Gamma^{*,\infty}_{A_p,1}(\RR^{2d})$ and $(a^{(\# n)})^w=(a^w)^n$. Consequently, if $b\in \Gamma^{*,\infty}_{A_p,\rho}(\RR^{2d})$ satisfies $b\precsim f$ for some positive continuous function $f$ on $\RR^{2d}$ such that $f(w)/\langle w\rangle^{nm}\rightarrow 0$, as $|w|\rightarrow \infty$, then $a^{(\# n)}+b$ is a hypoelliptic symbol in $\Gamma^{*,\infty}_{A_p,\rho}(\RR^{2d})$ and $(a^{(\# n)}+b)^w=(a^w)^n+b^w$.
\end{corollary}

\begin{proof} Notice that (\ref{est1forfintorder}) follows from (\ref{est2forfintorder}). To prove (\ref{est2forfintorder}), write $D^{\gamma}(a^{(\# n)}(w)-(a(w))^n)$ as in (\ref{terforthediff}) and estimate the latter as in the proof of Proposition \ref{lemfor135} $(ii)$.
\end{proof}

In the rest of the section we obtain spectral asymptotics for infinite order pseudo-differential operators related to the symbols considered in Theorem \ref{serpolsym}.

\begin{theorem}
\label{thExWeyl} Let $a$ be
real-valued elliptic Shubin polynomial of degree $m\geq 2$ with principal symbol $a'(w)=\sum_{|\gamma|=m}c_{\gamma}w^{\gamma}>0$, $w\in \mathbb{R}^{d}\setminus{\{0\}}$. Let $s,\rho$ and the sequence $\widehat{M}_{n}$ be as in Theorem \ref{serpolsym}. Furthermore, let $P$ be given by \eqref{especial ultrapolydef} and let the symbol $b\in\Gamma^{*,\infty}_{A_p,\rho}(\RR^{2d})$ be real-valued and satisfy: for every $h>0$ there exists $C>0$ (resp. there exist $h,C>0$) such that
\beq\label{eqlowerorderterm}
|D^{\alpha}_w b(w)|\leq Ch^{|\alpha|}A_{\alpha}\frac{P(a'(w))}{\langle w\rangle^{\rho(|\alpha|+1)}}, \quad \forall w\in\RR^{2d},\, \forall \alpha\in\NN^{2d}.
\eeq
Then,
$$
A_{1}= P(a^{w}) + b^{w} \mbox{ and } A_2= \left( P \circ a \right)^{w}+b^{w}
$$
are $\Gamma^{*,\infty}_{A_p,\rho}$-hypoelliptic pseudo-differential operators with spectral asymptotics
\beq
\label{weylextheq1}
N_{i}(\lambda)\sim c\cdot (P^{-1}(\lambda))^{\frac{2d}{m}}
\quad \mbox{and} \quad
\lambda^{(i)}_{j}=P\left(( j/c)^{\frac{m}{2d}}(1+o(1))\right)
\eeq
where
\beq
\label{eqcstWeylasymptotics}
c=\frac{\pi}{(2\pi)^{d+1}d}\int_{\mathbb{S}^{2d-1}} \frac{d\vartheta}{(a'(\vartheta))^{\frac{2d}{m}}}\:,
\eeq
 $N_i$ is the spectral counting function of $A_i$ and $\{\lambda^{(i)}_{j}\}_{j\in\mathbb{N}}$ its sequence of eigenvalues, $i=1,2$. Furthermore, if in addition the sequence $\widehat{M}_{n}$ is log-convex, i.e. the condition $(M.1)$ holds for it, we also have
 \beq
\label{weylextheq2}
N_{i}(\lambda)\sim c\cdot (\widehat{M}^{-1}(\ln \lambda))^{\frac{2d}{m}}
\quad \mbox{and} \quad
\lambda^{(i)}_{j}=e^{\widehat{M}\left(( j/c)^{\frac{m}{2d}}(1+o(1))\right)}
\eeq
with $\widehat{M}(y)=\sup  _{n\in\NN}\ln_+ y^{n}/\widehat{M}_{n}$, $y>0$, the associated function of the sequence $\widehat{M}_{n}$.
\end{theorem}
\begin{proof}
The $\Gamma^{*,\infty}_{A_p,\rho}$-hypoellipticity follows at once from the estimates obtained in Theorem \ref{serpolsym} and the assumption (\ref{eqlowerorderterm}) on the ``lower order'' perturbation $b$. Write $b_i$ for the symbol of $A_i$, $i=1,2$. It is obvious that for each $0<\varepsilon<1$ one has bounds of the form
$$
c_{\varepsilon} P((1-\varepsilon)r^{m}a'(\vartheta))\leq b_{i}(r\vartheta)\leq C_{\varepsilon} P((1+\varepsilon)r^{m}a'(\vartheta)),  \quad r>B_{\varepsilon}, \ \vartheta\in \mathbb{S}^{2d-1}, \ i=1,2.
$$
Thus, both symbols $b_1$ and $b_2$ satisfy the lower and upper bounds (\ref{weyleq3}) with $f(y)=P(y^{m})$ and $\Phi(\vartheta)=(a'(\vartheta))^{1/m}$. Theorem \ref{Weylth1} then yields (\ref{weylextheq1}) if we verify that this $f$ satisfies (\ref{weyleq2}).  Let $k\in\ZZ_+$ be arbitrary but fixed. We have
\begin{align*}
\frac{y f'(y)}{f(y)}&\geq \frac{1}{f(y)}\sum_{n=k}^{\infty}\frac{mny^{mn}}{\widehat{M}_n}\geq mk-\frac{mk}{f(y)}\sum_{n=0}^{k-1}\frac{y^{mn}}{\widehat{M}_n}=mk+O_{k}\left(\frac{y^{mk}}{f(y)}\right)
\\
&=mk+o_{k}(1),
\end{align*}
since
$$
y^{-mk}f(y)\geq y^{-mk} y^{m(k+1)}/\widehat{M}_{k+1}=y^m/\widehat{M}_{k+1}\to\infty,\,\, \mbox{as}\,\, y\rightarrow \infty.
$$
Thus we have $\ds\liminf_{y\rightarrow\infty}y f'(y)/f(y)\geq mk$. As $k$ was arbitrary, we conclude that (\ref{weyleq2}) holds and hence (\ref{weylextheq1}) has been established.

On the other hand, given an arbitrary $0<\varepsilon<1$, we have bounds
$$
e^{\widehat{M}(y)}\leq P(y)\leq \frac{1+\varepsilon}{\varepsilon} \: e^{\widehat{M}((1+\varepsilon) y)}, \quad y>1.
$$
So, the estimates (\ref{weyleq3}) hold with $f(y)=e^{\widehat{M}(y^{m})}$ and $\Phi(\vartheta)=(a'(\vartheta))^{1/m}$ as well. Under the assumption $(M.1)$ for the sequence $\widehat{M}_n$, we have the representation \cite[p. 50]{Komatsu1}
$$
\widehat{M}(y^{m})=\int_{0}^{y^{m}}\frac{\widehat{m}(t)}{t}dt, \quad y\geq 1,
$$
where $\widehat{m}$ is the counting function of the sequence $\widehat{m}_{p}=\widehat{M}_{p}/\widehat{M}_{p-1}$, $p\in\mathbb{Z}_{+}$. Hence, $f(y)=e^{\widehat{M}(y^{m})}$ also satisfies
$$
\frac{y f'(y)}{f(y)}= m \cdot \widehat{m}(y^{m})\to\infty, \quad \mbox{as}\quad y\to\infty.
$$
The asymptotic formulae (\ref{weylextheq2}) follow once again from Theorem \ref{Weylth1}, which completes the proof of the theorem.
\end{proof}

It turns out that when the closure of the differential operator $a^w$ has only non-negative eigenvalues and $b=0$, the spectrum of the operator $\overline{A}_1$ from Theorem \ref{thExWeyl} has a very simple structure:

\begin{corollary}\label{cor easy case eigenvalues} Under the hypotheses of Theorem \ref{serpolsym} and additionally assuming that the eigenvalues of the closure of the differential operator $a^w$ are non-negative, the spectrum of the maximal (equivalently, minimal) realisation in $L^{2}(\RR^d)$ of the pseudo-differential operator $P(a^{w})$ is given by its eigenvalues, which are given by the sequence $\{P(\mu_j)\}_{j\in\mathbb{N}}$ where $\{\mu_j\}_{j\in\mathbb{N}}$ is the sequence of eigenvalues of $a^w$ (taking multiplicities into account in both cases). Furthermore, for each $j\in\mathbb{N}$, the eigenspace of $P(a^{w})$ that corresponds to $P(\mu_j)$ coincides with the eigenspace that corresponds to $\mu_{j}$.
 \end{corollary}
 \begin{proof}
Denote as $A$ the unbounded operator on $L^2(\RR^d)$ given by $P(a^{w})$. Its maximal realisation is $\overline{A}$. As we have repeatedly used through this section, we know that  the spectrum of $\overline{A}$ consists of a sequence of real isolated eigenvalues diverging to $+\infty$, each of them with finite multiplicity and eigenfunctions belonging to $\SSS^*(\RR^d)$ (cf. Corollary \ref{harmosers} and Proposition \ref{discretness_of_spe}). The same applies for the closure of $a^{w}$. Find an orthonormal basis $\{u_{j}|\: j\in\mathbb{N}\}$ of $L^{2}(\mathbb{R}^{2d})$ such that for each $j\in\mathbb{N}$ the function $u_{j}\in \mathcal{S}^{\ast}(\mathbb{R}^{d})$ is an eigenfunction corresponding to $\mu_{j}$. Let $\lambda$ be an eigenvalue of $\overline{A}$ and $0\neq\varphi\in\SSS^*(\RR^d)$ an eigenfunction that corresponds to $\lambda$. Clearly, $\overline{A}u_{j}=P(a^{w})u_{j}=P(\mu_{j})u_{j}$. Denoting $c_{j}=(\varphi,u_{j})$, we have
\beqs
\sum_{j=0}^{\infty}\lambda c_{j}u_{j}=\lambda\varphi =P(a^{w})\varphi=\sum_{j=0}^{\infty}c_{j}P(a^{w})u_{j}=\sum_{j=0}^{\infty} c_{j}P(\mu_{j})u_{j},
\eeqs
where the last series absolutely converges in the space $\SSS^*(\RR^d)$ since we have $P(a^w)\in \mathcal{L}_b(\mathcal \SSS^*(\RR^d),\mathcal \SSS^*(\RR^d))$ in view of Corollary \ref{harmosers}. Note that the series $\sum_j c_j u_j$ absolutely converges to $\varphi$ in $\SSS^*(\RR^d)$ because of \cite[Theorem 4.1]{djvin}. Thus,
\beq\label{kkrtt99911}
\lambda c_{j}=c_{j}P(\mu_{j}),\,\, j \in\NN.
\eeq
As $\varphi\neq0$, there exists $j\in\NN$ such that $c_j\neq0$ and hence $\lambda=P(\mu_{j})$, as claimed. If $k\in\NN$ is such that $\mu_{k}\neq\mu_{j}$, then (\ref{kkrtt99911}) implies $c_{k}=0$ (as $\mu_p$, $p\in\NN$, are non-negative and $P$ is strictly increasing on $[0,\infty)$). Thus, $\varphi$ belongs to the eigenspace of the closure of $a^w$ which corresponds to $\mu_j$. This automatically implies that the eigenspace of $\overline{A}$ which corresponds to $P(\mu_{j})$ is a subspace of the eigenspace which corresponds to $\mu_j$. On the other hand, each $u_{j}$ is an eigenfunction of $\overline{A}$ with eigenvalue $P(\mu_{j})$. The proof of the corollary is complete.
\end{proof}

The prototypical example for all results in this section is the polynomial symbol $a(x,\xi)=|x|^2+|\xi|^2$. Then, $a^w$ is the harmonic oscillator $H=|x|^2-\Delta$. Hence
$$
P(H)=\operatorname*{Id}+\sum_{n=1}^{\infty}H^n/\widehat{M}_n=\left(\sum_{n=0}^{\infty}a^{(\# n)}/\widehat{M}_n\right)^w
$$
is hypoelliptic, where $a^{(\# n)}$ are defined by (\ref{kktlt17}); in this case the symbol $P\circ a$ also considered in Theorem \ref{serpolsym} is $P(|w|^{2})=1+\sum_{n=1}^{\infty}|w|^{2n}/\widehat{M}_n$. Furthermore, by Theorem \ref{thExWeyl}, we have for the operators $A_1=P(H)+b^{w}$ and $A_2=(P\circ a)^{w}+b^w$
$$
N_{i}(\lambda)\sim \frac{1}{2^d d!}\: (P^{-1}(\lambda))^{d}
\quad \mbox{and} \quad
\lambda^{(i)}_{j}=P\left(( jd!)^{\frac{1}{d}}2(1+o(1))\right), \quad i=1,2,
$$
where $\{\lambda^{(i)}_{j}\}_{j\in\mathbb{N}}$ and $N_{i}$ are their sequences of eigenvalues and their spectral counting functions, and $b$ satisfies (\ref{eqlowerorderterm}) with $a'(w)=|w|^{2}$.

The next example treats an instance of a function $P$ for which asymptotic formulae (\ref{weylextheq2}) can even be made more explicit.

\begin{example}
\label{Weylex4} Let the symbols $a,a'$ and the parameters $s,\rho$ be as in Theorem \ref{thExWeyl}. We consider here
$$
P(\lambda)=\sum_{n=0}^{\infty} \frac{(h\lambda)^{n}}{n^{snm}}, \quad \lambda\in\mathbb{R},
$$
where the parameter $h>0$. The sequence $\widehat{M}_{n}= h^{-n}n^{snm}$ clearly satisfies (\ref{stk997733}) and $(M.1)$, so that Theorem \ref{thExWeyl} applies to conclude
(\ref{weylextheq2}) for the $\Gamma^{*,\infty}_{A_p,\rho}$-hypoelliptic pseudo-differential operators $A_1$ and $A_2$ under consideration. Observe that
\beqs
e^{-s} \exp\left(\frac{sy^{1/s}}{e}\right) \leq \sup_{p\in \mathbb{Z}_{+}} \frac{y^{p}}{p^{sp}}\leq e^s \exp\left(\frac{sy^{1/s}}{e}\right), \quad y\geq e^{s},
\eeqs
because the only critical point of $g(t)=t\ln y - st \ln t$ lies at $t= e^{-1}y^{1/s}$. Thus
$$
e^{-sm} \exp\left(\frac{sm\: (hy)^{\frac{1}{sm}}}{e}\right) \leq \exp\left(\widehat{M}(y)\right)\leq e^{sm} \exp\left(\frac{sm\: (hy)^{\frac{1}{sm}}}{e}\right), \quad y\geq \frac{e^{sm}}{h},
$$
whence
$$ \widehat{M}^{-1}(\ln \lambda)\sim \frac{1}{h} \left(\frac{e\ln \lambda}{sm}\right)^{sm}, \quad \lambda\to\infty.
$$
Combining these two facts with (\ref{weylextheq2}), we deduce that
$$
N_{i}(\lambda)\sim \frac{ e^{2ds} c}{h^{2d/m}(sm)^{2ds}}\:(\ln \lambda)^{2ds}
\quad \mbox{and} \quad
\lambda^{(i)}_{j}=\exp\left( \frac{smh^{\frac{1}{sm}}}{e} \left( \frac{j}{c}\right)^{\frac{1}{2ds}}(1+o(1))\right),
$$
with $c$ given by (\ref{eqcstWeylasymptotics}).

In the special case of the symbol $a(w)=|w|^{2}$ of the harmonic oscillator, the constant is $c=(2^{d}d!)^{-1}$ and we obtain
$$
N_{i}(\lambda)\sim \frac{ e^{2ds}}{h^{d}s^{2ds} 2^{d(2s+1)} d!}\:(\ln \lambda)^{2ds}\,\,\,\,
\mbox{and}\,\,\,\,
\lambda^{(i)}_{j}= \exp\left(\frac{2^{\frac{2s+1}{2s}}sh^{\frac{1}{2s}}(d!)^{\frac{1}{2ds}}}{e} \: j^{\frac{1}{2ds}}(1+o(1))\right).
$$
\end{example}

We end this section with a remark.

\begin{remark} Theorem \ref{thExWeyl} gives the spectral asymptotics of the pseudo-differential operator $A=P(a^w)$ from Corollary \ref{cor easy case eigenvalues}, but alternatively they can also be obtained from results for operators of finite order
(which is of course not the case for the pseudo-differential operators $A_1$ and $A_2$ from Theorem \ref{thExWeyl}, whose analysis requires the use of Theorem \ref{Weylth1} in an essential way). In fact, retaining the notation and assumptions from Corollary \ref{cor easy case eigenvalues}, and writing $N_{P(a^{w})}$ and $N_{a^w}$ for the spectral counting functions of $P(a^w)$ and $a^w$ and $\{\lambda_j\}_{j\in\mathbb{N}}$ and $\{\mu_j\}_{j\in\mathbb{N}}$ for their sequences of eigenvalues, respectively,  we have that $N_{P(a^w)}(\lambda)= N_{a^w}(P^{-1}(\lambda))$. The well-known facts $N_{a^w}(\lambda)\sim c \lambda^{2d/m}$ and $\mu_{j}\sim (j/c)^{m/2d}$  (which also follow from \cite[Theorem 5.2]{PPV-BMS}) and Corollary \ref{cor easy case eigenvalues} then yield directly $N_{P(a^w)}(\lambda)\sim c\cdot (P^{-1}(\lambda))^{2d/m}$ and $\lambda_j=P(\mu_j)= P((j/c)^{m/2d}(1+o(1)))$.
\end{remark}

\section{Ellipticity. Shubin-Sobolev space of infinite order}
\label{Elliptic Shubin}
\subsection{Elipticity}
We start this section by defining a useful notion of ellipticity for our symbol classes.
\begin{definition}
A positive continuous function $f$ on $\RR^{2d}$ will be called $\Gamma_{A_p,\rho}^{*,\infty}$-admissible if both $f$ and $1/f$ are of ultrapolynomial growth of class $*$ and there exists a symbol $a\in\Gamma_{A_p,\rho}^{*,\infty}(\RR^{2d})$ which is hypoelliptic and satisfies: there exist $c,C,B>0$ such that
\beq\label{krttti113}
cf(w)\leq |a(w)|\leq Cf(w),\,\,\, \mbox{for all}\,\, w\in Q^c_B.
\eeq
In this case, we say that $a$ is an $f-\Gamma_{A_p,\rho}^{*,\infty}$-elliptic symbol.
\end{definition}
When $a$ is $f-\Gamma_{A_p,\rho}^{*,\infty}$-elliptic, if there is no danger of confusion we will often abbreviate terminology and simply call it $f$-elliptic.\\
\indent The next two remarks are related to examples of $\Gamma_{A_p,\rho}^{*,\infty}$-admissible functions.

\begin{remark}\label{ktsvv1573}
Let $f$ be positive, continuous and temperate, i.e. there exist $C,s>0$ such that $f(w+\tilde{w})\leq Cf(w)\langle \tilde{w}\rangle^s$, $\forall w,\tilde{w}\in\RR^{2d}$; notice that this automatically implies that $f$ is bounded from below by $\langle w\rangle^{-s}$ and from above by $\langle w\rangle^s$. Assume that $f$ additionally satisfies the following slow variation condition: there exist $c,C>0$ such that $$|w-\tilde{w}|\leq c\langle \tilde{w}\rangle^{\rho}\Rightarrow  f(\tilde{w})/C\leq f(w)\leq Cf(\tilde{w}).$$
Then $f$ is $\Gamma_{A_p,\rho}^{*,\infty}$-admissible.

The proof that there exists $a\in\Gamma_{A_p,\rho}^{*,\infty}(\RR^{2d})$ that is $f$-elliptic is the same as in the distributional setting (see \cite[Section 1.3.2, p. 38]{NR}; the only difference is that one has to use cut-off functions from $\DD^{(A_p)}(\RR^{2d})$ in the Beurling case and from $\DD^{\{A_p\}}(\RR^{2d})$ in the Roumieu case). It is important to stress that $a$ is of finite order and hence $a^w$ also acts continuously on $\SSS(\RR^d)$ and $\SSS'(\RR^d)$.
\end{remark}

\begin{remark}\label{admfctwit}
Let $1\geq \rho_1>\rho$ and $m\geq 1$, $m\in\RR$. Let $a\in\Gamma^m_{\rho_1}(\RR^{2d})$ be elliptic in $\Gamma^m_{\rho_1}$-sense and positive, i.e. $c_1\langle w\rangle^m\leq a(w)\leq c_2\langle w\rangle^m$ on $\RR^{2d}$. Assume that $a$ satisfies the following condition: for every $h>0$ there exists $C>0$ (resp. there exist $h,C>0$) such that
\beqs
|D^{\alpha}a(w)|\leq Ch^{|\alpha|}A_{\alpha}a(w)\langle w\rangle^{-\rho_1|\alpha|},\,\, w\in \RR^{2d},\, \alpha\in\NN^{2d}.
\eeqs
Let $s\geq 1/(\rho_1-\rho)$ be such that $M_p\subset p!^s$ in the Beurling case and $M_p\prec p!^s$ in the Roumieu case, respectively. Then \cite[Remark 7.6]{PP1} yields that
$$
e^{a(w)^{1/(sm)}} \mbox{ and } e^{-a(w)^{1/(sm)}} \mbox{ are } \Gamma_{A_p,\rho}^{*,\infty}-\mbox{admissible and }
$$
$$e^{a(w)^{1/(sm)}}\mbox{ is } e^{a(w)^{1/(sm)}}-\Gamma_{A_p,\rho}^{*,\infty}-\mbox{elliptic};\,\,\,
e^{-a(w)^{1/(sm)}} \mbox{ is } e^{-a(w)^{1/(sm)}}-\Gamma_{A_p,\rho}^{*,\infty}-\mbox{elliptic}.
$$
In particular, as $|D^{\alpha}\langle w\rangle|\leq 2^{|\alpha|+1}|\alpha|!\langle w\rangle^{1-|\alpha|}$, we conclude that $e^{r\langle w\rangle^{1/s}}$ is $\Gamma_{A_p,\rho}^{*,\infty}$-admissible for each $r\in\RR\backslash\{0\}$ and $s\geq 1/(1-\rho)$ such that $M_p\subset p!^s$ (resp. $M_p\prec p!^s$); furthermore, $e^{r\langle w\rangle^{1/s}}$ is $e^{r\langle w\rangle^{1/s}}-\Gamma_{A_p,\rho}^{*,\infty}$-elliptic.

  Similarly, for such $s$ and $r$, $e^{r|w|^{1/s}}$ is $\Gamma_{A_p,\rho}^{*,\infty}$-admissible. In order to generate an $e^{r|w|^{1/s}}$-elliptic symbol, take $a(w)=r^{2s}|w|^2$ and modify it near the origin to be positive; of course, one has to use cut-off functions from $\DD^{(A_p)}(\RR^{2d})$ in the Beurling case and from $\DD^{\{A_p\}}(\RR^{2d})$ in the Roumieu case, respectively.
\end{remark}

Let us prove some basic properties of $\Gamma_{A_p,\rho}^{*,\infty}$-admissible functions.

\begin{lemma}\label{ktn117951}
If $f$ and $g$ are $\Gamma_{A_p,\rho}^{*,\infty}$-admissible, then so are $1/f$ and $fg$.
\end{lemma}

\begin{proof} If $a$ is $f$-elliptic then its parametrix $q\in\Gamma_{A_p,\rho}^{*,\infty}(\RR^{2d})$ constructed in Remark \ref{ktv957939} is $1/f$-elliptic because of (\ref{ktr991509}) and (\ref{ktl997133}). If $a$ is $f$-elliptic and $b$ is $g$-elliptic then $ab\in\Gamma_{A_p,\rho}^{*,\infty}(\RR^{2d})$ is $fg$-elliptic.
\end{proof}

\begin{lemma}\label{lskvpc135}
If $f$ is $\Gamma_{A_p,\rho}^{*,\infty}$-admissible then so is $f^r$ for any $r\in\RR\backslash\{0\}$.
\end{lemma}

\begin{proof} Because of Lemma \ref{ktn117951}, it is enough to prove the claim for $0<r<1$. Fix such $r$. Let $b\in\Gamma_{A_p,\rho}^{*,\infty}(\RR^{2d})$ be $f$-elliptic. Then $|b|^2$ is $f^2$-elliptic. Modify $|b|^2$ near the origin so to be positive on the whole $\RR^{2d}$ and denote this symbol by $a\in \Gamma_{A_p,\rho}^{*,\infty}(\RR^{2d})$. Then, $a$ is $f^2$-elliptic. For $s>1$, we can apply the second estimate in \cite[Remark 7.6]{PP1} to obtain that for every $h>0$ there exists $C>0$ (resp. there exist $h,C>0$) such that
\beqs
|D^{\alpha}a(w)^{1/s}|\leq C h^{|\alpha|}A_{\alpha}a(w)^{1/s}\langle w\rangle^{-\rho|\alpha|},\, w\in\RR^{2d},\, \alpha\in\NN^{2d}.
\eeqs
Taking $s=2/r>1$, this estimate readily yields the desired conclusion.
\end{proof}

\subsection{Infinite order Shubin-Sobolev spaces and the Fredholm property}\label{fredh}

Let $a$ be $f$-elliptic and $q\in\Gamma_{A_p,\rho}^{*,\infty}(\RR^{2d})$ be the parametrix of $a$ constructed in Remark \ref{ktv957939}. Then $q^wa^w=\mathrm{Id}+T$, where $T$ is $*$-regularising. We now introduce our new class of Shubin-Sobolev spaces.

\begin{definition}\label{spaceswithff}
Let $a$ be $f$-elliptic and $T$ be the $*$-regularising operator that equals $q^wa^w-\mathrm{Id}$, where $q\in\Gamma_{A_p,\rho}^{*,\infty}(\RR^{2d})$ is the parametrix of $a$ defined as above. We define
$$
H^*_{A_p,\rho}(f)=\{u\in\SSS'^*(\RR^d)|\, a^wu\in L^2(\RR^d)\}
$$
endowed with norm $\|u\|_{H^*_{A_p,\rho}(f)}=\|a^wu\|_{L^2}+\|Tu\|_{L^2}$.
\end{definition}

\begin{remark}\label{ktstt1975}
An easy preliminary observation is that one obtains an equivalent norm to the one given in Definition \ref{spaceswithff} if one chooses any other parametrix for $a^w$. To be precise, let $\tilde{q}\in \Gamma_{A_p,\rho}^{*,\infty}(\RR^{2d})$ be a symbol such that $\tilde{T}=\tilde{q}^wa^w-\mathrm{Id}$ is $*$-regularising. Then the norm $\|u\|_{H^*_{A_p,\rho}(f)}$ on $H^*_{A_p,\rho}(f)$ is equivalent to $\|a^wu\|_{L^2}+\|\tilde{T}u\|_{L^2}$. To see this, notice that $\|\tilde{T}u\|_{L^2}\leq \|\tilde{T}q^wa^wu\|_{L^2}+\|\tilde{T}Tu\|_{L^2}$. As $\tilde{T}q^w$ and $\tilde{T}$ are $*$-regularising, hence continuous on $L^2(\RR^d)$, we have $\|\tilde{T}u\|_{L^2}\leq C'(\|a^wu\|_{L^2}+\|Tu\|_{L^2})$. Analogously, $\|Tu\|_{L^2}\leq C''(\|a^wu\|_{L^2}+\|\tilde{T}u\|_{L^2})$.
\end{remark}

\begin{remark}\label{ktkjosk17}
Let $a$, $q$ and $T$ be as in Definition \ref{spaceswithff}. Let $a_1$ be any $f$-elliptic symbol with the corresponding parametrix $q_1^w$ as constructed in Remark \ref{ktv957939} and a $*$-regularising operator $T_1$ such that $q_1^wa_1^w=\mathrm{Id}+T_1$. Then, Theorem \ref{weylq} $(i)$ yields
$$
\{a\# q_1, q_1\# a, a_1\# q, q\# a_1\}\precsim 1\quad \mbox{in}\,\, FS_{A_p,\rho}^{*,\infty}(\RR^{2d};0).
$$
Hence, Theorem \ref{weylq} $(ii)$ proves $a^wq_1^w, q_1^w a^w, a_1^w q^w, q^w a_1^w\in\mathcal{L}(L^2(\RR^d),L^2(\RR^d))$ (cf. \cite[Lemma 5.3]{PP1}). More generally, let $b_j\in\Gamma_{A_p,\rho}^{*,\infty}(\RR^{2d})$, $j=1,\ldots,n$, satisfy the following: for every $h>0$ there exists $C>0$ (resp. there exist $h,C>0$) such that
\beqs
|D^{\alpha}b_j(w)|\leq Ch^{|\alpha|}A_{\alpha}g_j(w)\langle w\rangle^{-\rho|\alpha|},\,\, w\in\RR^{2d},\, \alpha\in\NN^{2d}, j=1,\ldots,n,
\eeqs
where $g_j$ are positive continuous functions on $\RR^{2d}$ of ultrapolynomial growth of class $*$ (not necessarily $\Gamma_{A_p,\rho}^{*,\infty}$-admissible). Theorem \ref{weylq} $(i)$ gives
$$
\{b_{j_1}\#\ldots\#b_{j_l}\}\precsim \prod_{k=1}^l g_{j_k}\quad \mbox{in}\,\, FS_{A_p,\rho}^{*,\infty}(\RR^{2d};0),\,\, \mbox{for}\,\, \{j_1,\ldots,j_l\}\subseteq\{1,\ldots,n\}.
$$
Thus, repeated application of
Theorem \ref{weylq} $(ii)$ shows the existence of $b\in\Gamma_{A_p,\rho}^{*,\infty}(\RR^{2d})$ such that $b_1^w\circ\ldots \circ b_n^w-b^w$ is $*$-regularising and $b\precsim \prod_{j=1}^n g_j$; i.e.: for every $h>0$ there exists $C>0$ (resp. there exist $h,C>0$) such that
\beqs
|D^{\alpha}b(w)|\leq Ch^{|\alpha|}A_{\alpha}g_1(w)\cdot\ldots\cdot g_n(w)\langle w\rangle^{-\rho|\alpha|},\,\, w\in\RR^{2d},\, \alpha\in\NN^{2d}.
\eeqs
\end{remark}
\bigskip

Employing Remarks \ref{ktstt1975} and \ref{ktkjosk17}, one can prove the following properties of $H^*_{A_p,\rho}(f)$ in the same manner as in the distributional setting. We simply outline the arguments.
\begin{itemize}
\item[$(i)$] Let $a_1$ be any other $f$-elliptic symbol with $q_1\in\Gamma_{A_p,\rho}^{*,\infty}(\RR^{2d})$ being its parametrix constructed in Remark \ref{ktv957939} and $T_1=q_1^wa_1^w-\mathrm{Id}$ be the corresponding $*$-regularising operator. Then the space $H^*_{A_p,\rho}(f)$ defined by $a_1$ is the same as the one defined by $a$ and the norm $\|a^w_1u\|_{L^2}+\|T_1u\|_{L^2}$ is equivalent to the one defined by $a^w$ and $T$. The proof relies on Remark \ref{ktkjosk17} and is the same as that of \cite[Proposition 1.5.3 (a), p. 42]{NR}. Employing Remark \ref{ktstt1975}, we conclude that the space $H^*_{A_p,\rho}(f)$ and the topology induced on it by the norm $\|\cdot\|_{H^*_{A_p,\rho}(f)}$ do not depend on the particular choice of $a$, its parametrix $q$ and the resulting $*$-regularising operator $T=q^wa^w-\mathrm{Id}$.\\
    \indent More generally, let $b$ be an $f$-elliptic symbol and $\tau\in\RR$. Let $\tilde{q}_1\in\Gamma_{A_p,\rho}^{*,\infty}(\RR^{2d})$ be the symbol of any left $\tau$-parametrix of $\Op_{\tau}(b)$, i.e. $T_1=\Op_{\tau}(\tilde{q}_1)\Op_{\tau}(b)-\mathrm{Id}$ is $*$-regularising; because of Remark \ref{kth995559} we can use $\Op_{\tau}(\tilde{q}_1)$ as a right parametrix of $\Op_{\tau}(b)$ as well. For the moment, set $\tilde{H}=\{u\in\SSS'^*(\RR^d)|\, \Op_{\tau}(b)u\in L^2(\RR^d)\}$ with norm $\|u\|_{\tilde{H}}=\|\Op_{\tau}(b)u\|_{L^2}+\|T_1u\|_{L^2}$. Then $\tilde{H}$ is topologically isomorphic to $H^*_{A_p,\rho}(f)$. To verify this, take the Weyl symbol $a\in\Gamma_{A_p,\rho}^{*,\infty}(\RR^{2d})$ constructed out of $b$ as in \cite[Proposition 3.5]{PP1}. Then $T_2=\Op_{\tau}(b)-a^w$ is $*$-regularising and $a$ is $f$-elliptic. For this $a$ let $q\in\Gamma_{A_p,\rho}^{*,\infty}(\RR^{2d})$ be the parametrix as constructed in Remark \ref{ktv957939} and denote $T=q^wa^w-\mathrm{Id}\in\mathcal{L}(\SSS'^*(\RR^d),\SSS^*(\RR^d))$; because of Remark \ref{kts951307}, we can use $q^w$ as a right parametrix of $a^w$ as well. Define the space $H^*_{A_p,\rho}(f)$ with $a^w$ and the norm on it by $\|u\|_{H^*_{A_p,\rho}(f)}=\|a^w u\|_{L^2}+\|Tu\|_{L^2}$. Since $T_2=\Op_{\tau}(b)-a^w$, $\tilde{H}$ and $H^*_{A_p,\rho}(f)$ are algebraically isomorphic. To prove that the isomorphism is also topological, notice that $T_2=\Op_{\tau}(b)-a^w$ implies that $\Op_{\tau}(\tilde{q}_1)-q^w$ is $*$-regularising. Now, by similar technique as in the proof of \cite[Proposition 1.5.3 (a), p. 42]{NR} we can conclude that the norms $\|u\|_{H^*_{A_p,\rho}(f)}$ and $\|u\|_{\tilde{H}}$ are equivalent. Thus, $H^*_{A_p,\rho}(f)$ does not depend on the quantisation as either.
\item[$(ii)$] The space $H^*_{A_p,\rho}(f)$ becomes a Hilbert space with the inner product $$(u,v)_{H^*_{A_p,\rho}(f)}=(a^wu,a^wv)_{L^2(\RR^d)}+(Tu,Tv)_{L^2(\RR^d)}$$
which induces an equivalent norm to the one introduced in Definition \ref{spaceswithff}. The proof is the same as that of \cite[Proposition 1.5.3 (b), p. 42]{NR}.
\item[$(iii)$] The space $\SSS^*(\RR^d)$ is continuously and densely injected into $H^*_{A_p,\rho}(f)$ and $H^*_{A_p,\rho}(f)$ is continuously and densely injected into $\SSS'^*(\RR^d)$. The proof is the same as the one of \cite[Proposition 1.5.4, p. 43]{NR}.
\item[$(iv)$] Let $f$ and $g$ be $\Gamma_{A_p,\rho}^{*,\infty}$-admissible and $b\in\Gamma_{A_p,\rho}^{*,\infty}(\RR^{2d})$ be such that for every $h>0$ there exists $C>0$ (resp. there exist $h,C>0$) such that
    \beqs
    |D^{\alpha}b(w)|\leq Ch^{|\alpha|}A_{\alpha}g(w)\langle w\rangle^{-\rho|\alpha|},\,\, w\in\RR^{2d},\, \alpha\in\NN^{2d}.
    \eeqs
    Then, $b^w$ defines a continuous operator from $H^*_{A_p,\rho}(f)$ into $H^*_{A_p,\rho}(f/g)$ ($f/g$ is $\Gamma_{A_p,\rho}^{*,\infty}$-admissible because of Lemma \ref{ktn117951}). If $g_1$ is $\Gamma_{A_p,\rho}^{*,\infty}$-admissible and $g(w)g_1(w)/f(w)\rightarrow0$ as $|w|\rightarrow\infty$, then $b^w$ defines a compact operator from $H^*_{A_p,\rho}(f)$ into $H^*_{A_p,\rho}(g_1)$. The proof relies on Remark \ref{ktkjosk17} and is the same as the proof of \cite[Proposition 1.5.5, p. 43]{NR}. For the compactness, one additionally has to use a result on compactness on $L^2(\RR^d)$, for example \cite[Section 24.4, p. 192]{Shubin}, or to derive one by similar technique as in \cite[Remark 8.7]{PPV-BMS}; cf. Remark \ref{ktv957939}. In particular, if $f(w)\geq cg(w)$, $\forall w\in\RR^{2d}$, then $H^*_{A_p,\rho}(f)$ is continuously injected into $H^*_{A_p,\rho}(g)$ and, if $g(w)/f(w)\rightarrow0$ as $|w|\rightarrow \infty$, then the inclusion $H^*_{A_p,\rho}(f)\rightarrow H^*_{A_p,\rho}(g)$ is compact. To prove this just take $b(w)=1$, $w\in\RR^{2d}$.
\item[$(v)$] (Global regularity and a priori estimates) Let $f,g,g_1$ be $\Gamma_{A_p,\rho}^{*,\infty}$-admissible and $b\in\Gamma_{A_p,\rho}^{*,\infty}(\RR^{2d})$ hypoelliptic. Assume there exist $c,B>0$ such that $cg(w)\leq|b(w)|$, $\forall w\in Q^c_B$. If $u\in\SSS'^*(\RR^d)$ and $b^wu\in H^*_{A_p,\rho}(f/g)$, then $u\in H^*_{A_p,\rho}(f)$. Furthermore, there exists $C_1>0$ such that
    \beqs
    \|\varphi\|_{H^*_{A_p,\rho}(f)}\leq C_1(\|b^w\varphi\|_{H^*_{A_p,\rho}(f/g)}+\|\varphi\|_{H^*_{A_p,\rho}(g_1)}),\,\, \forall\varphi\in\SSS^*(\RR^d).
    \eeqs
    The proof is the same as that of \cite[Proposition 1.5.8, p. 44]{NR}; use the parametrix for $b$ as constructed in Remark \ref{ktv957939}.
\item[$(vi)$] (Duality) Let $f$ be $\Gamma_{A_p,\rho}^{*,\infty}$-admissible and set $\tilde{f}(x,\xi) = f(x,-\xi)$, $(x,\xi)\in\mathbb R^{2d}$. Clearly, $\tilde{f}$ is $\Gamma_{A_p,\rho}^{*,\infty}$-admissible (if $a$ is $f$-elliptic then $a(x,-\xi)$ is $\tilde{f}$-elliptic). The bilinear mapping
    \beqs
    \langle \varphi,\psi\rangle=\int_{\RR^d}\varphi(x)\psi(x)dx,\,\, \SSS^*(\RR^d)\times\SSS^*(\RR^d)\rightarrow\CC,
    \eeqs
    extends to a continuous bilinear mapping $H^*_{A_p,\rho}(f)\times H^*_{A_p,\rho}(1/\tilde{f})\rightarrow\CC$. The strong dual of $H^*_{A_p,\rho}(f)$ is isomorphic to $H^*_{A_p,\rho}(1/\tilde{f})$ and the duality is given by this bilinear mapping. The proof is the same as that of \cite[Proposition 1.5.9, p. 44]{NR}. Clearly, this duality is compatible with $\langle \SSS^*(\RR^d),\SSS'^*(\RR^d)\rangle$. In particular, if $f(x,-\xi)=f(x,\xi)$, then the strong dual of $H^*_{A_p,\rho}(f)$ is isomorphic to $H^*_{A_p,\rho}(1/f)$.
\item[$(vii)$] If $f$ is temperate and slowly varying, then it is $\Gamma_{A_p,\rho}^{*,\infty}$-admissible by Remark \ref{ktsvv1573}. The classical (distributional) Shubin-Sobolev space defined by $f$ is isomorphic with $H^*_{A_p,\rho}(f)$ of Definition \ref{spaceswithff} (cf. Remark \ref{ktsvv1573}). In particular, if $g$ is $\Gamma_{A_p,\rho}^{*,\infty}$-admissible and for every $s>0$ there exists $C_s>0$ such that $\langle w\rangle^s\leq C_sg(w)$, $\forall w\in\RR^{2d}$, then $(iv)$ implies $H^*_{A_p,\rho}(g)\hookrightarrow\SSS(\RR^d)$. Similarly, if for every $s>0$ there exists $C_s>0$ such that $\langle w\rangle^{-s}\geq C_sg(w)$, then $\SSS'(\RR^d)\hookrightarrow H^*_{A_p,\rho}(g)$.
\item[$(viii)$] Let $f$ be $\Gamma_{A_p,\rho}^{*,\infty}$-admissible and $a\in\Gamma_{A_p,\rho}^{*,\infty}(\RR^{2d})$ be $f-\Gamma_{A_p,\rho}^{*,\infty}$-elliptic. Let $\tilde{M}_p$ satisfies $(M.1)$, $(M.2)$, $(M.3)$, $\tilde{M}_0=1$ and $M_p\subset \tilde{M}_p$; clearly, $\SSS^{(M_p)}(\RR^d)\hookrightarrow \SSS^{(\tilde{M}_p)}(\RR^d)$ and $\SSS^{\{M_p\}}(\RR^d)\hookrightarrow \SSS^{\{\tilde{M}_p\}}(\RR^d)$. Since $A_p\subset \tilde{M}_p^{\rho}$, we can also consider the symbol classes $\Gamma_{A_p,\rho}^{(\tilde{M}_p),\infty}(\RR^{2d})$ and $\Gamma_{A_p,\rho}^{\{\tilde{M}_p\},\infty}(\RR^{2d})$; clearly, $\Gamma_{A_p,\rho}^{(\tilde{M}_p),\infty}(\RR^{2d})$ and $\Gamma_{A_p,\rho}^{\{\tilde{M}_p\},\infty}(\RR^{2d})$ are continuously injected into the spaces $\Gamma_{A_p,\rho}^{(M_p),\infty}(\RR^{2d})$ and $\Gamma_{A_p,\rho}^{\{M_p\},\infty}(\RR^{2d})$ respectively. If $f$ and $1/f$ are of ultrapolynomial growth of class $(\tilde{M}_p)$ (resp. of class $\{\tilde{M}_p\}$), then $a\in \Gamma_{A_p,\rho}^{(\tilde{M}_p),\infty}(\RR^{2d})$ (resp. $a\in \Gamma_{A_p,\rho}^{\{\tilde{M}_p\},\infty}(\RR^{2d})$) and, in fact, $f$ is $\Gamma_{A_p,\rho}^{(\tilde{M}_p),\infty}$-admissible and $a$ is $f-\Gamma_{A_p,\rho}^{(\tilde{M}_p),\infty}$-elliptic (resp. $f$ is $\Gamma_{A_p,\rho}^{\{\tilde{M}_p\},\infty}$-admissible and $a$ is $f-\Gamma_{A_p,\rho}^{\{\tilde{M}_p\},\infty}$-elliptic). Similarly, the parametrix $q$ constructed in Remark \ref{ktv957939} is additionally $1/f-\Gamma_{A_p,\rho}^{(\tilde{M}_p),\infty}$-elliptic (resp. $1/f-\Gamma_{A_p,\rho}^{\{\tilde{M}_p\},\infty}$-elliptic). It is easy to see that $H^{(M_p)}_{A_p,\rho}(f)$ is topologically isomorphic to $H^{(\tilde{M}_p)}_{A_p,\rho}(f)$ (resp. $H^{\{M_p\}}_{A_p,\rho}(f)$ is topologically isomorphic to $H^{\{\tilde{M}_p\}}_{A_p,\rho}(f)$). In particular, $\SSS^{(\tilde{M}_p)}(\RR^d)\hookrightarrow H^{(M_p)}_{A_p,\rho}(f)\hookrightarrow\SSS'^{(\tilde{M}_p)}(\RR^d)$ (resp. $\SSS^{\{\tilde{M}_p\}}(\RR^d)\hookrightarrow H^{\{M_p\}}_{A_p,\rho}(f)\hookrightarrow\SSS'^{\{\tilde{M}_p\}}(\RR^d)$).\\
    \indent If $M_p\prec \tilde{M}_p$, then $\SSS^{\{M_p\}}(\RR^d)\hookrightarrow \SSS^{(\tilde{M}_p)}(\RR^d)$. Assume that $f$ is $\Gamma_{A_p,\rho}^{(M_p),\infty}$-admissible and $a\in\Gamma_{A_p,\rho}^{(M_p),\infty}(\RR^{2d})$ be $f-\Gamma_{A_p,\rho}^{(M_p),\infty}$-elliptic. If $f$ and $1/f$ are of ultrapolynomial growth of class $(\tilde{M}_p)$, then $a\in \Gamma_{A_p,\rho}^{\{M_p\},\infty}(\RR^{2d})\cap \Gamma_{A_p,\rho}^{(\tilde{M}_p),\infty}(\RR^{2d})$ and, in fact, $f$ is both $\Gamma_{A_p,\rho}^{\{M_p\},\infty}$-admissible and $\Gamma_{A_p,\rho}^{(\tilde{M}_p),\infty}$-admissible with $a$ being both $f-\Gamma_{A_p,\rho}^{\{M_p\},\infty}$-elliptic and $f-\Gamma_{A_p,\rho}^{(\tilde{M}_p),\infty}$-elliptic. Similarly, the parametrix $q$ constructed in Remark \ref{ktv957939} (in $\Gamma_{A_p,\rho}^{(M_p),\infty}$ sense) is both $1/f-\Gamma_{A_p,\rho}^{\{M_p\},\infty}$-elliptic and $1/f-\Gamma_{A_p,\rho}^{(\tilde{M}_p),\infty}$-elliptic. In this case, one can easily verify that $H^{\{M_p\}}_{A_p,\rho}(f)$, $H^{(\tilde{M}_p)}_{A_p,\rho}(f)$ and $H^{(M_p)}_{A_p,\rho}(f)$ are topologically isomorphic among each other. Hence, $\SSS^{(\tilde{M}_p)}(\RR^d)\hookrightarrow H^{\{M_p\}}_{A_p,\rho}(f)\hookrightarrow\SSS'^{(\tilde{M}_p)}(\RR^d)$.
\item[$(ix)$] Let $f$ be $\Gamma_{A_p,\rho}^{(M_p),\infty}$-admissible and $a\in\Gamma_{A_p,\rho}^{(M_p),\infty}(\RR^{2d})$ an $f-\Gamma_{A_p,\rho}^{(M_p),\infty}$-elliptic symbol. If $f$ and $1/f$ are of ultrapolynomial growth of class $\{M_p\}$, then $a\in \Gamma_{A_p,\rho}^{\{M_p\},\infty}(\RR^{2d})\cap \Gamma_{A_p,\rho}^{(M_p),\infty}(\RR^{2d})$ and, in fact, $f$ is $\Gamma_{A_p,\rho}^{\{M_p\},\infty}$-admissible with $a$ being $f-\Gamma_{A_p,\rho}^{\{M_p\},\infty}$-elliptic. Similarly, the parametrix $q$ constructed in Remark \ref{ktv957939} (in $\Gamma_{A_p,\rho}^{(M_p),\infty}$ sense) is additionally $1/f-\Gamma_{A_p,\rho}^{\{M_p\},\infty}$-elliptic. In this case, it is easy to show that $H^{(M_p)}_{A_p,\rho}(f)$ is topologically isomorphic to $H^{\{M_p\}}_{A_p,\rho}(f)$ and, in particular, $\SSS^{\{M_p\}}(\RR^d)\hookrightarrow H^{(M_p)}_{A_p,\rho}(f)\hookrightarrow\SSS'^{\{M_p\}}(\RR^d)$.
\item[$(x)$] (The Fredholm property) Let $f$ and $g$ be $\Gamma_{A_p,\rho}^{*,\infty}$-admissible and $a$ be $g$-elliptic. Then $a^w_{|H^*_{A_p,\rho}(f)}:H^*_{A_p,\rho}(f)\rightarrow H^*_{A_p,\rho}(f/g)$ is a Fredholm operator. Moreover
    \beqs
    \mathrm{ind}\, a^w_{|H^*_{A_p,\rho}(f)}=\mathrm{dim}\,\mathrm{Ker}\, a^w-\mathrm{dim}\,\mathrm{Ker}\, (a^w)^*=\mathrm{dim}\,\mathrm{Ker}\, a^w-\mathrm{dim}\,\mathrm{Ker}\, {}^t(a^w),
    \eeqs
    where $(a^w)^*$ is the formal adjoint (which is in fact equal to $\bar{a}^w$) and ${}^t(a^w)$ is the transpose (which is in fact equal to $\Op_{1/2}(a(x,-\xi))$). The proof is the same as that of \cite[Theorem 1.6.9, p. 51]{NR}; cf. Remark \ref{kts951307}. In particular, the index of $a^w$ does not depend on the particular choice of $f$. If $g_1$ is $\Gamma_{A_p,\rho}^{*,\infty}$-admissible such that $g_1/g$ tends to zero at infinity and $b\in\Gamma_{A_p,\rho}^{*,\infty}(\RR^{2d})$ satisfies the following estimate: for every $h>0$ there exists $C>0$ (resp. there exist $h,C>0$) such that
    \beqs
    |D^{\alpha}b(w)|\leq Ch^{|\alpha|}A_{\alpha}g_1(w)\langle w\rangle^{-\rho|\alpha|},\,\, w\in\RR^{2d},\, \alpha\in\NN^{2d},
    \eeqs
    then $(iv)$ proves that $b^w_{|H^*_{A_p,\rho}(f)}:H^*_{A_p,\rho}(f)\rightarrow H^*_{A_p,\rho}(f/g)$ is compact; hence,
    $$
    a^w_{|H^*_{A_p,\rho}(f)}+b^w_{|H^*_{A_p,\rho}(f)}:H^*_{A_p,\rho}(f)\rightarrow H^*_{A_p,\rho}(f/g)
    $$
    is Fredholm and $\mathrm{ind}\, (a^w_{|H^*_{A_p,\rho}(f)}+b^w_{|H^*_{A_p,\rho}(f)})=\mathrm{ind}\, a^w_{|H^*_{A_p,\rho}(f)}$.
\end{itemize}

We supplement $(vii)$ with the following result on ``ultra-regularity''.

\begin{proposition}\label{ktr991101}
Let $s>1/(1-\rho)$ be such that $M_p\subset p!^s$ in the Beurling case and $M_p\prec p!^s$ in the Roumieu case, respectively, and assume that $\tilde{M}_p$, $p\in\NN$, satisfies $(M.1)$, $(M.2)$, $(M.3)$, $\tilde{M}_0=1$, $p!^s\subset\tilde{M}_p$ and $\tilde{M}_p/p!^s$, $p\in\NN$, is monotonically increasing.
\begin{itemize}
\item[$(a)$] Let $g$ be $\Gamma_{A_p,\rho}^{*,\infty}$-admissible and suppose there are $l,c>0$ such that
    \beq\label{ktr771305}
    g(w)\geq ce^{\tilde{M}(l|w|)},\,\, \forall w\in\RR^{2d}.
    \eeq
    Then $H^*_{A_p,\rho}(g)$ is continuously injected into $\SSS^{\{\tilde{M}_p\}}(\RR^d)$.
\item[$(b)$] If $g$ is $\Gamma_{A_p,\rho}^{*,\infty}$-admissible and satisfies \eqref{ktr771305} for every $l>0$ and a suitable $c=c(l)>0$, then $H^*_{A_p,\rho}(g)$ is continuously injected into $\SSS^{(\tilde{M}_p)}(\RR^d)$.
\end{itemize}
\end{proposition}

\begin{proof} Let $a(w)=|w|^2$. Hence, $H=a^w$ is the harmonic oscillator. Let $1\leq \mu_0\leq \mu_1\leq \mu_2\leq \ldots$ be the eigenvalues of $H$ arranged in an ascending order with multiplicities taken into account and let $\psi_j$, $j\in\NN$, be the corresponding Hermite function. We treat $(a)$ and $(b)$ simultaneously. For $(b)$, let $l'>0$ be arbitrary but fixed and, for $(a)$, let $l'\geq 4/l^2$, where $l$ is the constant in (\ref{ktr771305}). Put $\widehat{M}_n=l'^n\tilde{M}_{2n}$, $n\in\NN$. One easily verifies that $\widehat{M}_n$ satisfies the conditions in Theorem \ref{serpolsym} with the given $s>1/(1-\rho)$ and $m=2$ ((\ref{stk997733}) holds true because $\tilde{M}_p/p!^s$ is monotonically increasing). Let $P:\RR\rightarrow \RR$ be given by $P(\lambda)=1+\sum_{n=1}^{\infty}\lambda^n/\widehat{M}_n$. Because of Theorem \ref{serpolsym} and Corollary \ref{harmosers}, the function $f(w)=P(|w|^2)$ is $\Gamma_{A_p,\rho}^{*,\infty}$-admissible with $b=\sum_{n=0}^{\infty}a^{(\# n)}/\widehat{M}_n\in \Gamma_{A_p,\rho}^{*,\infty}(\RR^d)$ being $f$-elliptic; furthermore, the series $b^w=\sum_{n=0}^{\infty}H^n/\widehat{M}_n$ absolutely converges in $\mathcal{L}_b(\SSS^*(\RR^d),\SSS^*(\RR^d))$ and in $\mathcal{L}_b(\SSS'^*(\RR^d),\SSS'^*(\RR^d))$. Corollary \ref{cor easy case eigenvalues} implies that the spectrum of the closure of $b^w$ in $L^2(\RR^d)$ is given by its eigenvalues $\{P(\mu_j)|\, j\in\NN\}$ with multiplicities taken into account. Because of $(vii)$, $H^*_{A_p,\rho}(g)$ is continuously injected into $\SSS(\RR^d)$. One easily verifies that there exists $C'>0$ such that $g(w)\geq C'f(w)$, $\forall w\in\RR^{2d}$, and thus $(iv)$ implies that $b^w$ maps $H^*_{A_p,\rho}(g)$ continuously into $L^2(\RR^d)$. Fix $u\in H^*_{A_p,\rho}(g)$ and denote $c_j=(u,\psi_j)$, $j\in\NN$. Then $u=\sum_{j=0}^{\infty} c_j\psi_j$ and the series absolutely converges in $\SSS(\RR^d)$ and hence in $\SSS'^*(\RR^d)$ as well. Moreover, $L^2(\RR^d)\ni b^wu=\sum_{j=0}^{\infty}c_jP(\mu_j)\psi_j$, and consequently $\sum_{j=0}^{\infty}|c_j|^2(P(\mu_j))^2<\infty$. By Theorem \ref{thExWeyl}, $P(\mu_{j})=P(2(jd!)^{1/d}(1+o(1)))$; thus, there is a constant $c'>0$ such that $P(\mu_j)\geq c'P(j^{1/{d}})$, $j\in\NN$. Hence, $\sum_{j=0}^{\infty}|c_j|^2(P(j^{1/d}))^2<\infty$. Let $\tilde{H}\geq 1$ be the constant from the condition $(M.2)$ on $\tilde{M}_p$. One easily verifies that there exists $c''>0$, which depends only on $\tilde{M}_p$, such that $P(\lambda^{2})\geq c''e^{2\tilde{M}(\lambda/(\tilde{H}\sqrt{l'}))}\geq c''e^{\tilde{M}(\lambda/(\tilde{H}\sqrt{l'}))}$, $\lambda\geq 0$. Hence,
\beq\label{ktn995109}
\sup_{j\in\NN} |c_j|\exp\left(\tilde{M}\left(j^{1/(2d)}/(\tilde{H}\sqrt{l'})\right)\right)<\infty.
\eeq
For $(a)$, this estimate together with \cite[Theorem 4.1]{djvin} (see also \cite{gpr2}) automatically implies $u\in\SSS^{\{\tilde{M}_p\}}(\RR^d)$. For $(b)$, $l'$ was arbitrary, hence this estimate together with the quoted result implies $u\in\SSS^{(\tilde{M}_p)}(\RR^d)$. As the canonical inclusion $H^*_{A_p,\rho}(g)\rightarrow L^2(\RR^d)$ is continuous, it has a closed graph in $H^*_{A_p,\rho}(g)\times L^2(\RR^d)$ and, by what we just proved, its graph is in fact closed in $H^*_{A_p,\rho}(g)\times \SSS^{\{\tilde{M}_p\}}(\RR^d)$ for $(a)$ and in $H^*_{A_p,\rho}(g)\times \SSS^{(\tilde{M}_p)}(\RR^d)$ for $(b)$. Now, the Pt\'{a}k closed graph theorem \cite[Theorem 8.5, p. 166]{Sch} proves the desired continuity ($\SSS^{(\tilde{M}_p)}(\RR^d)$ and $\SSS^{\{\tilde{M}_p\}}(\RR^d)$ are Pt\'{a}k spaces since every $(F)$ and every $(DFS)$-space is Pt\'{a}k space; see \cite[Section 8, p. 162]{Sch}).
\end{proof}

\begin{remark} Proposition \ref{ktr991101} is applicable when $\tilde{M}_p=p!^s$. In this case $(a)$ gives that $H^*_{A_p,\rho}(g)$ is continuously injected into $\SSS^{\{s\}}(\RR^d)$ and $(b)$ yields that $H^*_{A_p,\rho}(g)$ is continuously injected into $\SSS^{(s)}(\RR^d)$.
\end{remark}

We have the following consequence of Proposition \ref{ktr991101}.

\begin{corollary}\label{corollaryontheregofsob}
Let $s$ and $\tilde{M}_p$ be as in Proposition \ref{ktr991101}. Let $g$ be $\Gamma_{A_p,\rho}^{*,\infty}$-admissible. If $1/g$ satisfies the assumption in Proposition \ref{ktr991101} $(a)$, then $\SSS'^{\{\tilde{M}_p\}}(\RR^d)$ is continuously injected into $H^*_{A_p,\rho}(g)$. If $1/g$ satisfies the assumption in Proposition \ref{ktr991101} $(b)$, then $\SSS'^{(\tilde{M}_p)}(\RR^d)$ is continuously injected into $H^*_{A_p,\rho}(g)$.
\end{corollary}

\begin{proof} Assume first that $1/g$ satisfies the assumption of Proposition \ref{ktr991101} $(a)$ with some $c,l>0$. Put $\widehat{M}_n=2^nl^{-2n}\tilde{M}_{2n}$, $n\in\NN$. It is easy to verify that $\widehat{M}_n$ satisfies the conditions of Theorem \ref{serpolsym} with the given $s$ and $m=2$ ((\ref{stk997733}) follows from the fact that $\tilde{M}_p/p!^s$ is monotonically increasing). Theorem \ref{serpolsym} and Corollary \ref{harmosers} imply that
$$
f:\RR^{2d}\rightarrow [1,\infty),\,\, f(w)=1+\sum_{n=1}^{\infty}|w|^{2n}/\widehat{M}_n,
$$
is $\Gamma_{A_p,\rho}^{*,\infty}$-admissible with $b=\sum_{n=0}^{\infty}a^{(\# n)}/\widehat{M}_n\in \Gamma_{A_p,\rho}^{*,\infty}(\RR^d)$ being $f$-elliptic, where we take $a(w)=|w|^2$. If we denote by $\tilde{H}\geq 1$ the constant from the condition $(M.2)$ on $\tilde{M}_p$, one can readily check that there exist $c_1,c_2,c_3>0$ such that
\beqs
c_1/g(w)\geq c_2e^{\tilde{M}(l|w|)}\geq f(w)\geq c_3e^{2\tilde{M}(l|w|/(\tilde{H}\sqrt{2}))}\geq c_3e^{\tilde{M}(l|w|/(\tilde{H}\sqrt{2}))}.
\eeqs
Thus, $H^*_{A_p,\rho}(1/f)$ is continuously injected into $H^*_{A_p,\rho}(g)$ and $f$ satisfies the condition from Proposition \ref{ktr991101} $(a)$. Hence, $H^*_{A_p,\rho}(f)$ is continuously (and densely) injected into $\SSS^{\{\tilde{M}_p\}}(\RR^d)$ and, by duality (cf. $(vi)$), we derive that $\SSS'^{\{\tilde{M}_p\}}(\RR^d)$ is continuously injected into $H^*_{A_p,\rho}(1/f)$. This ends the proof of the first part.\\
\indent Assume now that $1/g$ satisfies the assumptions from Proposition \ref{ktr991101} $(b)$. For each $l>0$, we define a sequence of positive numbers $\widehat{M}^{(l)}_n=l^{-2n}\tilde{M}_n$, $n\in\NN$. Clearly $\widehat{M}^{(l)}_n$ satisfies the conditions of Theorem \ref{serpolsym} with the given $s$ and $m=2$. Hence, Theorem \ref{serpolsym} and Corollary \ref{harmosers} imply that the functions
$$
f_l:\RR^{2d}\rightarrow [1,\infty),\,\, f_l(w)=1+\sum_{n=1}^{\infty}|w|^{2n}/\widehat{M}^{(l)}_n,
$$
are $\Gamma_{A_p,\rho}^{(M_p),\infty}$-admissible, for each $l>0$. There exist $c',c''>0$ such that
\beq\label{ktj995109}
c'e^{\tilde{M}(l\sqrt{2}|w|)}\geq f_l(w)\geq c''e^{2\tilde{M}(l|w|/\tilde{H})}\geq c''e^{\tilde{M}(l|w|/\tilde{H})},\,\, \forall w\in\RR^{2d},\, \forall l>0.
\eeq
Since $M_p\prec \tilde{M}_p$ in the Roumieu case, (\ref{ktj995109}) and the last part of $(viii)$ imply that $ f_l$ are also $\Gamma_{A_p,\rho}^{\{M_p\},\infty}$-admissible (in fact, $H^{(M_p)}_{A_p,\rho}( f_l)$ and $H^{\{M_p\}}_{A_p,\rho}( f_l)$ are topologically isomorphic for each $l>0$ by $(viii)$). Furthermore, if $l\leq l'$ then $f_l(w)\leq  f_{l'}(w)$, $\forall w\in\RR^{2d}$. Thus, $H^*_{A_p,\rho}( f_{l'})$ is continuously (and densely) injected into $H^*_{A_p,\rho}( f_l)$. The proof of the second part relies on the following lemma.

\begin{lemma}\label{lll} For each $l>0$, $\SSS^{(\tilde{M}_p)}(\RR^d)\hookrightarrow H^*_{A_p,\rho}( f_l)$. Furthermore, $\SSS^{(\tilde{M}_p)}(\RR^d)$ is topologically isomorphic to $\ds \lim_{\substack{\longleftarrow\\ l\rightarrow \infty}}H^*_{A_p,\rho}( f_l)$, where the linking mappings in the projective limit are the canonical inclusions.
\end{lemma}

Let $u\in \SSS'^{(\tilde{M}_p)}(\RR^d)$. By Lemma \ref{lll} and the definition of the projective limit topology, there exists $l>0$ such that $u$ can be continuously extended to a functional on $H^*_{A_p,\rho}( f_l)$. Applying $(vi)$, we conclude $u\in H^*_{A_p,\rho}(1/ f_l)$. Because of (\ref{ktj995109}) and the condition on $1/g$, we derive $u\in H^*_{A_p,\rho}(g)$, i.e. $\SSS'^{(\tilde{M}_p)}(\RR^d)\subseteq H^*_{A_p,\rho}(g)$. Since both space are continuously injected into $\SSS'^*(\RR^d)$, the continuity of this inclusion follows from the Pt\'{a}k closed graph theorem \cite[Theorem 8.5, p. 166]{Sch} in the same way as in the proof of Proposition \ref{ktr991101}.\\
\\
\noindent {\it Proof of Lemma \ref{lll}}.
Denote $\ds \lim_{\substack{\longleftarrow\\ l\rightarrow \infty}}H^*_{A_p,\rho}( f_l)$ by $E$. As the projective limit is equivalent to a countable one, $E$ is an $(F)$-space. Since $M_p\subset p!^s$ in the Beurling case and $M_p\prec p!^s$ in the Roumieu case respectively, (\ref{ktj995109}) together with $(viii)$ yields that $\SSS^{(\tilde{M}_p)}(\RR^d)\hookrightarrow H^*_{A_p,\rho}( f_l)$, for each $l>0$. Hence, $\SSS^{(\tilde{M}_p)}(\RR^d)$ is continuously injected into $E$. Because of $(vii)$, $E$ is continuously injected into $\SSS(\RR^d)$. Let $u\in E$, i.e. $u\in H^*_{A_p,\rho}( f_l)$, $\forall l>0$. Let $u=\sum_{j=0}^{\infty}c_j\psi_j$, where $c_j=(u,\psi_j)$, $j\in\NN$, and $\psi_j$, $j\in\NN$, are the Hermite functions; the series absolutely converges in $\SSS(\RR^d)$. Because of (\ref{ktj995109}), we can apply the same technique as in the proof of Proposition \ref{ktr991101} to obtain
\beqs
\sup_{j\in\NN}|c_j|\exp\left(\tilde{M}(lj^{1/(2d)}/\tilde{H})\right)<\infty,\,\, \forall l>0.
\eeqs
Hence, \cite[Theorem 4.1]{djvin} proves that $u\in\SSS^{(\tilde{M}_p)}(\RR^d)$. Now, the open mapping theorem for $(F)$-spaces completes the proof of the claim.
\end{proof}

As a direct consequence of Lemma \ref{lll} and (\ref{ktj995109}) together with Remark \ref{admfctwit}, we have the following result.

\begin{corollary}
Let $s>1/(1-\rho)$ is such that $M_p\subset p!^s$ in the Beurling case and $M_p\prec p!^s$ in the Roumieu case respectively. Then $\SSS^{(s)}(\RR^d)\hookrightarrow H^*_{A_p,\rho}(e^{l|w|^{1/s}})$ and
$$
\SSS^{(s)}(\RR^d)=\lim_{\substack{\longleftarrow\\ l\rightarrow \infty}} H^*_{A_p,\rho}(e^{l|w|^{1/s}})\quad \mbox{topologically},
$$
where the linking mappings in the projective limit are the canonical inclusions.
\end{corollary}

\noindent \textbf{An application.}
We end this section with the following illustrative application of the Fredholm property (see $(x)$) and the regularity result given in Proposition \ref{ktr991101}. Let $s$ and $\tilde{M}_p$ be as in Proposition \ref{ktr991101}. For $l'>0$ fixed, set $\widehat{M}_n=l'^{-2n}\tilde{M}_{2n}$, $n\in\NN$, and we consider again $P:\RR\rightarrow \RR$ given by $P(\lambda)=1+\sum_{n=1}^{\infty}\lambda^n/\widehat{M}_n$, the symbol $a(w)=|w|^2$ of the harmonic oscillator $H$, and the hypoelliptic symbol $b=\sum_{n=0}^{\infty}a^{(\# n)}/\widehat{M}_n\in\Gamma_{A_p,\rho}^{*,\infty}(\RR^{2d})$ whose Weyl quantisation is $b^w=\sum_{n=0}^{\infty}H^n/\widehat{M}_n$  (cf. Theorem \ref{serpolsym} and Corollary \ref{harmosers}). Once more, we set $f(w)=P(|w|^2)$ so that $f$ is $\Gamma_{A_p,\rho}^{*,\infty}$-admissible and $b$ is $f$-elliptic. For each $\lambda\in\RR$, $b-\lambda$ is also $f$-elliptic, and hence, $\mathrm{Ker}\, (b^w-\lambda)\subseteq\SSS^*(\RR^d)$. By $(x)$, $b^w_{|H^*_{A_p,\rho}(\tilde{f})}-\lambda: H^*_{A_p,\rho}(\tilde{f})\rightarrow L^2(\RR^d)$ is Fredholm and, since $b$ is real-valued, $\mathrm{ind}\,(b^w_{|H^*_{A_p,\rho}(\tilde{f})}-\lambda) =0$. If $\lambda_{\alpha}=\sum_{j=1}^d (2\alpha_j+1)$, $\alpha\in\NN^d$, denote the eigenvalues of $H$, then, because of Corollary \ref{cor easy case eigenvalues}, for any $\lambda\neq P(\lambda_{\alpha})$, $\forall \alpha\in\NN^d$, we have $\mathrm{Ker}\, (b^w-\lambda)=\{0\}$ and hence $b^w_{|H^*_{A_p,\rho}(f)}-\lambda$ is topological isomorphism between $H^*_{A_p,\rho}(f)$ and $L^2(\RR^d)$.

As a consequence, for each $v\in L^2(\RR^d)$, the equation $(b^w-\lambda)u=v$ is uniquely solvable in $\SSS'^*(\RR^d)$ and the solution must belong to $H^*_{A_p,\rho}(f)$ (this is particulary applicable with $\lambda=0$ since $P(\lambda_{\alpha})>1$, $\forall \alpha\in\NN^d$). Using Proposition \ref{ktr991101}, we can give qualitative information concerning its smoothness. Since $f$ satisfies (\ref{ktr771305}) with some $c,l>0$, we conclude that the solution $u$ must belong to $\SSS^{\{\tilde{M}_p\}}(\RR^d)$. If $\lambda=P(\lambda_{\alpha})$ for some $\alpha\in\NN^d$, then (because of Corollary \ref{cor easy case eigenvalues}) the solutions of the equation $(b^w-P(\lambda_{\alpha}))u=0$ are exactly all the functions that belong to the eigenspace of $H$ that corresponds to $\lambda_{\alpha}$ which is a finite dimensional subspace of $\SSS^*(\RR^d)$.

\section{Integral formula for the index}
\label{integral formula index}
The goal of this section is to give an integral formula for the index of the $\Psi$DO $a^w$ with $f$-elliptic symbol $a\in \Gamma_{A_p,\rho}^{*,\infty}(\RR^{2d})$. In fact, we will consider the more general case when $\mathbf{a}$ has values in $\mathcal{L}_b(V)=\mathcal{L}_b(V,V)$ for a finite dimensional complex $(B)$-space $V$. In order to even say anything about its index, we first need to provide sufficient conditions when such operator valued $\Psi$DO is Fredholm operator between appropriate $V$-valued Shubin-Sobolev spaces defined in Section \ref{fredh}. As one might expect, one of the (essential) conditions is the boundedness from below of $|\det \mathbf{a}|$ by an appropriate $\Gamma_{A_p,\rho}^{*,\infty}$-admissible functions. This agrees with the finite order case; see \cite{hor-1}. The integral formula we will show (see Theorem \ref{thminint} below) is in the spirit of Fedosov \cite{fedosov1,fedosov2} and H\"ormander \cite{hor-1}.\\
\indent We start with the following simple observation whose proof is trivial and we omit it; here and throughout the rest of the article, given a l.c.s. $E$ we denote by $E^n$ the l.c.s. $\ds\underbrace{E\times\ldots\times E}_n$. We note that its strong dual $(E^n)'_b$ is canonically isomorphic to $E'^n_b$.

\begin{lemma}\label{lemmaformatrixop}
Let $E$ and $F$ be l.c.s. and $n\in\ZZ_+$. If for each $k,j\in\{1,\ldots,n\}$ we are given $T_{k,j}\in\mathcal{L}(E,F)$ then the linear operator $(T_{k,j})_{k,j}:E^n\rightarrow F^n$, $(T_{k,j})_{k,j}(e_1,\ldots,e_n)=(\sum_{j=1}^nT_{1,j}e_j,\ldots,\sum_{j=1}^nT_{n,j}e_j)$, is well defined and continuous. If additionally all $T_{k,j}$ are compact then so is $(T_{k,j})_{k,j}$.
\end{lemma}

Let $n\in\ZZ_+$ and let $V$ be an $n$-dimensional complex $(B)$-space with norm $\|\cdot\|_V$. For each $h>0$, we define the $(B)$-space $\SSS^{M_p,h}(\RR^d;V)$ of all $\boldsymbol{\varphi}\in C^{\infty}(\RR^d;V)$ for which the norm $\|\boldsymbol{\varphi}\|_{h,V}=\sup_{\alpha\in\NN^d}\sup_{x\in\RR^d} h^{|\alpha|}\|D^{\alpha}\boldsymbol{\varphi}(x)\|_Ve^{M(h|x|)}/M_{\alpha}$ is finite. One easily verifies that for $h_1,h_2>0$ satisfying $Hh_1< h_2$, the inclusion $\SSS^{M_p,h_2}(\RR^d;V)\rightarrow \SSS^{M_p,h_1}(\RR^d;V)$ is compact. As l.c.s. we define
\beqs
\SSS^{(M_p)}(\RR^d;V)=\lim_{\substack{\longleftarrow\\ h\rightarrow \infty}}\SSS^{M_p,h}(\RR^d;V),\,\,\, \SSS^{\{M_p\}}(\RR^d;V)=\lim_{\substack{\longrightarrow \\ h\rightarrow 0}}\SSS^{M_p,h}(\RR^d;V).
\eeqs
Thus, $\SSS^{(M_p)}(\RR^d;V)$ is an $(FS)$-space and $\SSS^{\{M_p\}}(\RR^d;V)$ is a $(DFS)$-space. Fixing a basis $\{e_k\}_k$ for $V$ induces the topological isomorphism $\boldsymbol{\Omega}_{\{e_k\}_k}:\SSS^*(\RR^d;V)\rightarrow (\SSS^*(\RR^d))^n$, $\boldsymbol{\varphi}\mapsto (e'_1\circ\boldsymbol{\varphi},\ldots,e'_n\circ\boldsymbol{\varphi})$, where $\{e'_k\}_k$ is the dual basis. Following Komatsu \cite{Komatsu3}, we define the l.c.s. of $V$-valued temperate ultradistributions of class $*$ $\SSS'^*(\RR^d;V)$ as $\mathcal{L}_b(\SSS^*(\RR^d),V)$. It is canonically isomorphic to $\SSS'^*(\RR^d)\otimes V$ since $V$ is finite dimensional and $\SSS'^*(\RR^d)$ is a complete Montel space; the topology on the tensor product is $\pi=\epsilon$. Furthermore, $\SSS^*(\RR^d;V)$ is canonically isomorphic to $\SSS^*(\RR^d)\otimes V$ (with topology $\pi=\epsilon$) which, in turn, is canonically isomorphic to $\mathcal{L}_b(\SSS'^*(\RR^d),V)$. Hence $\SSS^*(\RR^d;V)$ is canonically included into $\SSS'^*(\RR^d;V)$ with dense image. Moreover, $\boldsymbol{\Omega}_{\{e_k\}_k}$ naturally extends to the topological isomorphism $\boldsymbol{\Omega}_{\{e_k\}_k}:\SSS'^*(\RR^d;V)\rightarrow (\SSS'^*(\RR^d))^n$, $\mathbf{f}\mapsto (e'_1\circ\mathbf{f},\ldots,e'_n\circ\mathbf{f})$. Finally, $\SSS^*(\RR^d;V)$ and $\SSS'^*(\RR^d;V)$ are nuclear and the strong dual of $\SSS^*(\RR^d;V)$ is $\SSS'^*(\RR^d;V')$ and the strong dual of $\SSS'^*(\RR^d;V)$ is $\SSS^*(\RR^d;V')$.\\
\indent Let $f$ be $\Gamma_{A_p,\rho}^{*,\infty}$-admissible function. We define the $V$-valued ultradistributional Shubin-Sobolev space of class $*$ and order $f$ as follows. Let $a\in \Gamma_{A_p,\rho}^{*,\infty}(\RR^d)$ be $f$-elliptic with a parametrix $q^w$, $q\in\Gamma_{A_p,\rho}^{*,\infty}(\RR^d)$, and $T$ be the $*$-regularising operator $q^wa^w-\mathrm{Id}$. We define $H^*_{A_p,\rho}(f;V)$ and the norm on it by
\beqs
H^*_{A_p,\rho}(f;V)&=&\{\mathbf{u}\in\SSS'^*(\RR^d;V)|\, a^w(e'\circ\mathbf{u})\in L^2(\RR^d),\, \forall e'\in V'\}\\
\|\mathbf{u}\|_{H^*_{A_p,\rho}(f;V)}&=&\sup_{\|e'\|_{V'}\leq 1}\|a^w(e'\circ\mathbf{u})\|_{L^2(\RR^d)}+ \sup_{\|e'\|_{V'}\leq 1}\|T(e'\circ\mathbf{u})\|_{L^2(\RR^d)}.
\eeqs
Because of Section \ref{fredh} $(i)$, we infer that $H^*_{A_p,\rho}(f;V)$ and the topology induced on it by the norm $\|\cdot\|_{H^*_{A_p,\rho}(f;V)}$ do not depend on the particular choices of $a$, its parametrix $q$ and the resulting $*$-regularising operator $T$. Clearly, $\boldsymbol{\Omega}_{\{e_k\}_k}$ restricts to a topological isomorphism from $H^*_{A_p,\rho}(f;V)$ onto $(H^*_{A_p,\rho}(f))^n$ and $H^*_{A_p,\rho}(f;V)$ is topologically isomorphic to $H^*_{A_p,\rho}(f)\otimes V$ (the topology is $\pi=\epsilon$). In view of Section \ref{fredh} $(ii)$, fixing an inner product on $V$ and $\{e_k\}_k$ an orthonormal basis for $V$ with $\{e'_k\}_k$ being its dual basis introduces the inner product
$$
\sum_{k=1}^n\left(a^w(e'_k\circ\mathbf{u}),a^w(e'_k\circ\mathbf{v})\right)_{L^2(\RR^d)}+ \sum_{k=1}^n\left(T(e'_k\circ\mathbf{u}),T(e'_k\circ\mathbf{v})\right)_{L^2(\RR^d)}
$$
on $H^*_{A_p,\rho}(f;V)$ and, with it, it becomes a Hilbert space with induced norm equivalent to the norm we introduced above. We have the following continuous and dense inclusions: $\SSS^*(\RR^d;V)\hookrightarrow H^*_{A_p,\rho}(f;V)\hookrightarrow \SSS'^*(\RR^d;V)$. Of course, $H^*_{A_p,\rho}(1;V)$ is just $L^2(\RR^d;V)$.\\
\indent We denote by $\Gamma_{A_p,\rho}^{*,\infty}(\RR^{2d};\mathcal{L}_b(V))$ the following subspace of $C^{\infty}(\RR^{2d};\mathcal{L}_b(V))$: a symbol $\mathbf{a}\in C^{\infty}(\RR^{2d};\mathcal{L}_b(V))$ belongs to $\Gamma_{A_p,\rho}^{*,\infty}(\RR^{2d};\mathcal{L}_b(V))$ if there exists $m>0$ such that for every $h>0$ there exists $C>0$ (resp. there exists $h>0$ such that for every $m>0$ there exists $C>0$) such that
\beqs
\|D^{\alpha}\mathbf{a}(w)\|_{\mathcal{L}_b(V)}\leq Ch^{|\alpha|}A_{\alpha}e^{M(m|w|)}\langle w\rangle^{-\rho|\alpha|},\,\, w\in\RR^{2d},\, \alpha\in\NN^{2d}.
\eeqs
Clearly, when $V=\CC$, this is nothing else but $\Gamma_{A_p,\rho}^{*,\infty}(\RR^{2d})$. In a similar fashion as for $\Gamma_{A_p,\rho}^{*,\infty}(\RR^{2d})$, one can introduce a Hausdorff locally convex topology on $\Gamma_{A_p,\rho}^{*,\infty}(\RR^{2d};\mathcal{L}_b(V))$, but we will not need this fact. For each $e\in V$ and $v'\in V'$ the function $w\mapsto v'(\mathbf{a}(w)e)$ is an element of $\Gamma_{A_p,\rho}^{*,\infty}(\RR^{2d})$. Now, given $\tau\in\RR$, we define the $\tau$-quantisation of $\mathbf{a}$ as
\beq\label{defofoperatv}
\Op_{\tau}(\mathbf{a})\boldsymbol{\varphi}(x)=\frac{1}{(2\pi)^d} \int_{\RR^d}\int_{\RR^d}e^{i(x-y)\xi}\mathbf{a}((1-\tau)x+\tau y,\xi)\boldsymbol{\varphi}(y)dyd\xi,\,\, \boldsymbol{\varphi}\in\SSS^*(\RR^d;V),
\eeq
and the integral should be interpreted as an iterated integral. In an analogous fashion as for the scalar valued case, one can prove that $\Op_{\tau}(\mathbf{a})\boldsymbol{\varphi}$ is a $V$-valued continuous function with ultrapolynomial growth of class $*$ (see the proof of \cite[Theorem 1]{BojanP}) and thus an element of $\SSS'^*(\RR^d;V)$. Fix a basis $\{e_k\}_k$ for $V$ with $\{e'_k\}_k$ being its dual basis and denote by $a_{k,j}\in \Gamma_{A_p,\rho}^{*,\infty}(\RR^{2d})$ the function $w\mapsto e'_k(\mathbf{a}(w)e_j)$. Then $\boldsymbol{\Omega}_{\{e_k\}_k}\Op_{\tau}(\mathbf{a})\boldsymbol{\varphi}= (\Op_{\tau}(a_{k,j}))_{k,j}\boldsymbol{\Omega}_{\{e_k\}_k}\boldsymbol{\varphi}$ and $\boldsymbol{\varphi}\mapsto (\Op_{\tau}(a_{k,j}))_{k,j}\boldsymbol{\Omega}_{\{e_k\}_k}\boldsymbol{\varphi}$, $\SSS^*(\RR^d;V)\rightarrow (\SSS^*(\RR^d))^n$, is continuous (cf. Lemma \ref{lemmaformatrixop}). Thus, $\Op_{\tau}(\mathbf{a})$ is a well defined and continuous operator from $\SSS^*(\RR^d;V)$ into itself. Since $(\Op_{\tau}(a_{k,j}))_{k,j}$ uniquely extends to a continuous operator on $(\SSS'^*(\RR^d))^n$, it follows that $\Op_{\tau}(\mathbf{a})$ uniquely extends to a continuous operator on $\SSS'^*(\RR^d;V)$; notice that $\Op_{\tau}(\mathbf{a})$ does not depend on the choice of $\{e_k\}_k$ (i.e., it is coordinate free) since it is given as extension by continuity of $\Op_{\tau}(\mathbf{a}):\SSS^*(\RR^d;V)\rightarrow\SSS'^*(\RR^d;V)$. Given $\mathbf{a}\in \Gamma_{A_p,\rho}^{*,\infty}(\RR^{2d};\mathcal{L}_b(V))$, a basis $\{e_k\}_k$ for $V$ and a basis $\{v'_k\}_k$ for $V'$, let $a_{k,j}(w)=v'_k(\mathbf{a}(w)e_j)$. We will call the operator $(\Op_{\tau}a_{k,j})_{k,j}$ the matrix representation of $\Op_{\tau}(\mathbf{a})$ with respect to $\{e_k\}_k$ and $\{v'_k\}_k$. Denoting by $\{v_k\}_k$ the dual basis (of $V$) of $\{v'_k\}_k$ we always have
\beq\label{transitioofop}
\Op_{\tau}(\mathbf{a})=\boldsymbol{\Omega}^{-1}_{\{v_k\}_k} (\Op_{\tau}(a_{k,j}))_{k,j}\boldsymbol{\Omega}_{\{e_k\}_k},\,\,\, \mbox{as operators on}\,\, \SSS^*(\RR^d;V)\,\, \mbox{and}\,\, \SSS'^*(\RR^d;V).
\eeq
\indent Let now $f$ be positive continuous function on $\RR^{2d}$ with ultrapolynomial growth of class $*$. Given $\mathbf{a}\in\Gamma_{A_p,\rho}^{*,\infty}(\RR^{2d};\mathcal{L}_b(V))$, we write $\mathbf{a}\precsim f$ if the following estimate holds true: for every $h>0$ there exists $C>0$ (resp. there exist $h,C>0$) such that
\beqs
\|D^{\alpha}\mathbf{a}(w)\|_{\mathcal{L}_b(V)}\leq Ch^{|\alpha|}A_{\alpha}f(w)\langle w\rangle^{-\rho|\alpha|},\,\, w\in\RR^{2d},\, \alpha\in\NN^{2d}.
\eeqs
Notice that $\mathbf{a}\precsim f$ if and only if for some (equivalently, for any) base $\{e_k\}_k$ of $V$ and $\{v'_k\}_k$ of $V'$ we have $\{a_{k,j}|\, k,j=1,\ldots,n\}\precsim f$, where $(a_{k,j})_{k,j}$ is the matrix representation of $\mathbf{a}$ with respect to $\{e_k\}_k$ and $\{v'_k\}_k$. If $\mathbf{a}\in\Gamma_{A_p,\rho}^{*,\infty}(\RR^{2d};\mathcal{L}_b(V))$, then $\det\mathbf{a}\in \Gamma_{A_p,\rho}^{*,\infty}(\RR^{2d})$. If additionally $\mathbf{a}\precsim f$, then $\det\mathbf{a}\precsim f^n$. If $f$ and $f_1$ are $\Gamma_{A_p,\rho}^{*,\infty}$-admissible and $\mathbf{a}\precsim f$, then (\ref{transitioofop}) together with Lemma \ref{lemmaformatrixop} imply that $\mathbf{a}^w$ restricts to a continuous operator from $H^*_{A_p,\rho}(f_1;V)$ into $H^*_{A_p,\rho}(f_1/f;V)$.

\begin{definition}
Let $f$ be $\Gamma_{A_p,\rho}^{*,\infty}$-admissible. We say that $\mathbf{a}\in \Gamma_{A_p,\rho}^{*,\infty}(\RR^{2d};\mathcal{L}_b(V))$ is $f$-elliptic if $\mathbf{a}\precsim f$ and there exists $B,c>0$ such that $|\det\mathbf{a}(w)|\geq c (f(w))^n$, $\forall w\in Q^c_B$.
\end{definition}

\begin{remark}\label{remfortheequivcinv}
There exists $c'_0\geq 1$, which only depends on $n=\mathrm{dim}\, V$ and $\|\cdot\|_V$, such that for any invertible $A\in\mathcal{L}(V)$ we have: $1/\|A\|_{\mathcal{L}_b(V)}\leq \|A^{-1}\|_{\mathcal{L}_b(V)}\leq c'_0\|A\|^{n-1}_{\mathcal{L}_b(V)}/|\det A|$. An easy consequence of this is that if $\mathbf{a}\in \Gamma_{A_p,\rho}^{*,\infty}(\RR^{2d};\mathcal{L}_b(V))$ satisfies $\mathbf{a}\precsim f$ for some $\Gamma_{A_p,\rho}^{*,\infty}$-admissible function $f$, then the condition $|\det\mathbf{a}(w)|\geq c (f(w))^n$, $\forall w\in Q^c_B$, is equivalent to the following one: there exist $B,c>0$ such that $\mathbf{a}(w)$ is invertible for all $w\in Q^c_B$ and $\|(\mathbf{a}(w))^{-1}\|_{\mathcal{L}_b(V)}\leq c/f(w)$, $\forall w\in Q^c_B$.
\end{remark}

\begin{remark}
If $\mathbf{a}\in \Gamma_{A_p,\rho}^{*,\infty}(\RR^{2d};\mathcal{L}_b(V))$ is $f$-elliptic for some $\Gamma_{A_p,\rho}^{*,\infty}$-admissible function $f$, then $\det\mathbf{a}\in \Gamma_{A_p,\rho}^{*,\infty}(\RR^{2d})$ is $f^n$-elliptic (cf. Lemma \ref{lskvpc135}).
\end{remark}

\begin{lemma}\label{parametrixvvalue}
Let $f$ be $\Gamma_{A_p,\rho}^{*,\infty}$-admissible and let $\mathbf{a}\in \Gamma_{A_p,\rho}^{*,\infty}(\RR^{2d};\mathcal{L}_b(V))$ be $f$-elliptic. Define $\mathbf{q}_0(w)=\mathbf{a}(w)^{-1}$ on $Q^c_B$ and inductively, for $j\in\ZZ_+$,
\beqs
\mathbf{q}_j(x,\xi)=-\sum_{s=1}^j\sum_{|\alpha+\beta|=s} \frac{(-1)^{|\beta|}}{\alpha!\beta!2^s}\partial^{\alpha}_{\xi} D^{\beta}_x \mathbf{q}_{j-s}(x,\xi) \partial^{\beta}_{\xi} D^{\alpha}_x \mathbf{a}(x,\xi)\mathbf{q}_0(x,\xi),\,\, (x,\xi)\in Q^c_B.
\eeqs
Then, for every $h>0$ there exists $C>0$ (resp. there exist $h,C>0$) such that
\beq\label{estofpara11}
\left\|D^{\alpha}_w \mathbf{q}_j(w)\right\|_{\mathcal{L}_b(V)}\leq C\frac{h^{|\alpha|+2j}A_{|\alpha|+2j}}{\langle w\rangle^{\rho(|\alpha|+2j)}f(w)},\,\, w\in Q^c_B,\,\alpha\in\NN^{2d},\, j\in\NN.
\eeq
One can extend $\mathbf{q}_0$ to an element of $\Gamma^{*,\infty}_{A_p,\rho}(\RR^{2d};\mathcal{L}_b(V))$ by modifying it on $Q_{B'}\backslash Q_B$, for $B'>B$.
Moreover, there exists $R>B'$ such that $\mathbf{q}=R(\sum_j\mathbf{q}_j)\in \Gamma_{A_p,\rho}^{*,\infty}(\RR^{2d};\mathcal{L}_b(V))$ is $1/f$-elliptic and $\mathbf{q}^w\mathbf{a}^w-\mathrm{Id}\in \mathcal{L}(\SSS'^*(\RR^d;V),\SSS^*(\RR^d;V))$.
\end{lemma}

\begin{proof} Remark \ref{remfortheequivcinv} proves (\ref{estofpara11}) for $j=0$ and $\alpha=0$; it also implies that $c'f(w)\leq \|\mathbf{a}(w)\|_{\mathcal{L}_b(V)}$, $\forall w\in Q^c_B$, for some $c'>0$. Let $\alpha\neq 0$. Differentiating  $\mathbf{q}_0(w)\mathbf{a}(w)=\mathbf{I}$, valid on $Q^c_B$ (where $\mathbf{I}$ is the identity operator on $V$), we have
\beqs
D^{\alpha}\mathbf{q}_0(w)=-\sum_{\substack{\beta\leq \alpha\\ \beta\neq 0}}{\alpha\choose\beta} D^{\alpha-\beta}\mathbf{q}_0(w)D^{\beta}\mathbf{a}(w)\mathbf{q}_0(w),\,\, w\in Q^c_B.
\eeqs
By induction on $|\alpha|$, employing the same technique as in the proof of \cite[Lemma 3.3]{CPP} one can prove (\ref{estofpara11}) for $j=0$, $\forall\alpha\in\NN^{2d}$. Now, by induction on $j$ and applying the same technique as in the proof of \cite[Lemma 3.4]{CPP} one can verify the validity of (\ref{estofpara11}) for all $j\in\NN$, $\alpha\in\NN^{2d}$.\\
\indent We extend $\mathbf{q}_0$ to an element of $\Gamma^{*,\infty}_{A_p,\rho}(\RR^{2d};\mathcal{L}_b(V))$ by modifying it on $Q_{B'}\backslash Q_B$, for some fixed $B'>B$. Notice that for each $R>B'$, $\mathbf{q}_R(w)=R(\sum_j \mathbf{q}_j)(w)$ belongs to $C^{\infty}(\RR^{2d};\mathcal{L}_b(V))$. Fix a basis $\{e_k\}_k$ for $V$ with $\{e'_k\}_k$ being its dual basis. Set
$$
q^{(k,s)}_j(w)=e'_k(\mathbf{q}_j(w)e_s),\, j\in\NN,\quad a^{(k,s)}(w)=e'_k(\mathbf{a}(w)e_s),\quad q^{(k,s)}_R(w)=e'_k(\mathbf{q}_R(w)e_s);
$$
clearly, $q^{(k,s)}_R=R(\sum_j q^{(k,s)}_j)$. The estimate (\ref{estofpara11}) implies that $\sum_j q^{(k,s)}_j\in FS_{A_p,\rho}^{*,\infty}(\RR^{2d};B)$ and $\{\sum_j q^{(k,s)}_j|\, k,s=1,\ldots,n\}\precsim 1/f$. Since $\{a^{(k,s)}|\, k,s=1,\ldots,n\}\precsim f$, Theorem \ref{weylq} gives the existence of $R>B'$ such that
\beq\label{vrstnp135}
\Op_{1/2}\left(R\left(\ssum q^{(k,s)}_j\# a^{(l,r)}\right)\right)=\Op_{1/2}\left(R\left(\ssum q^{(k,s)}_j\right)\right)\Op_{1/2}(a^{(l,r)})+T_{k,s,l,r}
\eeq
where $T_{k,s,l,r}$ is $*$-regularising; furthermore $q^{(k,s)}_R\in \Gamma^{*,\infty}_{A_p,\rho}(\RR^{2d})$ and $$\{q^{(k,s)}_R|\, k,s=1,\ldots,n\}\precsim_{1/f}\{\ssum q^{(k,s)}_j|\, k,s=1,\ldots,n\}$$ (cf. Proposition \ref{subexistest}) and consequently $q^{(k,s)}_R\precsim 1/f$. Incidently, this verifies that $\mathbf{q}_R\in \Gamma^{*,\infty}_{A_p,\rho}(\RR^{2d};\mathcal{L}_b(V))$ and $\mathbf{q}_R\precsim 1/f$. Since $|\det\mathbf{a}(w)|\leq c(f(w))^n$ we infer $|\det\mathbf{q}_0(w)|\geq c'/(f(w))^n$, $\forall w\in Q^c_B$.
Hence, $|\det\mathbf{q}_R(w)|\geq c''/(f(w))^n$, $\forall w\in Q^c_{B''}$, for some large enough $B''>B$.\footnote{There exists $\varepsilon=\varepsilon(n)>0$ such that for all $n\times n$ matrices $A$ whose all entries have modulus less then $\varepsilon$ it holds that $|\det (\mathbf{I}+A)|\geq 1/2$.} Consequently, $\mathbf{q}_R$ is $1/f$-elliptic. Notice that $\sum_{r=1}^n q^{(k,r)}_0a^{(r,s)}-\delta_{k,s}$ belongs to $\DD^{(A_p)}(\RR^{2d})$ in the $(M_p)$ case and in $\DD^{\{A_p\}}(\RR^{2d})$ in the $\{M_p\}$; here $\delta_{k,s}$ stands for the Kronecker delta. Furthermore, a direct inspection proves that for each $l\in\ZZ_+$ we have $\sum_{r=1}^n\left(\sum_j q^{(k,r)}_j\# a^{(r,s)}\right)_l=0$, for all $k,s=1,\ldots,n$. Thus, (\ref{vrstnp135}) together with (\ref{transitioofop}) proves that $\mathbf{q}_R^w\mathbf{a}^w-\mathrm{Id}$ belongs to $\mathcal{L}(\SSS'^*(\RR^d;V),\SSS^*(\RR^d;V))$.
\end{proof}

\begin{remark}
A direct consequence of this lemma is that if $\mathbf{a}$ is $f$-elliptic then $\mathbf{a}^w$ is globally regular, i.e. if $\mathbf{u}\in\SSS'^*(\RR^d;V)$ satisfies $\mathbf{a}^w\mathbf{u}\in\SSS^*(\RR^d;V)$ then $\mathbf{u}\in\SSS^*(\RR^d;V)$.
\end{remark}

\begin{remark}\label{rightparex}
If $f$ and $\mathbf{a}$ are as in Lemma \ref{parametrixvvalue} we can construct a right parametrix as follows. Define $\tilde{\mathbf{q}}_0(w)=\mathbf{a}(w)^{-1}$ on $Q^c_B$ and inductively, for $j\in\ZZ_+$,
\beqs
\tilde{\mathbf{q}}_j(x,\xi)=-\sum_{s=1}^j\sum_{|\alpha+\beta|=s} \frac{(-1)^{|\beta|}}{\alpha!\beta!2^s}\tilde{\mathbf{q}}_0(x,\xi) \partial^{\alpha}_{\xi} D^{\beta}_x \mathbf{a}(x,\xi) \partial^{\beta}_{\xi} D^{\alpha}_x \tilde{\mathbf{q}}_{j-s}(x,\xi),\,\, (x,\xi)\in Q^c_B,
\eeqs
In a completely analogous way as above, one can prove that $\tilde{\mathbf{q}}_j$, $j\in\NN$, satisfy (\ref{estofpara11}). Extending $\tilde{\mathbf{q}}_0$ to an element of $\Gamma^{*,\infty}_{A_p,\rho}(\RR^{2d};\mathcal{L}_b(V))$ (by modifying it on $Q_{B'}\backslash Q_B$, for $B'>B$), one can derive the existence of $R>B'$ such that $\tilde{\mathbf{q}}=R(\sum_j\tilde{\mathbf{q}}_j)\in \Gamma_{A_p,\rho}^{*,\infty}(\RR^{2d};\mathcal{L}_b(V))$ is $1/f$-elliptic and $\mathbf{a}^w\tilde{\mathbf{q}}^w-\mathrm{Id}\in \mathcal{L}(\SSS'^*(\RR^d;V),\SSS^*(\RR^d;V))$.
\end{remark}

\begin{corollary}\label{lemmaforinvindeforsp}
Let $f$ be $\Gamma_{A_p,\rho}^{*,\infty}$-admissible and $\mathbf{a}\in \Gamma_{A_p,\rho}^{*,\infty}(\RR^{2d};\mathcal{L}_b(V))$ be $f$-elliptic. Then, for any $\Gamma_{A_p,\rho}^{*,\infty}$-admissible function $f_1$, the operator $\mathbf{a}^w:H^*_{A_p,\rho}(f_1;V)\rightarrow H^*_{A_p,\rho}(f_1/f;V)$ is Fredholm and its index does not depend on $f_1$.
\end{corollary}

\begin{proof} The fact that $\mathbf{a}^w$ restricts to a Fredholm operator from $H^*_{A_p,\rho}(f_1;V)$ to $H^*_{A_p,\rho}(f_1/f;V)$ for any $\Gamma_{A_p,\rho}^{*,\infty}$-admissible $f_1$ is a direct consequence of Lemma \ref{parametrixvvalue} and Remark \ref{rightparex}. For each $\Gamma_{A_p,\rho}^{*,\infty}$-admissible function $f_1$, denote by $\mathbf{A}_{f_1}$ the restriction of $\mathbf{a}^w$ to $H^*_{A_p,\rho}(f_1;V)$ with codomain $H^*_{A_p,\rho}(f_1/f;V)$. Then $\mathrm{ker}\,\mathbf{A}_{f_1}=\mathrm{ker}\,\mathbf{a}^w\subseteq\SSS^*(\RR^d;V)$. Clearly, $\mathrm{coker}\, \mathbf{A}_{f_1}$ is isomorphic to the kernel of the transposed ${}^t\mathbf{A}_{f_1}:(H^*_{A_p,\rho}(f_1/f;V))'\rightarrow (H^*_{A_p,\rho}(f_1;V))'$. Notice that ${}^t\mathbf{A}_{f_1}$ is the restriction of ${}^t(\mathbf{a}^w):\SSS'^*(\RR^d;V')\rightarrow \SSS'^*(\RR^d;V')$ to $(H^*_{A_p,\rho}(f_1/f;V))'$. For each $(x,\xi)\in\RR^{2d}$ define $\tilde{\mathbf{a}}(x,\xi)\in \mathcal{L}(V')$ as the transposed of $\mathbf{a}(x,-\xi):V\rightarrow V$. One easily verifies that $w\mapsto \tilde{\mathbf{a}}(w)$ belongs to $\Gamma^{*,\infty}_{A_p,\rho}(\RR^{2d};\mathcal{L}_b(V'))$ and is $\tilde{f}$-elliptic, where $\tilde{f}(x,\xi)=f(x,-\xi)$. By direct inspection one verifies that ${}^t(\mathbf{a}^w)(\boldsymbol{\varphi})=\tilde{\mathbf{a}}^w\boldsymbol{\varphi}$, $\forall \boldsymbol{\varphi}\in\SSS^*(\RR^d;V')$ (cf. (\ref{defofoperatv})). Hence ${}^t(\mathbf{a}^w)$ is the $\Psi$DO $\tilde{\mathbf{a}}^w$ whose kernel is a subset of $\SSS^*(\RR^d;V')$ since it is $\tilde{f}$-elliptic. This completes the proof of the corollary.
\end{proof}

Because of this corollary, from now on we will be cavalier and not mention the domain and codomain when speaking about the index of $\mathbf{a}^w$ with $f$-elliptic symbol $\mathbf{a}$.\\
\indent Before we prove the main theorem of this section, we need the following technical result.

\begin{lemma}
Let $g_0$ be a positive continuous function on $(0,\infty)$ and let $b_0$ be a smooth function on $(s,\infty)$ for some $s>0$ which satisfies the following estimate: for every $h>0$ there exists $C>0$ (resp. there exist $h,C>0$) such that
\beq\label{nkldcg135}
|\partial^k b_0(\lambda)|\leq Ch^kA_k g_0(\lambda)\langle \lambda\rangle^{-\rho k},\,\,\, \lambda>s,\,\, k\in\NN.
\eeq
Then the function $b:x\mapsto b_0(|x|)$ is smooth on $\{x\in\RR^d|\, |x|>s\}$ and it satisfies the following estimate: for every $h>0$ there exists $C>0$ (resp. there exist $h,C>0$) such that
\beq\label{vlkrtc157}
|\partial^{\alpha} b(x)|\leq Ch^{|\alpha|}A_{\alpha} g_0(|x|)\langle x\rangle^{-\rho |\alpha|},\,\,\, |x|>s,\,\, \alpha\in\NN^d.
\eeq
\end{lemma}

\begin{proof} Clearly, $b$ is smooth on $\{x\in\RR^d|\, |x|>s\}$. The estimate in the lemma is trivial when $d=1$. When $d\geq 2$, the proof of the estimate relies on the Fa\`a di Bruno formula applied to the composition of $x\mapsto |x|$ and the real and imaginary part of $b_0$. Using the formulation as in \cite[Corollary 2.10]{faadib}, for $0\neq\alpha\in\NN^d$ and $|x|>s$, we infer
\beqs
|\partial^{\alpha}b(x)|\leq \sum_{r=1}^{|\alpha|}|(\partial^rb_0)(|x|)|\sum_{p(\alpha,r)}\alpha!\prod_{j=1}^{|\alpha|} \frac{\left|\partial^{\alpha^{(j)}}|x|\right|^{k_j}}{k_j! \left(\alpha^{(j)}!\right)^{k_j}},
\eeqs
where the set $p(\alpha,r)$ is as explained in \cite[Corollary 2.10]{faadib}. To derive a precise estimate on the derivatives of $|x|$, we consider the function $z\mapsto \sqrt{z_1^2+\ldots+z_d^2}$. It is analytic on $U=\{z\in\CC^d|\, |\mathrm{Re}\,z|>|\mathrm{Im}\,z|\}$ when we take the principal branch of the square root since $\mathrm{Re}\, (z_1^2+\ldots+z_d^2)>0$. For $0\neq x\in\RR^d$, we apply the Cauchy integral formula to this analytic function on the distinguished boundary of the polydisc $\{z\in\CC^d|\, |z_j-x_j|\leq |x|/(3\sqrt{d}),\, j=1,\ldots,d\}\subseteq U$ and a straightforward calculation yields
\beqs
|\partial^{\alpha}|x||\leq (3d)^{|\alpha|+1}\alpha! |x|^{1-|\alpha|},\,\,\, x\in\RR^d\backslash\{0\},\,\, \alpha\in\NN^d.
\eeqs
In the $(M_p)$ case, let $h>0$ be arbitrary but fixed and let $C_1>0$ be the constant for which (\ref{nkldcg135}) holds true with $h_1=h/(9\cdot 2^{d+1}d^2(1+s^{-1}))$. In the $\{M_p\}$ case, let $C_1,h_1>0$ be the constants for which (\ref{nkldcg135}) holds true. Since $\langle x\rangle\leq (1+s^{-1})|x|$ when $|x|>s$, we can estimate as follows
\beqs
|\partial^{\alpha}b(x)|&\leq& C_1\sum_{r=1}^{|\alpha|}h^r_1A_rg_0(|x|)\langle x\rangle^{-\rho r} \sum_{p(\alpha,r)}\alpha!\prod_{j=1}^{|\alpha|} \frac{(3d)^{|\alpha^{(j)}|k_j+k_j}|x|^{k_j-|\alpha^{(j)}|k_j}}{k_j!}\\
&\leq& C_1(3d(1+s^{-1}))^{|\alpha|}g_0(|x|)\langle x\rangle^{-\rho |\alpha|}\sum_{r=1}^{|\alpha|}(3dh_1)^rA_r \sum_{p(\alpha,r)}|\alpha|!\prod_{j=1}^{|\alpha|} \frac{1}{k_j!}\\
&=&\frac{C_1(9d^2(1+s^{-1})h_1)^{|\alpha|}g_0(|x|)}{\langle x\rangle^{\rho |\alpha|}} \sum_{r=1}^{|\alpha|}\frac{(|\alpha|-r)!A_r}{(3dh_1)^{|\alpha|-r}} {|\alpha|\choose r}\sum_{p(\alpha,r)} \frac{r!}{k_1!\cdot\ldots\cdot k_{|\alpha|}!}.
\eeqs
Because of  $(M.3)'$ for $A_p,$ there exists $C'>0$ such that $(3dh_1)^{-k}k!\leq C'A_k$, $\forall k\in\NN$. Thus, $(3dh_1)^{r-|\alpha|}(|\alpha|-r)!A_r\leq C'A_{\alpha}$. Now, \cite[Lemma 7.4]{PP1} gives
\beqs
\sum_{r=1}^{|\alpha|}{|\alpha|\choose r}\sum_{p(\alpha,r)} \frac{r!}{k_1!\cdot\ldots\cdot k_{|\alpha|}!}\leq 2^{|\alpha|(d+1)},
\eeqs
which yields the desired estimate (recall, $h_1=h/(9\cdot 2^{d+1}d^2(1+s^{-1}))$ in the $(M_p)$ case).
\end{proof}

Now we can state and prove the main result of this section.

\begin{theorem}\label{thminint}
Let $f$ be $\Gamma_{A_p,\rho}^{*,\infty}$-admissible and $\mathbf{a}\in\Gamma_{A_p,\rho}^{*,\infty}(\RR^{2d};\mathcal{L}_b(V))$ be $f$-elliptic. Assume that there exist a positive $f_0\in C(\RR)$ and $c,C>0$ such that
\beq\label{bfortheadmfun}
cf_0(|w|)\leq f(w)\leq Cf_0(|w|),\,\,\, \forall w\in\RR^{2d}.
\eeq
Then, for every $\Gamma_{A_p,\rho}^{*,\infty}$-admissible function $f_1$, $\mathbf{a}^w$ restricts to a Fredholm operator from $H^*_{A_p,\rho}(f_1;V)$ into $H^*_{A_p,\rho}(f_1/f;V)$ and its index does not depend on $f_1$. Furthermore,
\beq\label{intforml}
\mathrm{ind}\, \mathbf{a}^w=-\frac{(d-1)!}{(2d-1)!(2\pi i)^d}\int_{\partial \mathbb{B}^{2d}_s}\mathrm{tr}\,(\mathbf{a}^{-1}d\mathbf{a})^{2d-1},
\eeq
where $\partial \mathbb{B}^{2d}_s$ is the boundary of $\mathbb{B}^{2d}_s=\{w\in\RR^{2d}|\, |w|\leq s\}$ and $s\geq 2B$ is arbitrary. The orientation on $\partial \mathbb{B}^{2d}_s$ is the one induced by $\RR^{2d}$, where the latter is oriented by the nonvanishing $2d$-form $d\xi_1\wedge dx^1\wedge\ldots\wedge d\xi_d\wedge dx^d$.
\end{theorem}

\begin{remark}
The assumption (\ref{bfortheadmfun}) says that $f$ is ``almost radial at infinity''.\\
\indent Of course, the orientation on $\RR^{2d}$ used in the integral formula is the canonical one when $\RR^{2d}$ is viewed as the cotangent bundle of $\RR^d$ and equipped with the canonical symplectic form $d(\sum_{j=1}^n \xi_j dx^j)=\sum_{j=1}^d d\xi_j\wedge dx^j$. This is the orientation used by H\"ormander \cite{hor-1} and Fedosov \cite{fedosov2}. On the contrary, in Fedosov \cite{fedosov1} $\RR^{2d}$ is supplied with the orientation induced by the standard orientation on $\CC^d$ by the identification of $(x,\xi)$ with $x+i\xi$, i.e. the orientation induced by
\beq\label{stfofcom}
dx^1\wedge d\xi^1\wedge\ldots\wedge dx^d\wedge d\xi^d;
\eeq
consequently, there is an extra factor of $(-1)^d$ in the integral formula. As pointed out by H\"ormander \cite[p. 422]{hor-1} and Fedosov \cite[p. 320]{fedosov3}, the above integral formula agrees with the one by Atiyah-Singer \cite[Theorem 2.12]{as3}; the orientation used there corresponds to using the orientation induced by (\ref{stfofcom}) in our case.
\end{remark}

\begin{proof} The facts that $\mathbf{a}^w:H^*_{A_p,\rho}(f_1;V)\rightarrow H^*_{A_p,\rho}(f_1/f;V)$ is Fredholm and that its index does not depend on $f_1$ follow from Corollary \ref{lemmaforinvindeforsp}. It remains to prove the integral formula. Clearly, it is enough to prove it for $(a^w_{k,j})_{k,j}$, where $(a^w_{k,j})_{k,j}$ is the matrix representation of $\mathbf{a}^w$ with respect to a fixed basis $\{e_k\}_k$ of $V$ and its dual basis $\{e'_k\}_k$ of $V'$ (cf. (\ref{transitioofop})). To make the notation less cumbersome, we still denote the matrix of symbols $(a_{k,j})_{k,j}$ by $\mathbf{a}$ and the operator matrix $(a^w_{k,j})_{k,j}$ by $\mathbf{a}^w$. Furthermore, for any matrix of symbols $\mathbf{b}=(b_{k,j})_{k,j}$ and any $\Gamma_{A_p,\rho}^{*,\infty}$-admissible function $g$, the notation $\mathbf{b}\precsim g$ will just mean $\{b_{k,j}|\, k,j=1,\ldots,n\}\precsim g$. Thus, in this notation, we still have $\mathbf{a}\precsim f$.\\
\indent Define $\tilde{f}:\RR^{2d}\rightarrow (0,\infty)$ by $\tilde{f}(w)=f_0(|w|)$. Clearly $\tilde{f}$ is positive and continuous and because of (\ref{bfortheadmfun}), both $\tilde{f}$ and $1/\tilde{f}$ are of ultrapolynomial growth of class $*$. Of course, $\det\mathbf{a}$ is $\tilde{f}^n$-elliptic and thus $\tilde{f}$ is $\Gamma_{A_p,\rho}^{*,\infty}$-admissible (cf. Lemma \ref{lskvpc135}).\\
\indent By the same technique as in the proof of Lemma \ref{lskvpc135}, one can find a $1/\tilde{f}$-elliptic positive symbol $\tilde{b}\in\Gamma_{A_p,\rho}^{*,\infty}(\RR^{2d})$. Then $b_0:\lambda\mapsto \tilde{b}(\lambda,0,\ldots,0)$ satisfies (\ref{nkldcg135}) on $(0,\infty)$ with $1/f_0$ in place of $g_0$. Consequently, $b:w\mapsto b_0(|w|)$ satisfies (\ref{vlkrtc157}) when $|w|>s/4$ with $1/f_0$ in place of $g_0$, where $s>0$ is arbitrary but fixed. We modify $b$ on $\{w\in\RR^{2d}|\, |w|<s/2\}$ so that it becomes positive $1/\tilde{f}$-elliptic symbol in $\Gamma_{A_p,\rho}^{*,\infty}(\RR^{2d})$; equivalently, $b$ is $1/f$-elliptic. The constructed $b$ has the additional property of being a radial function when $|w|\geq s/2$. This will become important later. Because of Corollary \ref{lemmaforinvindeforsp}, $\mathbf{a}^w$ restricts to a Fredholm operator $(L^2(\RR^d))^n\rightarrow (H^*_{A_p,\rho}(1/f))^n$ and $\mathbf{b}^w=(b\mathbf{I})^w$ restricts to a Fredholm operator $(H^*_{A_p,\rho}(1/f))^n\rightarrow (L^2(\RR^d))^n$, where $\mathbf{I}$ stands for the $n\times n$ identity matrix. The $(k,j)$-entry in the operator matrix of $\mathbf{b}^w\mathbf{a}^w$ is $b^wa^w_{k,j}$ which, in view of Theorem \ref{weylq}, is equal to $(R(b\# a_{k,j}))^w+T_{k,j}$ for some large enough $R>0$ and a $*$-regularising operator $T_{k,j}$. Since $R(b\# a_{k,j})=ba_{k,j}+a'_{k,j}$ with $a'_{k,j}\in \Gamma_{A_p,\rho}^{*,\infty}(\RR^{2d})$ satisfying $a'_{k,j}\precsim \langle \cdot\rangle^{-2\rho}$, we infer $\mathbf{b}^w\mathbf{a}^w=(b\mathbf{a})^w+\mathbf{a}'^w+\mathbf{T}$ where $\mathbf{T}=(T_{k,j})_{k,j}$ and $\mathbf{a}'=(a'_{k,j})_{k,j}$. As $\mathbf{a}'\precsim \langle \cdot\rangle^{-2\rho}$, $\mathbf{a}'^w$ restricts to a compact operator $(L^2(\RR^d))^n\rightarrow (L^2(\RR^d))^n$ (cf. Lemma \ref{lemmaformatrixop}). Clearly, $b\mathbf{a}\precsim 1$ and $\det (b\mathbf{a})$ is $1$-elliptic. Hence, Corollary \ref{lemmaforinvindeforsp} yields that it is a Fredholm operator $(L^2(\RR^d))^n\rightarrow (L^2(\RR^d))^n$. Since $\mathbf{b}^w$ coincides with its formal adjoint and ${}^t(\mathbf{b}^w)\boldsymbol{\varphi}= \overline{\mathbf{b}^w\overline{\boldsymbol{\varphi}}}$, $\boldsymbol{\varphi}\in (\SSS^*(\RR^d))^n$, it follows $\mathrm{ind}\, \mathbf{b}^w=0$ (cf. the proof of Corollary \ref{lemmaforinvindeforsp}). The properties of the index give $\mathrm{ind}\, \mathbf{a}^w=\mathrm{ind}\, \mathbf{b}^w\mathbf{a}^w=\mathrm{ind}\, (b\mathbf{a})^w$. This reduces the problem to finite order operators and we can use the Fedosov-H\"ormander integral formula \cite[Theorem 7.3]{hor-1} (see also Fedosov \cite{fedosov1,fedosov2}) to conclude
\beqs
\mathrm{ind}\, \mathbf{a}^w=\mathrm{ind}\, (b\mathbf{a})^w=-\frac{(d-1)!}{(2d-1)!(2\pi i)^d}\int_{\partial \mathbb{B}^{2d}_s}\mathrm{tr}\,((b\mathbf{a})^{-1}d(b\mathbf{a}))^{2d-1},
\eeqs
where $s\geq 2B$ is arbitrary and $\mathbb{B}^{2d}_s=\{w\in\RR^{2d}|\ |w|\leq s\}$; the orientation on $\partial \mathbb{B}^{2d}_s$ is the one described in the statement of the theorem. (The assumptions of \cite[Theorem 7.3]{hor-1} are fulfilled since $b\mathbf{a}$ belongs to the H\"ormander class $S(1,g)$ with $g$ being the $\sigma$-temperate metric $g_w=\langle w\rangle^{-2\rho}|dw|^2$ with symplectic dual $g^{\sigma}_w=\langle w\rangle^{2\rho}|dw|^2$; of course, $(b\mathbf{a})^{-1}$ exists and is bounded when $|w|\geq s$.) As $b$ is radial when $|w|\geq s/2$, the pullback of $db$ to $\partial \mathbb{B}^{2d}_s$ vanishes. Hence $(b\textbf{a})^{-1}d(b\textbf{a})=\textbf{a}^{-1}d\textbf{a}$ as forms on $\partial \mathbb{B}^{2d}_s$ and the validity of (\ref{intforml}) easily follows.
\end{proof}

\begin{remark}
If $f$ is $\Gamma_{A_p,\rho}^{*,\infty}$-admissible functions which satisfies (\ref{bfortheadmfun}), $a\in \Gamma_{A_p,\rho}^{*,\infty}(\RR^{2d})$ is $f$-elliptic symbol and $d\geq 2$ then (\ref{intforml}) implies $\mathrm{ind}\, a^w=0$. When $d=n=1$, as pointed out by H\"ormander \cite[p. 422]{hor-1}, the index of $a^w$ is the winding number of $a$ considered as a map from $\partial \mathbb{B}^2_s$ into $\CC\backslash\{0\}$.
\end{remark}

\end{document}